\documentclass[oneside,english]{amsart}
\usepackage[T1]{fontenc}
\usepackage[latin9]{inputenc}
\usepackage{babel}
\usepackage{amsbsy}
\usepackage{amstext}
\usepackage{amsthm}
\usepackage{amssymb}
\usepackage{stmaryrd}
\usepackage[all]{xy}
\usepackage[unicode=true,pdfusetitle,
 bookmarks=true,bookmarksnumbered=false,bookmarksopen=false,
 breaklinks=false,pdfborder={0 0 1},backref=false,colorlinks=false]
 {hyperref}

\makeatletter

\xymatrixrowsep{1pc} 
\xymatrixcolsep{1pc}

\numberwithin{equation}{section}
\numberwithin{figure}{section}
\theoremstyle{plain}
\newtheorem{thm}{\protect\theoremname}[section]
\theoremstyle{plain}
\newtheorem{cor}[thm]{\protect\corollaryname}
\theoremstyle{definition}
\newtheorem{defn}[thm]{\protect\definitionname}
\theoremstyle{plain}
\newtheorem{lem}[thm]{\protect\lemmaname}
\theoremstyle{remark}
\newtheorem{notation}[thm]{\protect\notationname}
\theoremstyle{plain}
\newtheorem{prop}[thm]{\protect\propositionname}
\theoremstyle{remark}
\newtheorem{rem}[thm]{\protect\remarkname}
\theoremstyle{plain}
\newtheorem{assumption}[thm]{\protect\assumptionname}
\theoremstyle{definition}
\newtheorem{example}[thm]{\protect\examplename}

\makeatother

\providecommand{\assumptionname}{Assumption}
\providecommand{\corollaryname}{Corollary}
\providecommand{\definitionname}{Definition}
\providecommand{\examplename}{Example}
\providecommand{\lemmaname}{Lemma}
\providecommand{\notationname}{Notation}
\providecommand{\propositionname}{Proposition}
\providecommand{\remarkname}{Remark}
\providecommand{\theoremname}{Theorem}

\begin{document}
\title{Motivic integration over wild Deligne-Mumford stacks}
\author{Takehiko Yasuda}
\begin{abstract}
We develop the motivic integration theory over formal Deligne-Mumford
stacks over a power series ring of arbitrary characteristic. This
is a generalization of the corresponding theory for tame and smooth
Deligne-Mumford stacks constructed in earlier papers of the author.
As an application, we obtain the wild motivic McKay correspondence
for linear actions of arbitrary finite groups, which has been known
only for cyclic groups of prime order. In particular, this implies
the motivic version of Bhargava's mass formula as a special case.
In fact, we prove a more general result, the invariance of stringy
motives of (stacky) log pairs under crepant morphisms.
\end{abstract}

\address{Department of Mathematics, Graduate School of Science, Osaka University,
Toyonaka, Osaka 560-0043, JAPAN}
\email{takehikoyasuda@math.sci.osaka-u.ac.jp}
\thanks{This work was supported by JSPS KAKENHI Grant Numbers 16H06337, JP18H01112
and JP18K18710.}

\maketitle
\global\long\def\bigmid{\mathrel{}\middle|\mathrel{}}%

\global\long\def\AA{\mathbb{A}}%

\global\long\def\CC{\mathbb{C}}%

\global\long\def\FF{\mathbb{F}}%

\global\long\def\GG{\mathbb{G}}%

\global\long\def\LL{\mathbb{L}}%

\global\long\def\MM{\mathbb{M}}%

\global\long\def\NN{\mathbb{N}}%

\global\long\def\PP{\mathbb{P}}%

\global\long\def\QQ{\mathbb{Q}}%

\global\long\def\RR{\mathbb{R}}%

\global\long\def\SS{\mathbb{S}}%

\global\long\def\ZZ{\mathbb{Z}}%

\global\long\def\ba{\mathbf{a}}%

\global\long\def\bb{\mathbf{b}}%

\global\long\def\bd{\mathbf{d}}%

\global\long\def\bf{\mathbf{f}}%

\global\long\def\bg{\mathbf{g}}%

\global\long\def\bh{\mathbf{h}}%

\global\long\def\bj{\mathbf{j}}%

\global\long\def\bm{\mathbf{m}}%

\global\long\def\bp{\mathbf{p}}%

\global\long\def\bq{\mathbf{q}}%

\global\long\def\br{\mathbf{r}}%

\global\long\def\bs{\mathbf{s}}%

\global\long\def\bt{\mathbf{t}}%

\global\long\def\bv{\mathbf{v}}%

\global\long\def\bw{\mathbf{w}}%

\global\long\def\bx{\boldsymbol{x}}%

\global\long\def\by{\boldsymbol{y}}%

\global\long\def\bz{\mathbf{z}}%

\global\long\def\bG{\mathbf{G}}%

\global\long\def\bM{\mathbf{M}}%

\global\long\def\bP{\mathbf{P}}%

\global\long\def\cA{\mathcal{A}}%

\global\long\def\cB{\mathcal{B}}%

\global\long\def\cC{\mathcal{C}}%

\global\long\def\cD{\mathcal{D}}%

\global\long\def\cE{\mathcal{E}}%

\global\long\def\cF{\mathcal{F}}%

\global\long\def\cG{\mathcal{G}}%

\global\long\def\cH{\mathcal{H}}%

\global\long\def\cI{\mathcal{I}}%

\global\long\def\cJ{\mathcal{J}}%

\global\long\def\cK{\mathcal{K}}%

\global\long\def\cL{\mathcal{L}}%

\global\long\def\cM{\mathcal{M}}%

\global\long\def\cN{\mathcal{N}}%

\global\long\def\cO{\mathcal{O}}%

\global\long\def\cP{\mathcal{P}}%

\global\long\def\cQ{\mathcal{Q}}%

\global\long\def\cR{\mathcal{R}}%

\global\long\def\cS{\mathcal{S}}%

\global\long\def\cT{\mathcal{T}}%

\global\long\def\cU{\mathcal{U}}%

\global\long\def\cV{\mathcal{V}}%

\global\long\def\cW{\mathcal{W}}%

\global\long\def\cX{\mathcal{X}}%

\global\long\def\cY{\mathcal{Y}}%

\global\long\def\cZ{\mathcal{Z}}%

\global\long\def\fa{\mathfrak{a}}%

\global\long\def\fb{\mathfrak{b}}%

\global\long\def\fc{\mathfrak{c}}%

\global\long\def\ff{\mathfrak{f}}%

\global\long\def\fj{\mathfrak{j}}%

\global\long\def\fm{\mathfrak{m}}%

\global\long\def\fp{\mathfrak{p}}%

\global\long\def\fs{\mathfrak{s}}%

\global\long\def\ft{\mathfrak{t}}%

\global\long\def\fx{\mathfrak{x}}%

\global\long\def\fv{\mathfrak{v}}%

\global\long\def\fD{\mathfrak{D}}%

\global\long\def\fM{\mathfrak{M}}%

\global\long\def\fO{\mathfrak{O}}%

\global\long\def\fS{\mathfrak{S}}%

\global\long\def\fV{\mathfrak{V}}%

\global\long\def\fX{\mathfrak{X}}%

\global\long\def\fY{\mathfrak{Y}}%

\global\long\def\ru{\mathrm{u}}%

\global\long\def\rv{\mathbf{\mathrm{v}}}%

\global\long\def\rw{\mathrm{w}}%

\global\long\def\rx{\mathrm{x}}%

\global\long\def\ry{\mathrm{y}}%

\global\long\def\rz{\mathrm{z}}%

\global\long\def\B{\operatorname{\mathrm{B}}}%

\global\long\def\C{\operatorname{\mathrm{C}}}%

\global\long\def\H{\operatorname{\mathrm{H}}}%

\global\long\def\I{\operatorname{\mathrm{I}}}%

\global\long\def\J{\operatorname{\mathrm{J}}}%

\global\long\def\M{\operatorname{\mathrm{M}}}%

\global\long\def\N{\operatorname{\mathrm{N}}}%

\global\long\def\R{\operatorname{\mathrm{R}}}%

\global\long\def\age{\operatorname{age}}%

\global\long\def\Aut{\operatorname{Aut}}%

\global\long\def\codim{\operatorname{codim}}%

\global\long\def\Conj{\operatorname{Conj}}%

\global\long\def\det{\operatorname{det}}%

\global\long\def\Emb{\operatorname{Emb}}%

\global\long\def\Gal{\operatorname{Gal}}%

\global\long\def\Hom{\operatorname{Hom}}%

\global\long\def\Image{\operatorname{\mathrm{Im}}}%

\global\long\def\Iso{\operatorname{Iso}}%

\global\long\def\Ker{\operatorname{Ker}}%

\global\long\def\lcm{\operatorname{\mathrm{lcm}}}%

\global\long\def\length{\operatorname{\mathrm{length}}}%

\global\long\def\ord{\operatorname{ord}}%

\global\long\def\Ram{\operatorname{\mathrm{Ram}}}%

\global\long\def\Sp{\operatorname{Sp}}%

\global\long\def\Spec{\operatorname{Spec}}%

\global\long\def\Spf{\operatorname{Spf}}%

\global\long\def\Stab{\operatorname{Stab}}%

\global\long\def\Supp{\operatorname{Supp}}%

\global\long\def\sbrats{\llbracket s\rrbracket}%

\global\long\def\spars{\llparenthesis s\rrparenthesis}%

\global\long\def\tbrats{\llbracket t\rrbracket}%

\global\long\def\tpars{\llparenthesis t\rrparenthesis}%

\global\long\def\ulAut{\operatorname{\underline{Aut}}}%

\global\long\def\ulHom{\operatorname{\underline{Hom}}}%

\global\long\def\ulIso{\operatorname{\underline{{Iso}}}}%

\global\long\def\Utg{\operatorname{Utg}}%

\global\long\def\Unt{\operatorname{Unt}}%

\global\long\def\a{\mathrm{a}}%

\global\long\def\AdGp{\mathrm{AdGp}}%

\global\long\def\Aff{\mathbf{Aff}}%

\global\long\def\calm{\mathrm{calm}}%

\global\long\def\D{\mathrm{D}}%

\global\long\def\Df{\mathrm{Df}}%

\global\long\def\GG{\mathrm{GalGps}}%

\global\long\def\hattimes{\hat{\times}}%

\global\long\def\hatotimes{\hat{\otimes}}%

\global\long\def\hyphen{\textrm{-}}%

\global\long\def\id{\mathrm{id}}%

\global\long\def\iper{\mathrm{iper}}%

\global\long\def\indperf{\mathrm{iper}}%

\global\long\def\Jac{\mathrm{Jac}}%

\global\long\def\lcr{\mathrm{lcr}}%

\global\long\def\nor{\mathrm{nor}}%

\global\long\def\pt{\mathbf{pt}}%

\global\long\def\pur{\mathrm{pur}}%

\global\long\def\perf{\mathrm{perf}}%

\global\long\def\pper{\mathrm{pper}}%

\global\long\def\properf{\mathrm{pper}}%

\global\long\def\pr{\mathrm{pr}}%

\global\long\def\rig{\mathrm{rig}}%

\global\long\def\red{\mathrm{red}}%

\global\long\def\reg{\mathrm{reg}}%

\global\long\def\rep{\mathrm{rep}}%

\global\long\def\rR{\mathrm{R}}%

\global\long\def\sep{\mathrm{sep}}%

\global\long\def\Set{\mathbf{Set}}%

\global\long\def\Sub{\mathrm{Sub}}%

\global\long\def\sm{\mathrm{sm}}%

\global\long\def\sing{\mathrm{sing}}%

\global\long\def\sht{\mathrm{sht}}%

\global\long\def\st{\mathrm{st}}%

\global\long\def\tame{\mathrm{tame}}%

\global\long\def\tors{\mathrm{tors}}%

\global\long\def\tr{\mathrm{tr}}%

\global\long\def\ut{\mathrm{ut}}%

\global\long\def\utd{\mathrm{utd}}%

\global\long\def\utg{\mathrm{utg}}%

\global\long\def\Var{\mathbf{Var}}%

\global\long\def\VAR{\mathbf{VAR}}%

\global\long\def\WRep{\mathbf{WRep}}%

\tableofcontents{}

\section{Introduction\label{sec:Introduction}}

\subsection{Background}

The aim of this paper is to develop motivic integration over \emph{wild}
Deligne-Mumford (for short, DM) stacks, that is, DM stacks whose stabilizers
may have order divisible by the characteristic of the base field.
Motivic integration was introduced by Kontsevich \cite{Kontsevich-motivic}
and developed by Denef and Loeser \cite{MR1664700} for varieties
in characteristic zero. In earlier papers \cite{MR2027195,MR2271984},
the author treated tame and smooth DM stacks, carrying forward the
approach of Batyrev \cite{MR1677693} and Denef-Loeser \cite{MR1905024}
to the McKay correspondence via motivic integration. An attempt of
generalization to the wild case was started in \cite{MR3230848} and
a grand design of the general theory was outlined in \cite{MR3791224}.
More works in this direction have been done in \cite{MR3431631,MR3508745,MR3730512,MR3665638,Tonini:2017qr,Formal-Torsors-II}
and the present work is based on them. Note that our theory is new
even in characteristic zero, since we allow DM stacks to be singular. 

Motivic integration for formal schemes over a complete discrete valuation
ring by Sebag \cite{MR2075915} is generalization in another direction.
In fact, the present paper treats formal DM stacks over a power series
ring $k\tbrats$. Thus we generalize also his theory except that we
consider only the equal characteristic case and we use a different
version of the complete Grothendieck ring of varieties.

\subsection{Twisted arcs}

In a classical setting of motivic integration, given a $k$-variety
$X$, one considers its\emph{ arcs}, that is, $k$-morphisms $\Spec L\tbrats\to X$
from a formal disk to $X$, where $L$ is an extension of $k$. The
space of arcs, denoted by $\J_{\infty}X$, carries the so-called \emph{motivic
measure}. One studies integrals with respect to it, which leads to
useful invariants of $X$. When $X$ is a formal scheme over $k\tbrats$,
one considers $k\tbrats$-morphisms $\Spf L\tbrats\to X$ instead. 

In generalization to DM stacks, we need to consider \emph{twisted
arcs. }Roughly, it means that we replace a formal disk $\Spec L\tbrats$
(or $\Spf L\tbrats$) with twisted formal disks; a \emph{twisted formal
disk} is the quotient stack $[E/G]$ associated to a normal Galois
cover of a formal disk with the Galois group $G$. In the tame case,
we only need to consider the $\mu_{l}$-covers $\Spec L\llbracket t^{1/l}\rrbracket\to\Spec L\tbrats$,
because they are the only possible Galois covers if $L$ is algebraically
closed. In the wild case, there are infinitely many Galois covers
even if we fix an algebraically closed coefficient field $L$ and
a Galois group. Moreover they cannot be parameterized by any finite
dimensional space. We need to take all Galois extensions into account
and to construct the moduli space of twisted arcs on a formal DM stack
$\cX$, which we denote by $\cJ_{\infty}\cX$. 

\subsection{The change of variables formula}

After defining the motivic measure on the space $\cJ_{\infty}\cX$
of twisted arcs, we then need to prove the \emph{change of variables
formula}. In the case of schemes, the formula is written as
\[
\int_{\J_{\infty}X}\LL^{h}\,d\mu_{X}=\int_{\J_{\infty}Y}\LL^{h\circ f_{\infty}-\fj_{f}}\,d\mu_{Y},
\]
say for a proper birational morphism $f\colon Y\to X$. Here $\fj_{f}$
is the Jacobian order function of $f$. The following generalization
the change of variables formula to DM stacks is one of our main results
(see Theorems \ref{thm:change-vars-II} and \ref{thm:change-vars-II-1}
for more general versions):
\begin{thm}
\label{thm:intro-change-vars}Let $\cY,\cX$ be DM stacks of finite
type over $k$ and suppose that they have the same pure dimension
and are generically smooth over $k$. Let $f\colon\cY\to\cX$ be a
proper birational morphism. Then, for a measurable function $h$ on
$\cJ_{\infty}\cX$, 
\[
\int_{\cJ_{\infty}\cX}\LL^{h+\fs_{\cX}}\,d\mu_{\cX}=\int_{\cJ_{\infty}\cY}\LL^{h\circ f_{\infty}-\fj_{f}+\fs_{\cY}}\,d\mu_{\cY}.
\]
\end{thm}

Note that this theorem as well as some theorems below generalizes
to formal DM stacks over $k\tbrats$. Actually we first prove them
for formal DM stacks over $k\tbrats$ and then specialize them to
the case of DM stacks over $k$. 

We call the new terms $\fs_{\cX}$ and $\fs_{\cY}$ in the above formula
\emph{shift functions. }When $\cX$ is tame and smooth, the function
$\fs_{\cX}$ is basically the same as the \emph{age} invariant in
the McKay correspondence or the \emph{fermion shift} in physics (see
\cite{MR1886756,MR1233848}). In the wild and linear case, the function
$\fs_{\cX}$ has been determined in \cite{MR3791224} and has turned
out to be related to \emph{Artin and Swan conductors} \cite{MR3431631}.
Furthermore it later turned out to be essentially the same as Fr{\"o}hlich's
\emph{module resolvent} \cite{MR0414520}. Determining $\fs_{\cX}$
in the non-linear and singular cases (Definition \ref{def:shift})
is also one of new results in this paper. 

\subsection{Stringy motives}

As an application of Theorem \ref{thm:intro-change-vars}, we obtain
the main propery of stringy motives of stacky log pairs, that is,
the invariance under crepant birational transforms. Consider a log
pair $(\cX,A)$ as is standard in birational geometry, allowing $\cX$
to be a normal DM stack over $k$. We define its \emph{stringy motive}
to be the motivic integral
\[
\M_{\st}(\cX,A):=\int_{\cJ_{\infty}\cX}\LL^{\ff_{(\cX,A)}+\fs_{\cX}}\,d\mu_{\cX},
\]
where $\ff_{(\cX,A)}$ is a function measuring the difference of $\bigwedge^{d}\Omega_{\cX/k}$
and the (log) canonical sheaf of $(\cX,A)$ (Definition \ref{def:stringy-motive}).
This generalizes the corresponding invariants previously considered
in more restrictive situations (for instances, see \cite{MR1677693,MR1905024,MR2271984}).
When $\cX$ is a smooth algebraic space and $A=0$, this invariant
specializes to the class $\{\cX\}$ of $\cX$ in our version of the
complete Grothendieck ring of varieties. Using Theorem \ref{thm:intro-change-vars},
we will prove the main property of this invariant:
\begin{thm}[Theorem \ref{thm:Mst-crepant-nonformal}]
\label{thm:intro-Mst-crepant}For a proper birational (not necessarily
representable) morphism $\cY\to\cX$ giving a crepant morphism $(\cY,B)\to(\cX,A)$
of log pairs, we have 
\[
\M_{\st}(\cY,B)=\M_{\st}(\cX,A).
\]
\end{thm}

This is a broad generalization of the theorem of Batyrev \cite{MR1714818}
and Kontsevich \cite{Kontsevich-motivic} that K-equivalent complex
smooth proper varieties have equal Hodge numbers as well as its counterpart
for complex orbifolds \cite{MR2069013,MR2027195,MR2271984}.

\subsection{The wild McKay correspondence}

A typical example of a proper birational morphism of DM stacks is
the morphism 
\[
[\AA_{k}^{d}/G]\to\AA_{k}^{d}/G
\]
from the quotient stack to the quotient scheme associated to a linear
action of a finite group. When the action has no pseudo-reflection,
this is crepant. Theorem \ref{thm:intro-Mst-crepant} implies the
following theorem, which we call the \emph{wild McKay correspondence:}
\begin{thm}[Corollaries \ref{cor:wild-McKay-linear} and \ref{cor:wild-McKay-linear-1}]
\label{thm:intro-wild-McKay}Suppose that a finite group $G$ linearly
acts on an affine space $\AA_{k}^{d}$ without pseudo-reflection.
Then 
\[
\M_{\st}(\AA_{k}^{d}/G)=\int_{\Delta_{G}}\LL^{d-v}.
\]
 
\end{thm}

Here $\Delta_{G}$ denotes the moduli space of $G$-torsors over the
punctured formal disk $\Spec k\tpars$ and $v$ is a function on $\Delta_{G}$
closely related to the shift function $\fs_{\cX}$. The well-definedness
of this integral was proved in \cite{Formal-Torsors-II}. The integral
on the right side is interpreted as the stringy motive of the quotient
stack $\cX=[\AA_{k}^{d}/G]$. When working over an algebraically closed
field of characteristic zero, the integral on the right side reduces
to the finite sum 
\[
\sum_{[g]\in\Conj(G)}\LL^{d-\age(g)}.
\]
Thus Theorem \ref{thm:intro-wild-McKay} is a generalization of the
McKay correspondence by Batyrev \cite{MR1677693} and Denef-Loeser
\cite{MR1905024} in characteristic zero. Note that Theorem \ref{thm:intro-wild-McKay}
for the cyclic group of prime order equal to the characteristic was
essentially obtained in \cite{MR3230848} and that the point-counting
version of the theorem was proved in \cite{MR3730512}. 

We give two applications of the wild McKay correspondence. The first
application is an estimate of discrepancy of quotient singularities;
the discrepancy is a basic invariant of singularities in the minimal
model program. The following corollary generalizes results in \cite{MR3230848,MR3929517}
for the cyclic group of prime order to an arbitrary finite group. 
\begin{cor}[Corollary \ref{cor:discrep}]
\label{cor:intro-discrep}Suppose that a finite group $G$ linearly
acts on an affine space $\AA_{k}^{d}$ and that $G$ has no pseudo-reflection. 
\begin{enumerate}
\item We have
\begin{align*}
 & \mathrm{discrep}(\mathrm{centers}\subset X_{\sing};X)\\
 & =d-1-\max\left\{ \dim X_{\sing},\dim\int_{\Delta_{G}\setminus\{o\}}\LL^{d-v}\right\} .
\end{align*}
Here $o\in\Delta_{G}$ is the point corresponding to the trivial $G$-torsor.
\item If the integral $\int_{\Delta_{G}}\LL^{d-v}$ converges, then $X$
is log terminal. If $X$ has a log resolution, then the converse is
also true.
\end{enumerate}
\end{cor}

This corollary allows us to compute the discrepancy in theory by computing
the integral $\int_{\Delta_{G}}\LL^{d-v}$, which is more of arithmetic
nature. It was carried out in \cite{MR3929517,MR3230848} for the
group $\ZZ/p\ZZ$ with $p$ the characteristic of $k$. The case of
$\ZZ/p^{2}\ZZ$ will be studied in a forthcoming paper by Tanno and
the author \cite{Tanno-Yasuda}. 

The second application of the wild McKay correspondence is the following
motivic version of Bhargava's mass formula \cite{MR2354798}:
\begin{cor}[Corollary \ref{cor:motivic-Bhargava}]
\label{cor:intro-motivic-Bhargava}Let $S_{n}$ be the $n$th symmetric
group and let $\ba\colon\Delta_{S_{n}}\to\ZZ$ be the Artin conductor.
Then
\[
\int_{\Delta_{S_{n}}}\LL^{-\ba}=\sum_{j=0}^{n-1}P(n,n-j)\LL^{-j},
\]
where $P(n,m)$ denotes the number of partitions of $n$ into exactly
$m$ parts. 
\end{cor}

The proof of this is basically translation of the proof given in \cite{MR3730512}
of the original formula into the motivic context. 

\subsection{Untwisting by Hom stacks}

The key ingredient in construction of our theory is the technique
of \emph{untwisting}, which reduces the study of twisted arcs to the
one of ordinary/untwisted arcs. The origin of this technique is found
in \cite{MR1905024}. In \cite{MR3791224,MR3508745}, it was further
generalized to wild linear actions and non-linear actions on affine
schemes respectively. In this paper, we develop it further and make
it more intrinsic by using Hom stacks. 

For untwisting, it is essential to work with formal DM stacks over
$k\tbrats$. Based on a work of Tonini and the author \cite{Tonini:2017qr},
we construct a moduli stack of twisted formal disks, denoted by $\Theta$,
and a universal family $\cE_{\Theta}\to\Spf k\tbrats\times\Theta$.
For a formal DM stack $\cX$ of finite type over $\Spf k\tbrats$,
we define its \emph{untwisting stack} as
\[
\Utg_{\Theta}(\cX):=\ulHom_{\Spf k\tbrats\times\Theta}^{\rep}(\cE_{\Theta},\cX\times\Theta),
\]
the Hom stack of representable morphisms. For a technical reason,
we modify it to what is denoted by $\Utg_{\Gamma}(\cX)^{\pur}$ by
stratifying $\Theta$ to another stack $\Gamma$ and killing $t$-torsions.
Then the desired stack $\cJ_{\infty}\cX$ of twisted arcs is obtained
as the stack of untwisted arcs $\J_{\infty}\Utg_{\Gamma}(\cX)^{\pur}$
of $\Utg_{\Gamma}(\cX)^{\pur}$. We can give it the motivic measure
in the same way as in the case of schemes. 

The proof of Theorem \ref{thm:intro-change-vars} is also based on
untwisting. We reduce the theorem to the case where $\cX$ is a formal
algebraic space $X$. The given morphism $f\colon\cY\to X$ induces
a morphism 
\[
f^{\utg}\colon\Utg_{\Gamma}(\cY)^{\pur}\to X,
\]
which in turn induces a map
\[
\cJ_{\infty}\cY=\J_{\infty}\Utg_{\Gamma}(\cY)^{\pur}\to\J_{\infty}X.
\]
Since this is a map of untwisted arc spaces, it is relatively easy
to generalize the change of variables formula to this map. We will
complete the proof by rewriting the formula in terms of $f$ instead
of $f^{\utg}$. At this point the shift function $\fs_{\cY}$ naturally
appears as the difference of the Jacobian orders of $f$ and $f^{\utg}$. 

\subsection{Related topics}

To end Introduction, we mention a few related topics. Understanding
wild ramification in higher dimension is one of the central themes
in arithmetic geometry (for instances, see \cite{MR2827799,MR3392531}).
The theory developed in this paper would be regarded as a new approach
to it. Given a wildly ramified finite morphism $f\colon Y\to X$ of
some spaces, one natural method is to look at ramification of $f$
along a general curve $C\to X$. In our approach, we look at all ``curves''
rather than a general one and integrate all information on ramification
along them. Our theory should be closely related to others, however
their precise relation have not been studied yet and it remains as
a future problem.

Recently Groechenig, Wyss and Ziegler \cite{1707.06417} found an
interesting application of a version of $p$-adic integration on DM
stacks to the study of Hitchin fibration. The motivic version of their
result was then obtained by Loeser and Wyss \cite{1912.11638}. Although
they considered only the tame case, it would be interesting to incorporate
the wild action into the subject. 

\subsection{Terminology, notation and convention\label{sec:Terminology}}

Here we fix our basic set-up such as terminology, notation and convention.
For reader's convention, we explain some notions whose definition
will be given later.

We denote the set of nonnegative integers by $\NN$. For a category
$\cC$, we mean by $c\in\cC$ that $c$ is an object of $\cC$. We
use the symbol $\coprod$ to mean the coproduct (abstract disjoint
union) and reserve the symbol $\bigsqcup$ to mean the disjoint union
of subsets.

We fix a base field $k$ of characteristic $p\ge0$. In Section \ref{sec:Frobenius-morphisms}
and Subsection \ref{subsec:Ramification}, we assume $p>0$. At the
beginning of these (sub)sections, we will repeat this assumption.

We denote by $\Aff$ or $\Aff_{k}$ the category of affine $k$-schemes.
For an affine $k$-scheme $S=\Spec R$, we denote by $\Aff_{R}$ or
$\Aff_{S}$ the category of affine schemes over $S$. Unless otherwise
noted, all schemes, algebraic spaces and stacks as well as their morphisms
are defined over $k$. Thus they are regarded as categories fibered
in groupoids over $\Aff$. Similarly, unless otherwise noted, all
rings and fields contain $k$ and their maps are supposed to be $k$-homomorphisms.
We consider only \emph{separated} schemes, algebraic spaces and (formal)
DM stacks. Namely their diagonal morphism are proper and hence finite.

A DM stack $\cX$ is said to be \emph{almost of finite type }if $\cX$
is isomorphic to the coproduct $\coprod_{i}\cX_{i}$ of countably
many DM stacks $\cX_{i}$ of finite type.

For a $k$-algebra $R$, we denote by $R\tbrats$ (resp.\ $R\tpars$)
the ring of power series (resp.\ Laurent power series) over $R$.
Thus $R\tpars$ is the localization $R\tbrats_{t}$ of $R\tbrats$
by $t$ (not the total quotient ring of $R\tbrats$). For an affine
scheme $S=\Spec R$, we denote by $\D_{S}$ (resp.\ $\D_{S}^{*}$,
$\Df_{S}$) the affine scheme $\Spec R\llbracket t\rrbracket$ (resp.\ the
affine scheme $\Spec R\tpars$, the formal affine scheme $\Spf R\tbrats$).
We often replace the subscript $S$ by $R$. When $R=k$, we usually
omit the subscript.

For a stack $\cX$ over $\Aff$, by the 2-Yoneda lemma, we often identify
an object of $\cX$ over $S\in\Aff$ with a morphism $S\to\cX$. We
call it an $S$-\emph{point }of $\cX$. We denote the point set of
$\cX$ by $|\cX|$; the set of equivalence classes of geometric points
$\Spec K\to\cX$. We give it the Zariski topology. For a field $K$,
we denote the set of isomorphism classes of $K$-points, $\cX(K)/\cong$,
by $\cX[K]$. More generally, for a subset $C\subset|\cX|$, $C(K)$
denotes the full subcategory of $\cX(K)$ consisting of $K$-points
$\Spec K\to\cX$ mapping into $C$. Then $C[K]$ denotes its set of
isomorphism classes, $C[K]:=C(K)/\cong$.

For a formal DM stack $\cX$ over $\Df$ and $n\in\NN$, we denote
by $\cX_{n}$ the fiber product $\cX\times_{\Df}\Spec k\tbrats/(t^{n+1})$,
unless otherwise noted.

We say that a morphism $\cY\to\cX$ of DM stacks is a \emph{thickening
}if it is representable (by algebraic spaces) and a closed immersion
such that the map $|\cY|\to|\cX|$ is bijective.

We say that a finite group $G$ is \emph{tame }if $p$ does not divide
$\sharp G$. We say that a formal DM stack is \emph{tame }if the automorphism
group of every geometric point is tame.

\subsection{Acknowledgments}

This work relies on results obtained in my joint works with Fabio
Tonini. I thank him for fruitful cooperation. I started to write the
present paper after the first draft of \cite{Formal-Torsors-II} was
completed. I thank also Takashi Suzuki for helpful discussion.

\section{Ind-DM stacks and formal DM stacks\label{sec:ind-formal}}

In this section, we introduce formal DM stacks, which are the main
object of our study, as a special case of ind-DM stacks. We then study
their basic properties such as ones concerning coarse moduli spaces.
\begin{defn}
\label{def:representable}Let $f\colon\cY\to\cX$ be a morphism of
stacks. We say that $f$ is \emph{representable }if for every morphism
$U\to\cX$ from an algebraic space, the fiber product $U\times_{\cX}\cY$
is isomorphic to an algebraic space. Let $\mathbf{P}$ be a property
of morphisms of algebraic spaces which is stable under base changes.
We say that a representable morphism $f\colon\cY\to\cX$ of stacks
\emph{has property $\mathbf{P}$ }if for every morphism $U\to\cX$
from an algebraic space, the projection $\cY\times_{\cX}U\to U$ has
property $\bP$.
\end{defn}

\begin{defn}
\label{def:ind-DM-coarse}An \emph{ind-DM stack }(resp.\ \emph{ind-algebraic
space, ind-scheme}, \emph{ind-affine scheme})\emph{ }is a fibered
category over $\Aff$ isomorphic to the limit $\varinjlim\cX_{n}$
of an inductive system 
\[
\cX_{0}\to\cX_{1}\to\cdots
\]
indexed by $\NN$ of DM stacks (reps.\ algebraic spaces, schemes,
affine schemes) in the sense of \cite[Appendix A]{Tonini:2017qr}.
An ind-DM stack (resp.\ ind-algebraic space, ind-scheme, ind-affine
scheme) is called a \emph{formal DM stack }(resp.\ \emph{formal algebraic
space, formal scheme}, \emph{formal affine scheme}) if we can further
assume transition maps $\cX_{n}\to\cX_{n+1}$ are (necessarily representable)
thickenings.
\end{defn}

From \cite[Prop. A.5]{Tonini:2017qr}, an ind-DM stack is a stack.
In \cite[tag 0AIH]{stacks-project}, a formal affine scheme is called
a countably indexed affine formal algebraic space. There it is showed
that a formal affine scheme is written as 
\[
\Spf A=\varinjlim\Spec A/I_{n}
\]
where $A$ is a weakly admissible topological ring which has a fundamental
system 
\[
I_{0}\supset I_{1}\supset\cdots
\]
of neighborhoods of $0\in A$ consisting of weak ideals of definition.
In a linearly topologized ring $A$, a weak ideal of definition is
an ideal consisting of topologically nilpotent elements (elements
$f$ such that $f^{n}\to0$ as $n\to\infty$) and a weakly admissible
ring is a separated and complete linearly topologized ring having
a weak ideal of definition.

For a formal DM stack $\cX=\varinjlim\cX_{n}$, its point set $|\cX|$
is defined in the same way as for a DM stack; it is the set of equivalence
classes of geometric points $\Spec K\to\cX$ (that is, morphisms with
$K$ algebraically closed fields). It is identified with $|\cX_{n}|$
for any $n$. Thus $|\cX|$ has the topology induced from the Zariski
topology on $|\cX_{n}|$.
\begin{lem}
\label{lem:diag-immer}Let $\cX$ be a formal DM stack.
\begin{enumerate}
\item The diagonal morphism $\Delta\colon\cX\to\cX\times\cX$ is representable,
unramified and finite (and hence affine).
\item For morphisms $U\to\cX$ and $V\to\cX$ from affine schemes (resp.\ formal
affine schemes), the fiber product $U\times_{\cX}V$ is an affine
scheme (resp.\ a formal affine scheme).
\end{enumerate}
\end{lem}

\begin{proof}
(1) Suppose $\cX\cong\varinjlim\cX_{n}$ for an inductive system $\cX_{n}$,
$n\in\NN$ as in the definition. Let $S$ be an affine scheme and
let $f\colon S\to\cX\times\cX$ be a morphism. By definition of limits,
it factors through $\cX_{n}\times\cX_{n}$ for $n\gg0$. From \cite[Prop. A.2]{Tonini:2017qr},
\[
S\times_{f,\cX\times\cX,\Delta}\cX\cong\varinjlim(S\times_{\cX_{n}\times\cX_{n}}\cX_{n}).
\]
However transition morphisms 
\[
S\times_{\cX_{n}\times\cX_{n}}\cX_{n}\to S\times_{\cX_{n+1}\times\cX_{n+1}}\cX_{n+1}
\]
are isomorphisms. Therefore, for $n\gg0$, 
\[
S\times_{f,\cX\times\cX,\Delta}\cX\cong S\times_{\cX_{n}\times\cX_{n}}\cX_{n}.
\]
Since the diagonal of $\cX_{n}$ is representable, unramified and
finite, the right side of this isomorphism is an affine scheme which
is unramified and finite over $S$. This shows (1).

(2) From $(1)$, if $U$ and $V$ are affine schemes, then 
\[
U\times_{\cX}V\cong(U\times V)\times_{\cX\times\cX,\Delta}\cX
\]
is also an affine scheme. Next suppose that $U$ and $V$ are formal
affine schemes and written as $\varinjlim U_{n}$ and $\varinjlim V_{n}$
as limits of affine schemes respectively. Moreover we may suppose
that morphisms $U\to\cX$ and $V\to\cX$ are given by morphisms $U_{n}\to\cX_{n}$
and $V_{n}\to\cX_{n}$. Then 
\[
U\times_{\cX}V\cong\varinjlim(U_{n}\times_{\cX_{n}}V_{n}).
\]
Since $U_{n}\times_{\cX_{n}}V_{n}$ are affine schemes and their transition
morphisms are thickenings, the limit is a formal affine scheme.
\end{proof}
\begin{lem}
\label{lem:atlas}Let $\cX$ be a formal DM stack. There exists a
family of representable etale morphisms $V_{i}\to\cX$, $i\in I$
such that $V_{i}$ are formal affine schemes and $\coprod_{i}V_{i}\to\cX$
is surjective.
\end{lem}

\begin{proof}
Suppose $\cX=\varinjlim\cX_{n}$ where $\cX_{n}$, $n\in\NN$ are
DM stacks. There exists a family of etale morphisms $V_{i,0}\to\cX_{0}$
from affine schemes such that $\coprod_{i}V_{i,0}\to\cX_{0}$ is surjective.
From \ref{lem:etale-invariance}, they uniquely extend to $V_{i,n}\to\cX_{n}$
with the same property and the diagrams 
\[
\xymatrix{V_{i,n}\ar[r]\ar[d] & V_{i,n+1}\ar[d]\\
\cX_{n}\ar[r] & \cX_{n+1}
}
\]
are 2-Cartesian. Let $V_{i}:=\varinjlim V_{i,n}$, which is a formal
affine scheme. From \cite[A.3]{Tonini:2017qr}, the diagrams
\[
\xymatrix{V_{i,n}\ar[r]\ar[d] & V_{i}\ar[d]\\
\cX_{n}\ar[r] & \cX
}
\]
are also 2-Cartesian. Therefore the morphisms $V_{i}\to\cX$ satisfy
the desired property.
\end{proof}
\begin{defn}
We call a morphism $\coprod_{i}V_{i}\to\cX$ as above or the formal
scheme $\coprod_{i}V_{i}$ an \emph{atlas }of $\cX$.
\end{defn}

\begin{defn}
We define $\D_{n}:=\Spec k\tbrats/(t^{n+1})$ and $\Df:=\Spf k\tbrats=\varinjlim\D_{n}$.
For a DM stack $\Phi$, we define $\D_{\Phi,n}:=\D_{n}\times\Phi$
and $\Df_{\Phi}:=\Df\times\Phi=\varinjlim\D_{\Phi,n}$. When $\Phi=\Spec R$,
we also write them as $\D_{R,n}$ and $\Df_{R}$, which are isomorphic
to $\Spec R\tbrats/(t^{n+1})$ and $\Spf R\tbrats$. (We denote by
$\D_{R}$ the affine scheme $\Spec R\tbrats$.)
\end{defn}

\begin{defn}
We say that a morphism $f\colon\cY\to\cX$ of formal DM stacks is
\emph{of finite type }if for every morphism $\cU\to\cX$ from a DM
stack, the fiber product $\cU\times_{\cX}\cY$ is a DM stack of finite
type over $\cU$.
\end{defn}

\begin{notation}
For a formal DM stack $\cX$ of finite type over $\Df_{\Phi}$ and
$n\in\NN$, we denote the fiber product $\cX\times_{\Df}\D_{n}=\cX\times_{\Df_{\Phi}}\D_{\Phi,n}$
by $\cX_{n}$.
\end{notation}

With this notation, we have 
\[
\cX\cong\cX\times_{\Df}\varinjlim\D_{n}\cong\varinjlim\cX_{n}.
\]

Recall that the restricted power series ring $k\tbrats\{x_{1},\dots,x_{n}\}$
is the $t$-adic completion of $k\tbrats[x_{1},\dots,x_{n}]$. A $k\tbrats$-algebra
is topologically finitely generated over $k\tbrats$ if and only if
it is isomorphic to a quotient ring of $k\tbrats\{x_{1},\dots,x_{n}\}$
for some $n\in\NN$.
\begin{lem}
Let $\cX$ be a formal DM stack of finite type over $\Df$. Then there
exists an atlas $\Spf A\to\cX$ from a formal affine scheme such that
$A$ is a topologically finitely generated $k\tbrats$-algebra.
\end{lem}

\begin{proof}
Let us take an atlas $\Spec B_{0}\to\cX_{0}$ from an affine scheme.
The ring $B_{0}$ is finitely generated as a $k$-algebra. From \ref{lem:etale-invariance},
we can inductively construct the following 2-Cartesian diagrams for
$n\in\NN$:
\[
\xymatrix{\Spec B_{n}\ar[r]\ar[d] & \Spec B_{n+1}\ar[d]\\
\cX_{n}\ar[r] & \cX_{n+1}
}
\]
Let $A:=\varprojlim B_{n}$. From \cite[Ch. 0, Prop. 7.5.5]{MR0163908},
$A$ is topologically finitely generated as a $k\tbrats$-algebra.
The morphism $\Spf A\to\cX$ is an atlas.
\end{proof}
\begin{lem}
\label{lem:formal-finite-type}Let $\cX_{n}$, $n\in\NN$ be DM stacks
of finite type over $\D_{n}$ and let $\cX_{n}\to\cX_{n+1}$, $n\in\NN$
be morphisms compatible with $\D_{n}\to\D_{n+1}$. Suppose that $\cX_{n}\to\cX_{n+1}\times_{\D_{n+1}}\D_{n}$
is an isomorphism. Then $\cX:=\varinjlim\cX_{n}$ is a formal DM stack
of finite type over $\Df$.
\end{lem}

\begin{proof}
From the assumption, $\cX_{n}\to\cX_{n+1}$ is a thickening. Therefore
$\cX$ is a formal DM stack. To check that it is of finite type over
$\Df$, it suffices to consider a morphism $\cU\to\Df$ from a quasi-compact
DM stack. The morphism factors through $\D_{n}$, for $n\gg0$. We
have 
\[
\cU\times_{\Df}\cX\cong\cU\times_{\D_{n}}\D_{n}\times_{\Df}\cX\cong\cU\times_{\D_{n}}\cX_{n},
\]
which is of finite type over $\cU$.
\end{proof}
\begin{defn}
\label{def:coarse-moduli}A \emph{coarse moduli space} of an ind-DM
stack $\cX$ is an ind-algebraic space $X$ together with a morphism
$\pi\colon\cX\to X$ such that
\begin{enumerate}
\item for any morphism $f\colon\cX\to Y$ with $Y$ an ind-algebraic space,
there exists a unique morphism $g\colon X\to Y$ such that $g\circ\pi=f$.
\item for any algebraically closed field $K/k$, the map 
\[
\cX[K](:=\cX(K)/\cong)\to Y(K)
\]
 is bijective.
\end{enumerate}
Obviously, if exists, a coarse moduli space of an ind-DM stack $\cX$
is unique up to unique isomorphism. Therefore we call it \emph{the
}coarse moduli space of $\cX$.
\end{defn}

\begin{lem}
\label{lem:coarse-ind}For a formal DM stack $\cX=\varinjlim\cX_{n}$
with $\cX_{n}$ quasi-compact, let $X_{n}$ be the coarse moduli space
of $\cX_{n}$ (which exists from Keel-Mori's theorem \cite{MR1432041}).
The limit $X=\varinjlim X_{n}$ is the coarse moduli space of $\cX$.
\end{lem}

\begin{proof}
Condition (2) in \ref{def:coarse-moduli} clearly holds. We will show
the universality. Let $f\colon\cX\to Y$ be a morphism with $Y=\varinjlim Y_{i}$
an ind-algebraic space. Since $\cX_{n}$ is a quasi-compact DM stack,
the composition $\cX_{n}\to\cX\to Y$ factors through $Y_{i(n)}$
for some $i(n)$. The morphism $\cX_{n}\to Y_{i(n)}$ uniquely factors
as $\cX_{n}\to X_{n}\to Y_{i(n)}$. We can take a strictly increasing
function as $i(n)$. The morphisms $X_{n}\to Y_{i(n)}$ define a morphism
$g\colon X\to Y$ such that the composition $\cX\to X\to Y$ is equal
to $f$. It remains to show the uniqueness of the morphism $X\to Y$.
Let $h\colon X\to Y$ be another morphism such that $\cX\to X\xrightarrow{h}Y$
is equal to $f$. Suppose that $h$ is represented by morphisms $h_{n}\colon X_{n}\to Y_{j(n)}$.
Replacing both $i(n)$ and $j(n)$ with $\max\{i(n),j(n)\}$, we may
suppose that $i(n)=j(n)$. From the universality of the coarse moduli
space $\cX_{n}\to X_{n}$, two morphisms $g_{n},h_{n}\colon X_{n}\rightrightarrows Y_{i(n)}$
coincide. Therefore $g=h$.
\end{proof}
\begin{lem}
\label{lem:coarse-flat-base-change}Let $\cX$ be a formal DM stack
with the coarse moduli space $\cX\to X$. Let $Y\to X$ be a morphism
of ind-algebraic spaces which is representable by algebraic spaces
and flat. Then $Y$ is the coarse moduli space of $\cX\times_{X}Y$.
\end{lem}

\begin{proof}
Let us write $\cX$ as the limit $\varinjlim\cX_{n}$ of DM stacks
$\cX_{n}$ and let $X_{n}$ be the coarse moduli space of $\cX_{n}$.
Let $Y_{n}:=Y\times_{X}X_{n}$ and $\cY_{n}:=\cX_{n}\times_{X_{n}}Y_{n}$.
Since $Y_{n}\to X_{n}$ is a flat morphism of algebraic spaces, $Y_{n}$
is the coarse moduli space of $\cY_{n}$. From \cite[Prop. A.2]{Tonini:2017qr}
and \ref{lem:coarse-ind}, we have 
\[
Y\cong Y\times_{X}\varinjlim X_{n}\cong\varinjlim(Y\times_{X}X_{n})\cong\varinjlim Y_{n}.
\]
Similarly 
\[
\cX\times_{X}Y\cong\cX\times_{X}\varinjlim Y_{n}\cong\varinjlim\cY_{n}.
\]
From \ref{lem:coarse-ind}, $Y\cong\varinjlim Y_{n}$ is the coarse
moduli space of $\cX\times_{X}Y\cong\varinjlim\cY_{n}$.
\end{proof}
For further analysis of coarse moduli spaces, we recall the following
result.
\begin{lem}[{\cite[pages 1160 and 1172]{MR595009}}]
\label{lem:invariant-subring}Let $R$ be a Noetherian $k$-algebra
such that the relative Frobenius $F_{R/k}\colon\Spec R\to(\Spec R)^{(1)}$
(see Section \ref{sec:Frobenius-morphisms} for the relative Frobenius
morphism) is a finite morphism (this holds for instance if $R$ is
a quotient of $k\llbracket t_{1},\dots,t_{m}\rrbracket\{x_{1},\dots,x_{n}\}$).
Suppose that a finite group $G$ acts on $R$ as a $k$-algebra. Then
the invariant subalgebra $R^{G}$ is Noetherian and $R$ is a finitely
generated $R^{G}$-module.
\end{lem}

\begin{lem}
\label{lem:Mittag-Leffler}With the same assumption as in \ref{lem:invariant-subring},
for an ideal $I\subset R^{G}$ and for every $n$, we have
\[
R^{G}/I^{n}=\Image\left((R/I^{n'}R)^{G}\to(R/I^{n}R)^{G}\right)\quad(n'\gg n).
\]
In particular, the projective system $(R/I^{n}R)^{G}$, $n\in\NN$
satisfies the Mittag-Leffler condition. 
\end{lem}

\begin{proof}
Since $R$ is a finitely generated $R^{G}$-module, we can identify
$R$ with the quotient module 
\[
(R^{G})^{\oplus l}/\sum_{q=1}^{c}R^{G}\cdot{}^{t}(f_{1}^{(q)},\dots,f_{l}^{(q)})\quad(f_{j}^{(q)}\in R^{G}).
\]
Here we regard elements of $(R^{G})^{\oplus l}$ as column vectors.
For each $g\in G$, we fix a lift $\tilde{g}\colon(R^{G})^{\oplus l}\to(R^{G})^{\oplus l}$
of $g\colon R\to R$ and suppose that it is given by the left multiplication
with a $l$-by-$l$ matrix $(\tilde{g}_{ij})_{1\le i,j\le n}$, $\tilde{g}_{ij}\in R^{G}$.
Then $(R/I^{n}R)^{G}$ consists of the residue classes $\overline{^{t}(x_{1},\dots,x_{l})}$
of those vectors $^{t}(x_{1},\dots,x_{l})\in(R^{G})^{\oplus l}$ such
that for some $y_{j}^{(q)}\in R^{G}$ $(1\le q\le c,\,1\le j\le l)$,
we have
\[
x_{i}-\sum_{j=1}^{l}\tilde{g_{ij}}x_{j}+\sum_{q=1}^{c}\sum_{j=1}^{l}y_{j}^{(q)}f_{j}^{(q)}\equiv0\mod I^{n}\quad(g\in G,\,1\le i\le l).
\]
We regard these equations as $R^{G}$-linear equations with unknowns
$x_{i},y_{j}^{(q)}$ and apply a version of approximation theorem
for linear equations, \cite[Th. 3.1]{MR2225775}; there exists $i_{0}\in\NN$
such that for a solution $(x_{i},y_{j}^{(q)})_{i,j,q}\in(R^{G})^{\oplus(c+1)l}$
of the above equation for $n=i+i_{0}$, there exists $(\check{x}_{i},\check{y}_{j}^{(q)})_{i,j,q}\in(R^{G})^{\oplus(c+1)l}$
satisfying
\[
\check{x}_{i}-\sum_{j=1}^{l}\tilde{g_{ij}}\check{x}_{j}+\sum_{q=1}^{c}\sum_{j=1}^{l}\check{y}_{j}^{(q)}f_{j}^{(q)}=0\quad(g\in G,\,1\le i\le l)
\]
and
\[
(x_{i},y_{j}^{(q)})_{i,j,q}\equiv(\check{x}_{i},\check{y}_{j}^{(q)})_{i,j,q}\mod I^{i}.
\]
This implies that if an element $\bx$ of $(R/I^{i}R)^{G}$ lifts
to an element of $(R/I^{i+i_{0}}R)^{G}$, then $\bx$ lifts also to
an element of $R^{G}$. The first assertion of the lemma follows.
The second assertion is then obvious.
\end{proof}
\begin{cor}
\label{cor:ind-quotient}With the assumption as in \ref{lem:invariant-subring},
for an ideal $I\subset R$, we have a natural isomorphism of ind-schemes
\[
\varinjlim\Spec R^{G}/I^{n}\cong\varinjlim\Spec(R/I^{n}R)^{G}.
\]
\end{cor}

\begin{proof}
Maps $R^{G}/I^{n}\to(R/I^{n}R)^{G}$ define a morphism
\[
a\colon\varinjlim\Spec(R/I^{n}R)^{G}\to\varinjlim\Spec R^{G}/I^{n}.
\]
Let $\beta\colon\NN\to\NN$ be a strictly increasing function such
that 
\[
\Image\left((R/I^{\beta(n)}R)^{G}\to(R/I^{n}R)^{G}\right)=R^{G}/I^{n}.
\]
The maps $(R/I^{\beta(n)}R)^{G}\to R^{G}/I^{n}$ define a morphism
\[
b\colon\varinjlim\Spec R^{G}/I^{n}\to\varinjlim\Spec(R/I^{\beta(n)}R)^{G}\cong\varinjlim\Spec(R/I^{n}R)^{G}.
\]
We see that $a$ and $b$ are inverses to each other.
\end{proof}
Let $\Spf R$ be a formal affine scheme of finite type over $\Df$
with an action of a finite group $G$. The quotient stack $[(\Spf R)/G]$
is defined in the same way as in the case of non-formal schemes and
isomorphic to $\varinjlim[(\Spec R/(t^{n}))/G]$.
\begin{lem}
\label{lem:coarse-inv-sub}With the above notation, suppose that $R$
satisfies the assumption as in \ref{lem:invariant-subring}. Then
the coarse moduli space of $[(\Spf R)/G]$ is $\Spf R^{G}$. Moreover,
if $\Spf R$ is of finite type over $\Df$, then so is $\Spf R^{G}$.
\end{lem}

\begin{proof}
The coarse moduli space of $[(\Spec R/(t^{n}))/G]$ is $\Spec(R/(t^{n}))^{G}$.
From \ref{lem:coarse-ind}, $\varinjlim\Spec(R/(t^{n}))^{G}$ is the
coarse moduli space of $[(\Spf R)/G]$. From \ref{cor:ind-quotient},
this is isomorphic to $\Spf R^{G}$. The first assertion has been
proved. If $\Spf R$ is of finite type over $\Df$, then since $(R/(t^{n}))^{G}\to R^{G}/(t^{m})$
is surjective for $n\gg m$, $(\Spf R^{G})_{m}=\Spec R^{G}/(t^{m+1})$
is of finite type over $k$. From \ref{lem:formal-finite-type}, the
second assertion follows.
\end{proof}
The following proposition says that every formal DM stack of finite
type over $\Df$ is locally a quotient stack.
\begin{prop}
\label{lem:formal-locally-quot}Let $\cX$ be a formal DM stack  of
finite type over $\Df$. Then the coarse moduli space $X$ of $\cX$
is a formal algebraic space  of finite type over $\Df$. Moreover
there exists an atlas $\coprod_{\gamma=1}^{l}V_{\gamma}\to X$ such
that for each $\gamma$, $V_{\gamma}$ is a formal affine scheme of
finite type over $\Df$ and $\cX\times_{X}V_{\gamma}\cong[W_{\gamma}/H_{\gamma}]$
for a formal affine scheme $W_{\gamma}$ with an action of a finite
group $H_{\gamma}$. In particular, we have a stabilizer-preserving
etale surjective morphism $\coprod_{\gamma=1}^{l}[W_{\gamma}/H_{\gamma}]\to\cX$.
\end{prop}

\begin{proof}
Let $\cX_{n}:=\cX\times_{D}D_{n}$ and $X_{n}$ their coarse moduli
spaces. From \ref{lem:coarse-ind}, the ind-algebraic space $X=\varinjlim X_{n}$
is the coarse moduli space of $\cX$. As is well-known (see \cite[Lem. 2.2.3]{MR1862797}),
there exists an etale cover $(V_{0,\gamma}\to X_{0})_{1\le\gamma\le l}$
such that for each $\gamma$, $V_{0,\gamma}$ is an affine scheme
of finite type over $k$, and $\cW_{0,\gamma}:=V_{0,\gamma}\times_{X_{0}}\cX_{0}$
is isomorphic to $[W_{0,\gamma}/H_{\gamma}]$ for some affine scheme
$W_{0,\gamma}$ and a finite constant group $H_{\gamma}$ acting on
it. The morphism $\cW_{0,\gamma}=[W_{0,\gamma}/H_{\gamma}]\to\cX_{0}$
is representable, etale, affine and stabilizer-preserving, because
$V_{0,\gamma}\to X_{0}$ is so. (Recall that we consider only separated
algebraic spaces. Therefore $V_{0,\gamma}\to X_{0}$ is automatically
affine.) We get the diagram
\[
W_{0,\gamma}\times H_{\gamma}\rightrightarrows W_{0,\gamma}\to\cW_{0,\gamma}\to\cX_{0}
\]
of representable etale affine morphisms such that $W_{0,\gamma}\times H_{\gamma}\rightrightarrows W_{0,\gamma}$
defines a groupoid in schemes and $\cW_{0,\gamma}$ is the associated
stack. Let $U\to\cX$ be an etale morphism with $U$ a formal affine
scheme. By pulling back by $U_{0}\to\cX_{0}$, we get a similar diagram
of affine schemes
\[
W_{0,\gamma,U_{0}}\times H_{\gamma}\rightrightarrows W_{0,\gamma,U_{0}}\to\cW_{0,\gamma,U_{0}}\to U_{0}.
\]
Note that $\cW_{0,\gamma,U_{0}}$ is a scheme in spite of its calligraphic
letter. From the invariance of small etale sites \cite[Th. 18.1.2]{MR0238860},
we get the corresponding diagram over each $U_{n}$,
\[
W_{n,\gamma,U_{n}}\times H_{\gamma}\rightrightarrows W_{n,\gamma,U_{n}}\to\cW_{n,\gamma,U_{n}}\to U_{n}.
\]
From \cite[tag 05YU]{stacks-project}, these are all affine schemes
and give quasi-coherent sheaves on the small etale site on $\cX_{n}$.
Taking their relative spectra, we get the diagram of etale affine
morphisms
\[
W_{n,\gamma}\times H_{\gamma}\rightrightarrows W_{n,\gamma}\to\cW_{n,\gamma}\to\cX_{n}
\]
such that $\cW_{n,\gamma}=[W_{n,\gamma}/H_{\gamma}]$. Since $\cW_{n,\gamma}\to\cX_{n}$
is etale and stabilizer preserving, if $V_{n,\gamma}$ denotes the
coarse moduli space of $\cW_{n,\gamma}$, then $V_{n,\gamma}\to X_{n}$
is etale and $\cW_{n,\gamma}\cong\cX_{n}\times_{X_{n}}V_{n,\gamma}$.
Moreover we have natural morphisms $V_{n,\gamma}\to V_{n+1,\gamma}$.
Let $W_{\gamma}:=\varinjlim W_{n,\gamma}$, $\cW_{\gamma}:=\varinjlim\cW_{n,\gamma}$
and $V_{\gamma}:=\varinjlim V_{n,\gamma}$. We have $\cW_{\gamma}\cong[W_{\gamma}/H_{\gamma}]\cong\cX\times_{X}V_{\gamma}$.
The $W_{\gamma}$ are formal affine schemes of finite type over $\Df$.

Morphisms $V_{n,\gamma}\to X_{n}$ are etale. The diagram
\[
\xymatrix{V_{n,\gamma}\ar[r]\ar[d] & V_{n+1,\gamma}\ar[d]\\
X_{n}\ar[r] & X_{n+1}
}
\]
is Cartesian. Indeed, since horizontal arrows are universal homeomorphisms
and vertical arrows are etale, the morphism $V_{n}\to V_{n+1,\gamma}\times_{X_{n+1}}X_{n}$
is etale and a universal homeomorphism. Thus it is an isomorphism.
From \cite[A.3]{Tonini:2017qr}, $V_{\gamma}\to X$ is representable
and etale. From \ref{lem:coarse-flat-base-change}, $V_{\gamma}$
is the coarse moduli space of $\cW_{\gamma}$. From \ref{lem:coarse-inv-sub},
$V_{\gamma}$ are formal affine schemes of finite type over $\Df$.

It remains to show that $X$ is a formal algebraic space. For $s\ge n$,
let $Y_{n,s}\subset X_{s}$ be the scheme-theoretic image of $X_{n}$.
From \ref{lem:Mittag-Leffler}, there exists a strictly increasing
function $s\colon\NN\to\NN$ such that for any $n$ and for any $s'\ge s(n)$,
the morphism $Y_{n,s(n)}\to Y_{n,s'}$ is an isomorphism. The natural
morphism $Y_{n,s(n)}\to Y_{n+1,s(n+1)}$ is a thickening and we have
$Y_{n+1,s(n+1)}\times_{\D_{n+1}}\D_{n}\cong Y_{n,s(n)}$. Morphisms
$X_{n}\to Y_{n,s(n)}$ gives a morphism $X\to\varinjlim Y_{n,s(n)}$
and morphisms $Y_{n,s(n)}\to X_{s(n)}$ gives a morphism $\varinjlim Y_{n,s(n)}\to X$.
These are inverses to each other. Thus $X\cong\varinjlim Y_{n,s(n)}$,
which is a formal algebraic space. From \ref{lem:formal-finite-type},
it is of finite type over $\Df$.
\end{proof}

\section{Galoisian group schemes\label{sec:Galoisian-group-schemes}}

In this section, we study moduli stacks of finite group schemes of
some class. In what follows, we often identify a finite group $G$
with the constant group scheme $\coprod_{g\in G}\Spec F$ over the
spectrum of a field $F$.
\begin{defn}
\label{def:Galoisian}A finite group $G$ is said to be \emph{Galoisian}
if there exist an algebraically closed field $K$ and a Galois extension
$L/K\llparenthesis t\rrparenthesis$ such that $G\cong\Gal(L/K\llparenthesis t\rrparenthesis)$.
We denote by $\GG$ a set of representatives of isomorphism classes
of Galoisian groups.

A finite etale group scheme $G$ over a scheme $S$ is said to be
\emph{Galoisian }if for every geometric point $s\colon\Spec K\to S$,
the fiber $G_{s}=G\times_{S}\Spec K$ is a Galoisian finite group.
For a Galoisian group $G$, a finite etale group scheme $G'\to S$
is said to be $G$\emph{-Galoisian }if every geometric fiber $G'_{s}$
is isomorphic to $G$, equivalently, if it is a twisted form of the
constant group scheme $G\times S\to S$.
\end{defn}

When $p=0$, Galoisian groups are nothing but cyclic groups. When
$p>0$, a Galoisian group is isomorphic to the semidirect product
$H\rtimes C$ of a $p$-group $H$ and a tame cyclic group $C$. For
these facts, see \cite[pages 67 and 68]{MR554237}. In particular,
a Galoisian group has a unique $p$-Sylow subgroup. Any subgroup of
a Galoisian group as well as a quotient group is again Galoisian.

\begin{defn}
We define a category fibered in groupoids $\cA\to\Aff$ as follows.
An object over $S\in\Aff$ is a Galoisian group scheme $H\to S$.
A morphism $(H\to S)\to(H'\to S')$ is the pair of a morphism $S\to S'$
and an isomorphism of group schemes $H\to H'\times_{S'}S$. We call
$\cA$ the \emph{moduli stack of Galoisian group schemes. }For an
isomorphism class $[G]$ of a Galoisian group $G$, we define $\cA_{[G]}\subset\cA$
to be the full subcategory of $G$-Galoisian group schemes.
\end{defn}

We have $\cA=\coprod_{G\in\GG}\cA_{[G]}$. There exists the universal
Galoisian group scheme $\cG$ over $\cA$. This is a category fibered
in sets over $\cA$ such that the fiber set over a Galoisian group
scheme $G\to S$ is the set of its sections $G(S)$. The morphism
$\cG\to\cA$ is representable and is finite and etale.
\begin{lem}
\label{lem:A-B(AutG)}For a Galoisian group $G$, we have $\cA_{[G]}\cong\B(\Aut G)$.
\end{lem}

\begin{proof}
We define a functor $\Phi\colon\B(\Aut G)\to\cA_{[G]}$; to an $\Aut G$-torosr
$T\to S$, we associate the quotient $G_{T}/\Aut G$ of $G_{T}:=G\times T$
by the diagonal action, which is a group scheme over $S$ and a twisted
form of the constant group scheme $G_{S}\to S$. Next we define a
functor $\Psi\colon\cA_{[G]}\to\B(\Aut G)$; to a $G$-Galoisian group
scheme $H\to S$, we associate an $\Aut G$-torsor $\ulIso_{S}(G_{S},H)$.
Here $\alpha\in\Aut(G)$ acts by $\alpha\cdot\phi:=\phi\circ\alpha^{-1}$.

Let us show that these functors are quasi-inverses to each other.
For an $\Aut G$-torosr $T\to S$, we have a natural isomorphism 
\[
u_{T\to S}\colon T\to\ulIso_{S}(G_{S},G_{T}/\Aut G),\,t\mapsto(g\mapsto\overline{(g,t)}).
\]
This is $\Aut G$-equivariant. Indeed, for $\alpha\in\Aut G$, the
above functor sends $\alpha\cdot t$ to 
\begin{align*}
(g\mapsto\overline{(g,\alpha\cdot t)}) & =(g\mapsto\overline{(\alpha^{-1}g,t)})\\
 & =\alpha\cdot(g\mapsto\overline{(g,t)}).
\end{align*}
Since any equivariant morphism of torsors is an isomorphism, $u_{T\to S}$
is an isomorphism of $\Aut G$-torsors and define a natural isomorphism
$u\colon\id_{\B(\Aut G)}\to\Psi\circ\Phi$.

To a $G$-Galoisian group scheme $H\to S$, we have a morphism
\[
G_{S}\times\ulIso_{S}(G_{S},H)\to H,\,(g,\phi)\mapsto\phi(g).
\]
This is $\Aut G$-invariant. Indeed, for $\alpha\in\Aut G$, 
\[
\alpha(g,\phi)=(\alpha(g),\phi\circ\alpha^{-1})\mapsto\phi\circ\alpha^{-1}(\alpha(g))=\phi(g).
\]
Thus we obtain a morphism 
\[
v_{H\to S}\colon(G_{S}\times\ulIso_{S}(G_{S},H))/\Aut G\to H.
\]
Etale locally on $S$, this is identified with the isomorphism 
\[
G_{S}\times\ulIso_{S}(G_{S,}G_{S})/\Aut G\to G_{S}.
\]
Therefore $v_{H\to S}$ is an isomorphism. We obtain a natural isomorphism
$v\colon\Phi\circ\Psi\to\id_{\cA_{[G]}}$.
\end{proof}
\begin{cor}
For a Galoisian group $G$, $\cA_{[G]}$ and $\cG_{[G]}:=\cG\times_{\cA}\cA_{[G]}$
are DM stacks etale, proper and quasi-finite over $k$.
\end{cor}

\begin{proof}
From \ref{lem:A-B(AutG)}, $\cA_{[G]}$ is a DM stack etale, proper
and quasi-finite over $k$. The morphism $\cG\to\cA$ is representable,
etale and finite. Therefore $\cG_{[G]}$ is also a DM stack etale,
proper and quasi-finite over $k$.
\end{proof}

\section{Frobenius morphisms and ind-perfection\label{sec:Frobenius-morphisms}}

In this section, we assume $p>0$ and discuss Frobenius morphisms
and ind-perfection of DM stacks.

For an $\FF_{p}$-algebra $A$, the $p$-th power map $A\to A$, $f\mapsto f^{p}$
gives a scheme endomorphism $F_{S}\colon S\to S$ of the affine scheme
$S=\Spec A$. This is called the \emph{absolute Frobenius morphism
}of $S$. Let $\cX$ be a stack over $\FF_{p}$. For each $S\in\Aff_{\FF_{p}}$,
we have the pullback functor $F_{S}^{*}\colon\cX(S)\to\cX(S)$ of
groupoids. These functor give an endomorphism of $\cX$ over $\FF_{p}$,
which we call the \emph{absolute Frobenius morphism of $\cX$ }and
denote by $F_{\cX}$. Let $S$ be an affine scheme over $\FF_{p}$
and $\cX$ a stack over $S$ with the structure morphism $\pi\colon\cX\to S$.
We define the \emph{relative Frobenius morphism }$F_{\cX/S}$ to be
the morphism 
\[
(F_{\cX},\pi)\colon\cX\to\cX^{(1)}:=\cX\times_{\pi,S,F_{S}}S.
\]
We say that $\cX$ is \emph{prefect over $S$ }if $F_{\cX/S}$ is
an isomorphism.

For each $n\in\NN$, we also define the \emph{$n$-iterated relative
Frobenius morphism} $F_{\cX/S}^{n}$ of $\cX$ by 
\[
(F_{\cX}^{n},\pi)\colon\cX\to\cX^{(n)}:=\cX\times_{\pi,S,F_{S}^{n}}S.
\]
We define the \emph{ind-perfection }$\cX^{\iper}$ of $\cX$ to be
the limit $\varinjlim\cX^{(n)}$ as a category. This has a natural
structure of a stack over $S$ \cite[Appendix A]{Tonini:2017qr},
which is clearly perfect over $S$. We have the the expected universality:
\begin{lem}
For a morphism $f\colon\cY\to\cX$ of stacks over $S$ with $\cX$
perfect over $S$, there exists an $S$-morphism $\cY^{\iper}\to\cX$
unique up to 2-isomorphisms such that the composition $\cY\to\cY^{\iper}\to\cX$
is isomorphic to $f$.
\end{lem}

\begin{proof}
For $n\in\NN$, we fix a quasi-inverse $\cX^{(n+1)}\to\cX^{(n)}$
of $\cX^{(n)}\to\cX^{(n+1)}$ so that we get the composition 
\[
\cX^{(n)}\to\cX^{(n-1)}\to\cdots\to\cX^{(0)}=\cX.
\]
To $T\in\Aff_{S}$ and an object $y\in\cY^{\iper}(T)$ represented
by $y'\in\cY^{(n)}(T)$, we assign its image by $\cY^{(n)}(T)\xrightarrow{f}\cX^{(n)}(T)\xrightarrow{\sim}\cX(T)$.
The induced morphism $\cY^{\iper}\to\cX$ is the desired morphism.
\end{proof}
\begin{lem}
\label{lem:limit-iper}Let 
\[
\cX_{0}\to\cX_{1}\to\cdots
\]
be an inductive system of DM stacks of finite type such that every
transition map is representable and a finite universal homeomorphism.
Suppose that the limit $\cY:=\varinjlim\cX_{i}$ is perfect over $k$.
Then the natural morphism $\cX_{0}^{\iper}\to\cY$ is an isomorphism.
\end{lem}

\begin{proof}
Consider the inductive system $\cX_{i}^{(n)}$, $(i,n)\in\NN\times\NN$.
We have
\[
\varinjlim_{i}\cX_{i}^{\iper}\cong\varinjlim_{i}\varinjlim_{n}\cX_{i}^{(n)}\cong\varinjlim_{n}\varinjlim_{i}\cX_{i}^{(n)}\cong\varinjlim_{n}\cY^{(n)}.
\]
Since $\cY$ is perfect, $\cY^{(n)}\to\cY^{(n+1)}$ are all isomorphisms.
We have $\varinjlim_{n}\cY^{(n)}\cong\cY$. We claim that morphisms
$\cX_{i}^{\iper}\to\cX_{i+1}^{\iper}$ are isomorphisms. Indeed, for
each $i$, there exists $m\in\NN$ such that the morphism $\cX_{i}\to\cX_{i}^{(m)}$
factors through $\cX_{i+1}$. We obtain morphisms
\[
\cX_{i}\to\cX_{i+1}\to\cX_{i}^{(m)}\to\cX_{i+1}^{(m)}.
\]
Taking ind-perfection, we obtain the following commutative diagram.
\[
\xymatrix{\cX_{i}^{\iper}\ar[r]\ar@(ur,ul)[rr]^{\id} & \cX_{i+1}^{\iper}\ar[r]\ar@(dr,dl)[rr]_{\id} & \cX_{i}^{\iper}\ar[r] & \cX_{i+1}^{\iper}}
\]
This shows the claim. From the claim, we have $\varinjlim_{i}\cX_{i}^{\iper}\cong\cX_{0}^{\iper}$
and the lemma follows.
\end{proof}

\section{Formal torsors\label{sec:Formal-torsors}}

Based upon results in \cite{Tonini:2017qr}, in this section, we construct
moduli stacks of torsors over the punctured formal disk and a ``universal''
family of twisted formal disks.

\subsection{Moduli stack $\Delta$}

For an affine scheme $S=\Spec R$, we denote by $\D_{S}^{*}$ or $\D_{R}^{*}$
the affine scheme $\Spec R\tpars$. Here $R\tpars$ denotes the localization
$R\tbrats_{t}$ of $R\tbrats$ by $t$ (rather than the total quotient
ring of $R\tbrats$). Namely this is the ring of Laurent power series
with coefficients in $R$.
\begin{defn}
For a finite etale group scheme $G\to S$ over an affine scheme $S$,
we define a category fibered in groupoids $\Delta_{G}\to\Aff_{S}$
as follows; for an affine $S$-scheme $T$, the fiber category $\Delta_{G}(T)$
is the category of $G$-torsors over $\D_{T}^{*}$. For a finite group
$G$, we define $\Delta_{G}\to\Aff_{k}$ by regarding $G$ as a constant
group scheme over $k$.
\end{defn}

\begin{defn}
We define a fibered category $\Delta\to\cA$ as follows; for a Galoisian
group scheme $G\to S$, the fiber category $\Delta(G\to S)$ is $\Delta_{G}(S)$.
(Thus the whole category $\Delta$ is the category of pairs $(G\to S,P\to\D_{S}^{*})$
of a Galoisian group scheme $G\to S$ and a $G$-torsor $P\to\D_{S}^{*}$.)
For an isomorphism class $[G]$ of Galoisian groups, we define $\Delta_{[G]}:=\Delta\times_{\cA}\cA_{[G]}$.
\end{defn}

For a morphism $S\to\cA$ corresponding to a Galoisian group scheme
$G\to S$, we have $\Delta\times_{\cA}S\cong\Delta_{G}$. Note that
as a fibered category over $\Aff$, the scheme $S$ is $\Aff_{S}$
and the projection $\Delta\times_{\cA}S\to S=\Aff_{S}$ corresponds
to the underlying functor $\Delta_{G}\to\Aff_{S}$.

\subsection{Uniformization}
\begin{defn}
For a finite etale group scheme $G\to S$, a $G$-torsor $P\to\D_{S}^{*}$
and a geometric point $x\colon\Spec F\to S$, we call the induced
$G_{F}$-torsor 
\[
P_{x}:=P\times_{\D_{S}^{*}}\D_{F}^{*}\to\D_{F}^{*}
\]
the \emph{fiber over $x$}. We say that $P\to\D_{S}^{*}$ or $P$
is \emph{fiberwise connected }or \emph{has connected fibers }if for
every geometric point $x$ of $S$, $P_{x}$ is connected.
\end{defn}

\begin{rem}
The fiber defined above is not the fiber of $P\to S$ over $x$ in
the usual sense; $P_{x}$ is not the fiber product $P\times_{S}\Spec F$,
but the ``complete fiber product.''
\end{rem}

The following notions are taken from \cite{Formal-Torsors-II}, though
we restrict ourselves to fiberwise connected torsors.
\begin{defn}
\label{def:uniformizable}Let $G\to\Spec R$ be a Galoisian group
scheme and let $P=\Spec A\to\D_{R}^{*}$ be a fiberwise connected
$G$-torsor. We say that this is \emph{uniformizable} if there exists
an isomorphism $A\cong R\llparenthesis s\rrparenthesis$ such that
the map $R\llparenthesis t\rrparenthesis\to A\cong R\llparenthesis s\rrparenthesis$
sends $t$ to a series of the form $s^{\sharp G}u$ for $u\in R\llbracket s\rrbracket^{*}$.
We call such an isomorphism $A\cong R\llparenthesis s\rrparenthesis$
a \emph{uniformization} of $P$ or $A$.
\end{defn}

When this is the case, the image of $s$ in $A$ gives a uniformizer
of the discrete valuation field $A\otimes_{R\llparenthesis t\rrparenthesis}K\llparenthesis t\rrparenthesis$
for any point $\Spec K\to\Spec R$.
\begin{lem}
\label{lem:indep-O}With notation as above, if $R$ is reduced and
$A$ is uniformizable, then the image of $R\llbracket s\rrbracket\to A$
is independent of the choice of the uniformization $A\cong R\llparenthesis s\rrparenthesis$.
\end{lem}

\begin{proof}
Let $A\cong R\llparenthesis r\rrparenthesis$ be another isomorphism
as in \ref{def:uniformizable}. We regard $R\llbracket s\rrbracket$
and $R\llbracket r\rrbracket$ as subrings of $A$ and write $r=su$,
$u\in A=R\llparenthesis s\rrparenthesis$. For any point $\Spec K\to\Spec R$,
the image of $u$ in $K\llparenthesis s\rrparenthesis=A\otimes_{R\llparenthesis t\rrparenthesis}K\llparenthesis t\rrparenthesis$
is a unit of $K\llbracket s\rrbracket$. This shows that $u\in R\llbracket s\rrbracket^{*}$,
since $R$ is reduced. Here we used the fact that for a reduced ring
$R$, an element $u\in R$ is zero (resp.\ a unit) if and only if
for every point $\Spec K\to\Spec R$, the image of $u$ in $K$ is
zero (resp.\ a unit). It follows that $r\in R\llbracket s\rrbracket$.
Similarly $s\in R\llbracket r\rrbracket$. Hence $R\llbracket s\rrbracket=R\llbracket r\rrbracket$.
\end{proof}
\begin{defn}
\label{def:integral-ring}With notation as above, when $R$ is reduced
and $A$ is uniformizable, we define a subring $O_{A}\subset A$ to
be the image of $R\llbracket s\rrbracket$ by a uniformization $A\cong R\llparenthesis s\rrparenthesis$.
We call it the \emph{integral ring }of $A$ or of the corresponding
torsor $\Spec A\to\D_{R}^{*}$.
\end{defn}

\begin{lem}
\label{lem:stable-G}Keeping the assumption, the integral ring $O_{A}$
is stable under the $G$-action on $A$.
\end{lem}

\begin{proof}
Let $A\cong R\llparenthesis s\rrparenthesis$ be an isomorphism as
above. For $g\in G$ and for every point $\Spec K\to\Spec R$, the
image of $g(s)$ in $K\llparenthesis s\rrparenthesis$ is a uniformizer.
By the same argument as in the proof of \ref{lem:indep-O}, $g(s)$
is of the form $su$, $u\in R\llbracket s\rrbracket^{*}$. In particular,
$g(s)\in O_{A}$, which implies the lemma.
\end{proof}

\subsection{Ramification\label{subsec:Ramification}}

In this subsection, we assume $p>0$.

Let $G$ be a Galoisian group and let $N\subset H\subset G$ be a
subgroups such that $N$ is a normal subgroup of $H$. For a fiberwise
connected $G$-torsor $P\to\D_{S}^{*}$ and a geometric point $x\colon\Spec F\to S$,
we have an $H/N$-torsor $(P/N)_{x}\to(P/H)_{x}$ with connected source
and target. This corresponds to a Galois extension $L/K$ of complete
discrete valuation fields with the common residue field $F$ and its
Galois group $\Gal(L/K)$ is identified with $H/N$. The Galois group
is equipped with the (lower numbering) ramification filtration 
\[
\Gal(L/K)=\Gal(L/K)_{0}\supset\Gal(L/K)_{1}\supset\cdots
\]
such that $\Gal(L/K)_{i}=1$ for $i\gg0$ (for instance, see \cite[Ch. IV]{MR554237}).
When $\Gal(L/K)\cong H/N$ is a cyclic group of order $p$, there
exists a unique $i$ such that $\Gal(L/K)_{i}\ne\Gal(L/K)_{i+1}$.
This $i$ is called the \emph{ramification jump }of the $H/N$-torsor
$(P/N)_{x}\to(P/H)_{x}$. If we write $(P/H)_{x}=\Spec F\llparenthesis s\rrparenthesis$,
then from the Artin-Schreier theory, we can write 
\[
(P/N)_{x}=\Spec F\llparenthesis s\rrparenthesis[x]/(x^{p}-x-f)
\]
for $f\in F\llparenthesis s\rrparenthesis$ of the form
\[
f=\sum_{j>0,\,p\nmid j}f_{j}s^{-j}\quad(f_{j}\in F).
\]
The action of $G=\langle g\rangle$ is given by $g\colon x\mapsto x+1$.
This $f$ is uniquely determined and the ramification jump is equal
to $\max\{j\mid f_{j}\ne0\}$. In particular, the ramification jump
is prime to $p$. If two geometric points $x,x'$ have the same image
in $S$, then the ramification jumps of $(P/N)_{x}\to(P/H)_{x}$ and
$(P/N)_{x'}\to(P/H)_{x'}$ are equal. Thus we obtain a function $|S|\to\NN$.
\begin{defn}
\label{def:ramif-datum}Let $G$ be a Galoisian group. A \emph{ramification
datum }for $G$ is a function 
\[
\br\colon(H,N)\mapsto\br(H,N)\in\NN
\]
on pairs $(H,N)$ of a subgroup $H\subset G$ and a normal subgroup
$N\lhd H$ with $\sharp(H/N)=p$. We denote the set of ramification
data for $G$ by $\Ram(G)$.
\end{defn}

Let $\br$ be a ramification datum for $G$. For a subgroup $G'\subset G$,
restricting to pairs $(H,N)$ with $H\subset G'$, we obtain a ramification
datum on $G'$. When $G'$ is a normal subgroup, we also have the
induced ramification datum $\bs$ for $G/G'$; for $N\lhd H\subset G/G'$,
if $\tilde{N},\tilde{H}\subset G$ are their preimages, then $\bs(H,N):=\br(\tilde{H},\tilde{N})$.

\begin{defn}
When $G$ is a Galoisian group and $P\to\D_{F}^{*}$ is a connected
$G$-torsor with $F$ an algebraically closed field, we define the
ramification datum $\br_{P}$ so that $\br_{P}(H,N)$ is the ramification
jump of $P/N\to P/H$. We say that a fiberwise connected $G$-torsor
$P\to\D_{S}^{*}$ has \emph{constant ramification} if the map 
\[
|S|\to\Ram(G),\,x\mapsto\br_{P_{x}}
\]
is constant. When $G\to S$ is a Galoisian group scheme, we say that
a $G$-torsor $P\to\D_{S}^{*}$ has \emph{locally constant ramification
}if there exists an etale cover $(S_{i}\to S)_{i\in I}$ such that
for each $i$, $G\times_{S}S_{i}$ is constant and $P_{S_{i}}\to\D_{S_{i}}^{*}$
has constant ramification.
\end{defn}

For $n\in\NN$ and $S\in\Aff$, let $S^{(-n)}$ be the scheme $S$
regarded as a $k$-scheme by the morphism 
\[
S\to\Spec k\xrightarrow{(F_{\Spec k})^{n}}\Spec k.
\]
The morphism 
\[
S^{(-n)}=S\xrightarrow{(F_{S})^{n}}S
\]
is a $k$-morphism with respect to the $k$-scheme structure $S^{(-n)}$.
When $k$ is perfect, we have canonical $k$-isomorphisms $(S^{(-n)})^{(n)}\cong(S^{(n)})^{(-n)}\cong S$.

\begin{prop}
\label{prop:exist-model}Let $G$ be a $p$-group.
\begin{enumerate}
\item Let $\br\in\Ram(G)$ and let $P\to\D_{S}^{*}$ be a $G$-torsor. Let
\[
C:=\{[x]\in|S|\mid P_{x}\text{ is connected and }\br_{P_{x}}=\br\}.
\]
Then $C$ is a locally closed subset of $|S|$.
\item Let $P\to\D_{S}^{*}$ be a fiberwise connected $G$-torsor having
constant ramification. Suppose that $S$ is reduced. For $n\gg0$,
the induced $G$-torsor $P_{S^{(-n)}}\to\D_{S^{(-n)}}^{*}$ is uniformizable.
\end{enumerate}
\end{prop}

\begin{proof}
We first prove both assertions in the case $G=\ZZ/p$. Write $S=\Spec R$.
There exists a Laurent polynomial of the form
\[
f=f_{j}t^{-jp^{m}}+f_{j-1}t^{-(j-1)p^{m}}+\cdots+f_{1}t^{-p^{m}}+f_{0}\quad(f_{i}\in R)
\]
such that for $i>0$ with $p\mid i$, we have $f_{i}=0$ and $P=\Spec R\llparenthesis t\rrparenthesis[x]/(x^{p}-x+f)$
(see \cite[Th. 4.13]{Tonini:2017qr} and the part of its proof concerning
the essential surjectivity). Consider the iterated Frobenius $S^{(-m)}=\Spec R^{(-m)}\to S=\Spec R$
with the corresponding map $\phi\colon R\to R^{(-m)}$. For each $i$,
there exists $g_{i}\in R^{(-m)}$ such that $(g_{i})^{p^{m}}=\phi(f_{i})$.
From \cite[Lem. 4.9]{Tonini:2017qr}, the torsor $P_{S^{(-m)}}$ is
then isomorphic to the torsor associated to 
\[
g=g_{j}t^{-j}+g_{j-1}t^{-(j-1)}+\cdots+g_{0}.
\]
For $j_{0}>0$, the locus in $S^{(-m)}$ of connected fibers and ramification
jump $j_{0}$ is 
\[
\left(\bigcap_{j>j_{0}}V(g_{j})\right)\setminus V(g_{j_{0}}).
\]
This shows (1) in the case $G=\ZZ/p$. Note that the locus of non-connected
fibers is $\bigcap_{j>0}V(g_{j})$. To show (2) in this case, suppose
that $S$ is reduced and that $P$ has constant ramification jump
$j_{0}$. Then for every $j>j_{0}$, $V(g_{j})=S^{(-m)}$, thus $g_{j}=0$.
Since $V(g_{j_{0}})=\emptyset$, $g_{j_{0}}$ is invertible. Since
$j_{0}$ is prime to $p$, there exist positive integers $a,b$ such
that $-aj+bp=1$. Let us write $S^{(-m)}=\Spec R'$ and $P_{S^{(-m)}}=\Spec A=\Spec R'\llparenthesis t\rrparenthesis[y]/(y^{p}-y+g)$
and let $s:=y^{a}t^{b}\in A$. For each point $r\colon\Spec K\to S^{(-m)}$,
the elements $s_{r},y_{r}\in A_{r}$ induced from $s,y$ respectively
have valuation $1$ and $-j_{0}$ with respect to the normalized valuation
of $A_{r}$. We have 
\[
R'\llparenthesis t\rrparenthesis[y]/(y^{p}-y+g)=\bigoplus_{l=0}^{p-1}R'\llparenthesis t\rrparenthesis y^{l}
\]
and each $s^{i}$ is uniquely written as $s^{i}=\sum_{l=0}^{p-1}t_{i,l}y^{l}$,
$t_{i,l}\in R'\llparenthesis t\rrparenthesis$. For each point $r$,
$(t_{i,l})_{r}$ has valuation $\ge i-lj_{0}$. Since $R'$ is reduced,
we have $(t_{i,l})\in t^{i-lj_{0}}\cdot R'\llbracket t\rrbracket$.
It follows that any power series $\sum_{i\ge0}a_{i}s^{i}$, $a_{i}\in R'$
in the variable $s$ converges. It follows that there exists a unique
map $\phi\colon R'\llbracket u\rrbracket\to A$ sending $u$ to $s$.
Thus we may think of the image of $\phi$ as the power series ring
$R\llbracket s\rrbracket$ with the variable $s$. Clearly $R'\spars\cong R'\llbracket s\rrbracket\otimes_{R'\llbracket t\rrbracket}R'\llparenthesis t\rrparenthesis\cong A$.
The image of $t$ in $A$ is written as $us^{p}$, $u\in R\llparenthesis s\rrparenthesis$
such that at each point $r\colon\Spec K\to S$, the image of $u$,
$u_{r}\in K\llparenthesis s\rrparenthesis$ is a unit of $K\llbracket s\rrbracket$.
This means that $u\in R\llbracket s\rrbracket^{*}$. Thus the isomorphism
$R'\spars\cong A$ is a uniformization. Assertion (2) holds when $G=\ZZ/p$.

Next we prove the general case by induction. Let $n\in\NN$. We assume
that both assertions hold if $\sharp G<n$ and will prove the case
$\sharp G=n$. Let $P\to\D_{S}^{*}$ be a $G$-torsor. For every normal
subgroup $N\lhd G$ of order $p$, the locus where the $G/N$-torsor
$P/N\to\D_{S}^{*}$ has connected fibers and the ramification datum
induced from $\br$ is a locally closed subset, say $C\subset S$.
We give it the reduced structure. For any affine open $U\subset C$
and for $m\gg0$, $(P/N)_{U^{(-m)}}$ is uniformizable. From the case
$G=\ZZ/p$, the locus where the $N$-torsor $P_{U^{(-m)}}\to(P/N)_{U^{(-m)}}=\D_{S^{(-m)}}^{*}$
has connected fibers and the ramification jump induced from $\br$
is locally closed. This shows assertion (1). If $P$ has connected
fibers and constant ramification and if $S$ is reduced, then from
the assumption of induction, we have $(P/N)_{S^{(-m)}}$ is uniformizable
for $m\gg0$. The $N$-torsor $P_{S^{(-m)}}\to(P/N)_{S^{(-m)}}=\D_{S^{(-m)}}^{*}$
also has constant ramification. Thus, from assertion (2) for the case
of $\ZZ/p$, for $n\gg m$, the $N$-torsor $P_{S^{(-n)}}\to(P/N)_{S^{(-n)}}=\D_{S^{(-n)}}^{*}$
is uniformizable, which implies assertion (2) for a general $p$-group.
\end{proof}
\begin{defn}
\label{def:Lambda}We define $\Lambda$ (resp.\ $\Lambda_{[G]}$,
$\Lambda_{G}$) to be the full substack of $\Delta$ (resp.\ $\Delta_{[G]}$,
$\Delta_{G}$) of fiberwise connected torsors with locally constant
ramification. When $G$ is constant and $\br$ is a ramification datum
for $G$, we define $\Lambda_{G}^{\br}$ to be the full subcategory
of $\Lambda_{G}$ of torsors with ramification datum $\br$.
\end{defn}

For a Galoisian group $G$, we have $\Lambda_{[G]}=\Lambda\times_{\cA}\cA_{[G]}$
and $\Lambda_{G}=\coprod_{\br\in\Ram(G)}\Lambda_{G}^{\br}$. For a
Galoisian group scheme $G\to S$ corresponding to $S\to\cA$, we have
$\Lambda_{G}=\Lambda\times_{\cA}S$.
\begin{defn}
For a ramification datum $\br$ for a Galoisian group $G$, let $[\br]$
be its $\Aut(G)$-orbit. We define $\Lambda_{[G]}^{[\br]}$ to be
the essential image of the functor $\coprod_{\bs\in[\br]}\Lambda_{G}^{\bs}\to\Lambda_{[G]}$.
\end{defn}

For $\Spec k\to\cA_{[G]}$ corresponding to the constant group $G$,
we have 
\[
\Lambda_{[G]}^{[\br]}\times_{\cA_{[G]}}\Spec k\cong\coprod_{\bs\in[\br]}\Lambda_{G}^{\bs}.
\]

\begin{lem}
\label{lem:Lambda-perfect}For a Galoisian group $G$ and a ramification
datum $\br$, the stacks $\Delta_{G}$ and $\Lambda_{G}^{\br}$ are
perfect over $k$.
\end{lem}

\begin{proof}
If we prove the lemma when $k=\FF_{p}$, then the general case follows
by the base change along $\Spec k\to\Spec\FF_{p}$. Thus we may suppose
that $k$ is perfect. For $S\in\Aff$, the Frobenius morphism 
\[
\Delta_{G}(S)\to\Delta_{G}(S^{(-1)})=\Delta_{G}^{(1)}(S)
\]
is an equivalence \cite[Prop. 4.6]{Tonini:2017qr}. By this equivalence,
the property of having connected fibers is preserved. The ramification
data are also preserved. These prove the lemma. 
\end{proof}
\begin{lem}
\label{lem:Lambda-r}For a Galoisian group $G$ and a ramification
datum $\br$ for $G$, the stack $\Lambda_{G}^{\br}$ is isomorphic
to the ind-perfection $\cX^{\iper}$ of a reduced DM stack $\cX$
of finite type. The same is true for $\Lambda_{[G]}^{[\br]}$.
\end{lem}

\begin{proof}
Again we may suppose that $k$ is perfect. From \cite{Tonini:2017qr},
we can write $\Delta_{G}=\varinjlim\cY_{i}$ as the limit of DM stacks
of finite type such that the transition morphisms $\cY_{i}\to\cY_{i+1}$
are representable, finite and universally injective. In particular,
we have $|\Delta_{G}|=\bigcup_{i}|\cY_{i}|$. Since torsors with the
same ramification datum $\br$ have the same discriminant, we have
$|\Lambda_{G}^{\br}|\subset|\cY_{i_{0}}|$ for some $i_{0}$ \cite{Formal-Torsors-II}.
(For each $n$, we can construct a family $P\to\D_{S}^{*}$ of totally
ramified extensions of $k\tpars$ with $S$ of finite type such that
every totally ramified extension of discriminant exponent at most
$n$. This shows that $\Lambda_{G}^{\br}$ is of finite type.) From
\ref{prop:exist-model}, $|\Lambda_{G}^{\br}|$ is a locally closed
subset of $|\cY_{i_{0}}|$. For $i\ge i_{0}$, let $\cX_{i}\subset\cY_{i}$
be the reduced substack such that $|\cX_{i}|=|\Lambda_{G}^{\br}|$.

Since $\Delta_{G}$ is perfect, we also have $\Delta_{G}\cong\varinjlim_{i,n}\cY_{i}^{(n)}$.
We claim that 
\[
\Lambda_{G}^{\br}\cong\varinjlim_{i,n}\cX_{i}^{(n)}.
\]
To show this, we first observe that $\Lambda_{G}^{\br}$ is a full
subcategory of $\Delta_{G}$ and $\varinjlim_{i,n}\cX_{i}^{(n)}$
is a full subcategory of $\varinjlim_{i,n}\cY_{i}^{(n)}$. We need
to show that their objects correspond through the isomorphism $\Delta_{G}\cong\varinjlim_{i,n}\cY_{i}^{(n)}$.
It is clear that every object of $\varinjlim_{i,n}\cX_{i}^{(n)}$
maps to an object of $\Lambda_{G}^{\br}$. It remains to show that
every object of $\Lambda_{G}^{\br}$ comes from an object of $\cX_{i}^{(n)}$
for some $i,n$. Let $P\to\D_{S}^{*}$ be an object of $\Lambda_{G}^{\br}$
over an affine scheme $S$. For some $i\ge i_{0}$, the morphism $S\to\Lambda_{G}^{\br}$
factors through $\cY_{i}$. The image of $|S|\to|\cY_{i}|$ is contained
in $|\cX_{i}|$. Let $\cU_{i}\subset\cY_{i}$ be an open substack
such that $\cX_{i}$ is a closed substack of it. Let $\cI\subset\cO_{\cU_{i}}$
be the defining ideal of $\cX_{i}$. The morphism $S_{\red}\to S\to\cU_{i}$
factors through $\cX_{i}$, which means that $\cI$ is pulled back
to the zero ideal sheaf by this morphism. Since $\cU_{i}$ is of finite
type and $\cI$ is locally finitely generated, for $n\gg0$, the ideal
$\cI$ is pulled back to the zero ideal sheaf also by the morphism
$S^{(-n)}\to S\to\cU_{i}$. This means that the morphism $S^{(-n)}\to\cU_{i}$
factors through $\cX_{i}$, equivalently, $S\to\cU_{i}^{(n)}$ factors
through $\cX_{i}^{(n)}$. We have proved the claim.

From \ref{lem:Lambda-perfect} and \ref{lem:limit-iper}, if we put
$\cX:=\cX_{i_{0}}$, we have $\Lambda_{G}^{\br}\cong\cX^{\iper}$.
We have proved the first assertion. Since $\Lambda_{[G]}^{[\br]}\times_{\cA_{[G]}}\Spec k\cong\coprod_{\bs\in[\br]}\Lambda_{G}^{\bs}$,
from \cite[Lem. 3.5]{Tonini:2017qr}, $\Lambda_{[G]}^{[\br]}$ is
also the inductive limit of DM stacks of finite type such that the
transition morphisms $\cY_{i}\to\cY_{i+1}$ are representable, finite
and universally injective. The Frobenius morphism of $\Lambda_{[G]}^{[\br]}$
becomes the one of $\coprod_{\bs\in[\br]}\Lambda_{G}^{\bs}$ by the
base change with $\Spec k\to\cA_{[G]}$. Therefore $\Lambda_{[G]}^{[\br]}$
is also perfect over $k$. The second assertion similarly follows.
\end{proof}

\subsection{\label{subsec:Lambda-tame}The stack $\Lambda$ in characteristic
zero or in the tame case}

W return to the situation of arbitrary characteristic.

Let $C_{l}$ be the cyclic group of order $l$ such that if $p>0$,
then $p\nmid l$. Every tame Galoisian group is of this form. In characteristic
zero, we define $\Lambda$, $\Lambda_{[G]}$ and $\Lambda_{G}$ as
in \ref{def:Lambda} removing the condition of locally constant ramification;
they are simply be the substacks of $\Delta$, $\Delta_{[G]}$ and
$\Delta_{G}$ respectively of fiberwise connected torsors. For the
morphism $\mu_{l}\colon\Spec k\to\cA_{[C_{l}]}$, we have
\[
\Lambda_{[C_{l}]}\times_{\cA_{[C_{l}]},\mu_{l}}\Spec k\cong\Lambda_{\mu_{l}}\cong\coprod_{p\in(\ZZ/l\ZZ)^{*}}\B\mu_{l}.
\]
Here the right isomorphism follows from \cite[Th. B]{Tonini:2017qr}.
In particular, $\Lambda_{[C_{l}]}$ is a DM stack etale over $k$.
The morphism 
\[
\Lambda_{\mu_{l}}=\Lambda_{[C_{l}]}\times_{\cA_{[C_{l}]}}\Spec k\to\Lambda_{[C_{l}]}
\]
is a torsor under the group $\Aut(C_{l})=(\ZZ/l\ZZ)^{*}$. The $(\ZZ/l\ZZ)^{*}$-action
permutes the components of $\Lambda_{\mu_{l}}$ freely and transitively.
Therefore 
\[
\Lambda_{[C_{l}]}\cong[\Lambda_{\mu_{l}}/(\ZZ/l\ZZ)^{*}]\cong\B\mu_{l}
\]
and, if $\Lambda_{\mu_{l},i}$, $i\in(\ZZ/l\ZZ)^{*}$ denotes the
$i$th component, then the morphism 
\begin{equation}
\Lambda_{\mu_{l},i}\to\Lambda_{[C_{l}]}\label{eq:Delta-mu-1}
\end{equation}
is an isomorphism.

In characteristic zero, we have $\Lambda=\coprod_{l>0}\Lambda_{[C_{l}]}$,
which is a reduced DM stack almost of finite type. In arbitrary characteristic,
we put $\Lambda_{\tame}:=\coprod_{l>0;\,p\nmid l}\Lambda_{[C_{l}]}$.
Note that the component $\Lambda_{[1]}=\Lambda_{[C_{1}]}$, corresponding
to the trivial cover $\Df\to\Df$, is isomorphic to $\Spec k$.

\subsection{Integral models}
\begin{lem}
\label{lem:nice-finite-etale-cover}Suppose that $p>0$ (resp.\,$p=0$).
Let $(G\to S,P\to\D_{S}^{*})\in\Lambda$, where $G\to S$ is a Galoisian
group scheme and $P\to\D_{S}^{*}$ is a $G$-torsor. Suppose that
$S$ is reduced. Then there exists a finite etale cover $\coprod_{i=1}^{n}S_{i}\to S$
such that for each $i$, $G_{S_{i}}$ is constant and $P_{S_{i}}$
has constant ramification and $P_{S_{i}^{(-m)}}$ is uniformizable
for $m\gg0$ (resp.\,$P_{S_{i}}$ is uniformizable).
\end{lem}

\begin{proof}
First consider the case $p>0$. Since there exists a finite etale
cover $U\to S$ such that $G_{U}$ is Zariski-locally constant, we
may suppose that $G$ is constant. Let $H$ be the unique $p$-Sylow
subgroup of $G$. The quotient group $G/H$ is a tame cyclic group,
say of order $l$. Adding the $l$-th roots of unity to $k$, we may
also suppose that the group scheme $\mu_{l}$ is isomorphic to the
constant group $G/H$. From \cite[Th. B]{Tonini:2017qr}, 
\[
\Delta_{G/H}\cong\coprod_{i\in\ZZ/l\ZZ}\B(G/H)
\]
and components for $i\in(\ZZ/l\ZZ)^{*}$ correspond to connected torsors.
For $i\in(\ZZ/l\ZZ)^{*}$, if $i'$ denotes the integer representing
$i$ with $0<i'<l$, then the standard morphism onto the $i$-th component
\[
\Spec k\to\Delta_{G/H}
\]
corresponds to the uniformizable torsor $\Spec k\tpars[x]/(x^{l}-t^{i'})\to\Spec k\tpars$.
Since this morphism is finite and etale, there exists a finite etale
cover $V\to S$ such that the morphism $V\to S\to\Delta_{G/H}$ lifts
to $\coprod_{i\in(\ZZ/l\ZZ)^{*}}\Spec k$. Then $(P/H)_{V}$ is uniformizable.
Thus we may also assume that $P/H$ has uniformization $P/H\cong\D_{S}^{*}$.
Since the $H$-torsor $P\to P/H\cong\D_{S}^{*}$ has locally constant
ramification and $S$ is quasi-compact (recall $S$ is affine), there
exists a stratification $V=\bigsqcup_{i=1}^{l}S_{i}$ into open and
closed subsets such that each $P_{S_{i}}$ has constant ramification.
From \ref{prop:exist-model}, for $m\gg0$, the $H$-torsor $P_{S_{i}^{(-m)}}\to P_{S_{i}^{(-m)}}/H\cong\D_{S_{i}^{(-m)}}^{*}$
is uniformizable and so is the $G$-torsor $P_{S_{i}^{(-m)}}\to\D_{S_{i}^{(-m)}}^{*}$.
We have proved the lemma when $p>0$.

When $p=0$, the above argument for $G/H$ shows the lemma. 
\end{proof}
\begin{lem}
\label{cor:subring-O}Let $(G\to S=\Spec R,P=\Spec A\to\D_{S}^{*})\in\Lambda$.
Suppose that $S$ is reduced and that there exists a finite etale
cover $\coprod_{i\in I}S_{i}\to S$ such that every $P_{S_{i}}$ is
uniformizable. Then there exists a unique $R\llbracket t\rrbracket$-subalgebra
$O\subset A$ such that
\begin{enumerate}
\item $O$ is a finite and flat $R\llbracket t\rrbracket$-module,
\item $O_{t}=A$, where the left side is the localization of $O$ by $t$,
\item \label{enu:pullback}for any $U=\Spec B\to S$ such that $P_{U}$
is uniformizable, the subring $O\otimes_{R\llbracket t\rrbracket}B\llbracket t\rrbracket\subset A\otimes_{R\llbracket t\rrbracket}B\llbracket t\rrbracket$
is the integral ring of $P_{U}$ (see \ref{def:integral-ring}).
\end{enumerate}
\end{lem}

\begin{proof}
Let $S_{ij}=\Spec R_{ij}:=S_{i}\times_{S}S_{j}$ and write $P_{S_{i}}=\Spec A_{i}$
and $P_{S_{ij}}=\Spec A_{ij}$. Let $O_{i}\subset A_{i}$ and $O_{ij}\subset A_{ij}$
be the integral rings. Since $R_{ij}$ is a finitely presented $R_{i}$-module
(recall that an etale morphism is by definition locally of finite
presentation), from \cite[Lem. 2.4]{Tonini:2017qr}, we have $R_{i}\tbrats\otimes_{R_{i}}R_{ij}\cong R_{ij}\tbrats$.
Therefore
\begin{gather*}
O_{ij}=O_{i}\otimes_{R_{i}\tbrats}R_{ij}\tbrats=O_{i}\otimes_{R_{i}}R_{ij},\\
A_{ij}=A_{i}\otimes_{R_{i}\tpars}R_{ij}\tpars=A_{i}\otimes_{R_{i}\tpars}R_{ij}\tpars=A_{i}\otimes_{R_{i}}R_{ij}.
\end{gather*}
By descent, we get a submodule $O\subset A$. Properties (1) and (2)
hold, since they hold after pulled-back to $S_{i}$.

Let $U\to S$ be as in (\ref{enu:pullback}). Consider the finite
etale cover $\coprod_{i}U\times_{S}S_{i}\to U$. The integral ring
of $P_{U}$ and $O\otimes_{R\llbracket t\rrbracket}B\llbracket t\rrbracket$
are the same because they become the same when pulled back to $U\times_{S}S_{i}$.
\end{proof}
\begin{defn}
For a torsor $P\to\D_{S}^{*}$ as in \ref{cor:subring-O}, we call
$\Spf O$ the \emph{integral model }of $P$ and denote it by $\overline{P}$.
We also call the morphism $\overline{P}\to\Df_{S}^{*}$ a $G$\emph{-cover.}
\end{defn}

\begin{prop}
\label{prop:Theta}
\begin{enumerate}
\item Suppose $p>0$. Let $G$ be a Galoisian group and $\br$ a ramification
datum for $G$. Then there exist a reduced DM stack $\cX$ of finite
type and an isomorphism $\cX^{\iper}\xrightarrow{\sim}\Lambda_{[G]}^{[\br]}$
such that for any morphism $S\to\cX$ from an affine scheme, there
exists an etale cover $T\to S$ by an affine scheme such that the
composite morphism 
\[
T\to\cX\to\cX^{\iper}\xrightarrow{\sim}\Lambda_{[G]}^{[\br]}
\]
corresponds to a uniformizable torsor.
\item Let $G$ be a tame cyclic group. Then $\Lambda_{[G]}$ is a reduced
DM stack of finite type and for any morphism $S\to\Lambda_{[G]}$
from an affine scheme, there exists a finite etale cover $T\to S$
(necessarily by an affine scheme) such that the induced morphism 
\[
T\to\Lambda_{[G]}
\]
corresponds to a uniformizable torsor.
\end{enumerate}
\end{prop}

\begin{proof}
(1) First consider the case where $k=\FF_{p}$. From \ref{lem:Lambda-r},
there exist a reduced DM stack $\cY$ of finite type and an isomorphism
$\cY^{\iper}\cong\Lambda_{[G]}^{[\br]}$. Take an atlas $U\to\cY$
with $U$ an affine scheme. From \ref{lem:nice-finite-etale-cover},
there exists a finite etale cover $V\to U$ such that for $m\gg0$,
$V^{(-m)}\to\cY\to\Lambda_{[G]}^{[\br]}$ corresponds to a uniformizable
torsor. The stack $\cX=\cY^{(-m)}$ satisfies the desired property.
Indeed, since $k$ is perfect, $\cX^{(-m)}$ is of finite type. Obviously
we have $\cX^{\iper}\cong\Lambda_{[G]}^{[\br]}$. For any morphism
$S\to\cX$, the morphism $S\times_{\cX}V^{(-m)}\to S$ is an etale
cover and the morphism $S\times_{\cX}V^{(-m)}\to\Lambda_{[G]}^{[\br]}$
is uniformizable since it factors through $V^{(-m)}$. For a general
field $k$, we just take the base change $\cX\otimes_{\FF_{p}}k$.

(2) Obvious.
\end{proof}
\begin{defn}
\label{def:Theta}When $p>0$, for each $[G]$ and $[\br]$, we choose
a reduced DM stack of finite type $\Theta_{[G]}^{[\br]}$ as in \ref{prop:Theta},
(1). When $G$ is tame, there is only one class $[\br]$ and choose
$\Theta_{[G]}^{[\br]}$ to be $\Lambda_{[G]}$. We then define 
\begin{gather*}
\Theta:=\coprod_{G\in\GG}\coprod_{[\br]\in\Ram(G)/\Aut(G)}\Theta_{[G]}^{[\br]},\\
\Theta_{[G]}:=\Theta\times_{\cA}\cA_{[G]}=\coprod_{[\br]\in\Ram(G)/\Aut(G)}\Theta_{[G]}^{[\br]}.
\end{gather*}
When $p=0$, we define $\Theta:=\Lambda$ and $\Theta_{[G]}=\Lambda_{[G]}$.
For a Galoisian group scheme $G\to S$, we define $\Theta_{G}:=\Theta\times_{\cA}S$.
\end{defn}

All these stacks are almost of finite type, that is, each of them
has only countably many connected components each of which is of finite
type.  From \ref{cor:subring-O}, for any morphism $S\to\Theta$ from
an affine scheme $S$, the induced torsor $P=\Spec A\to\D_{S}^{*}$
has the subalgebra $O\subset A$. Let $O_{n}:=O/(t^{n+1})$.
\begin{defn}
We define a finite affine scheme $E_{\Theta,n}$ over $\Theta$ by
the property $E_{\Theta,n}\times_{\Theta}S=\Spec O_{n}$ for each
$S\to\Theta$. (This is also constructed as the relative spectrum
$\underline{\Spec}_{\Theta}O_{\Theta,n}$ of the coherent sheaf $O_{\Theta,n}$
on $\Theta$ given by modules $O_{U,n}$ in \ref{cor:subring-O}.)
There exists a natural morphism $E_{\Theta,n}\to\D_{\Theta,n}:=\D_{n}\times\Theta$,
which is representable, finite and flat. We define a formal DM stack
$E_{\Theta}:=\varinjlim E_{\Theta,n}$ and call it the \emph{universal
integral model }over $\Theta$. There exists a natural morphism $E_{\Theta}\to\Df_{\Theta}$.
For a morphism of stacks $\Sigma\to\Theta$, we call the induced morphism
\[
E_{\Sigma}:=E_{\Theta}\times_{\Df_{\Theta}}\Df_{\Sigma}=E_{\Theta}\times_{\Theta}\Sigma\to\Df_{\Sigma}
\]
a $G$\emph{-cover.}
\end{defn}

For each point $x\colon\Spec K\to\Theta$, the fiber product $E_{\Theta}\times_{\Theta}\Spec K$
is the integral model of the $G$-torsor $P\to\D_{K}^{*}$ corresponding
to $x$. The $E_{\Theta,n}\times_{\Theta}\Spec K$ is its closed subscheme
defined by $t^{n+1}$.

From \ref{lem:stable-G}, the group scheme $\cG_{\Theta}:=\cG\times_{\cA}\Theta$
over $\Theta$ acts on $E_{\Theta,n}$.
\begin{defn}
We define $\cE_{\Theta,n}$ to be the quotient stack $[E_{\Theta,n}/\cG_{\Theta}]$.
More precisely, we define a stack $\cE_{\Theta,n}$ as follows. For
$U\in\Aff$, an object of $\cE_{\Theta,n}(U)$ is a tuple $(U\to\Theta,P\to U,P\to E_{U,n})$
where $P\to U$ is a $\cG_{U}$-torsor and $P\to E_{U,n}$ is a $\cG_{U}$-equivariant
morphism. We then define $\cE_{\Theta}:=\varinjlim\cE_{\Theta,n}$
and call it the \emph{universal twisted formal disk }over $\Theta$.
\end{defn}

For each morphism $U\to\Theta$ from an affine scheme, we have $\cE_{\Theta,n}\times_{\Theta}U\cong[E_{U,n}/\cG_{U}]$.
We have natural morphisms of DM stacks $E_{\Theta,n}\to\cE_{\Theta,n}\to\D_{\Theta,n}$
and ones of formal DM stacks $E_{\Theta}\to\cE_{\Theta}\to\Df_{\Theta}$.
\begin{defn}
A \emph{twisted formal disk over a field $K$ }is a formal DM stack
$\cE$ over $\Df_{K}$ induced as the base change of $\cE_{\Theta}\to\Df_{\Theta}$
by a $K$-point $\Spec K\to\Theta$. Two twisted formal disks $\cE$
and $\cE'$ over $K$ are said to be \emph{isomorphic }if there exists
an isomorphism $\cE\to\cE'$ which makes the diagram
\[
\xymatrix{\cE\ar[r]\ar[dr] & \cE'\ar[d]\\
 & \Df_{K}
}
\]
2-commutative.
\end{defn}

\begin{lem}
\label{lem:tw-disk-iso}For an algebraically closed field $K$, we
have one-to-one correspondences
\[
\Lambda[K]\longleftrightarrow\Theta[K]\longleftrightarrow\{\textrm{twisted formal disk over }K\}/\mathord{\cong}.
\]
\end{lem}

\begin{proof}
The first correspondence is clear. An isomorphism in $\Lambda(K)$
gives an isomorphism $\alpha\colon G\cong G'$ of Galoisian groups
and $P\cong P'$ a $\D_{K}^{*}$-isomorphism compatible with $\alpha$
for a $G$-torsor $P\to\D_{K}^{*}$ and a $G'$-torsor $P'\to\D_{K}^{*}$.
For their models $\overline{P}$ and $\overline{P'}$, the quotient
stacks $[\overline{P}/G]$ and $[\overline{P'}/G']$ are clearly isomorphic
twisted formal disks. Conversely, given an isomorphism $[E/G]\to[E'/G']$
of twisted formal disks over $K$, consider the fiber product $U:=E\times_{[E'/G']}E'$.
This is finite and etale over both $E$ and $E'$. Therefore $U$
is isomorphic to the coproduct of copies of $E$ and to the coproduct
of copies of $E'$. Thus, if $U_{0}$ is a connected component of
$U$, then we have isomorphisms $E\cong U_{0}\cong E'$ over $\Df_{K}$.
It follows that the corresponding Galois extensions of $K\tpars$
are isomorphic and determines isomorphic $K$-points of $\Lambda$.
\end{proof}

\section{Untwisting stacks\label{sec:Untwisting-stacks}}

In this section, we introduce our main technical ingredient, the untwisting
stack. We define it as a certain Hom stack.

\subsection{Some results on Hom stacks}
\begin{defn}
Let $\cY,\cX,\cS$ be stacks over $\Aff$ with morphisms $\cY\to\cS$
and $\cX\to\cS$. We define $\ulHom_{\cS}(\cY,\cX)$ (resp.\ $\ulHom_{\cS}^{\rep}(\cY,\cX)$)
to be the fibered category over $\cS$ whose fiber category over an
$S$-point $S\to\cS$ is $\Hom_{S}(\cY\times_{\cS}S,\cX\times_{\cS}S)$
(resp.\,$\Hom_{S}^{\rep}(\cY\times_{\cS}S,\cX\times_{\cS}S)$), the
category of $S$-morphisms (resp.\ representable $S$-morphisms).
\end{defn}

Note that the canonical functor 
\[
\Hom_{S}^{\rep}(\cY\times_{\cS}S,\cX\times_{\cS}S)\to\Hom_{\cS}^{\rep}(\cY\times_{\cS}S,\cX)
\]
is an equivalence. Through this equivalence, we often identify the
two categories. A basic result on Hom stacks is the following one
by Olsson:
\begin{thm}[\cite{MR2357471,MR2239345}]
Let $S$ be a scheme and let $\cX$ and $\cY$ be DM stacks of finite
presentation over $S$. Suppose that $\cY$ is proper and flat over
$S$ and there exists a finite, finitely presented, flat and surjective
morphism $Z\to\cY$ from an algebraic space. Then $\ulHom_{S}(\cY,\cX)$
is a DM stack locally of finite presentation over $S$ and $\ulHom_{S}^{\rep}(\cY,\cX)$
is an open substack of it.
\end{thm}

We prove a slight generalization of this theorem as well as auxiliary
results on Hom stacks.
\begin{lem}
\label{lem:Hom-closed-imm}Let $S$ be an algebraic space, let $Y,X,W$
be algebraic spaces of finite presentation over $S$. Suppose that
$Y$ is flat and proper over $S$. Let $X\hookrightarrow W$ be a
closed immersion. Then the induced morphism 
\[
\ulHom_{S}(Y,X)\to\ulHom_{S}(Y,W)
\]
is also a closed immersion.
\end{lem}

\begin{proof}
A closed immersion is characterized by three properties: universally
closed, unramified and universally injective \cite[tag 04XV]{stacks-project}.
We check these properties.

Universally injective: Obvious.

Unramified: Let $T\hookrightarrow T'$ be a thickening of $S$-schemes.
We need to show that given the diagram below of solid arrows
\[
\xymatrix{T\ar[r]\ar[d] & \ulHom_{S}(Y,X)\ar[d]\\
T'\ar[r]\ar@{-->}[ur] & \ulHom_{S}(Y,W),
}
\]
then there exists at most one dashed arrow making the whole diagram
commutative. This holds because for the corresponding diagram below
of solid arrows
\[
\xymatrix{Y\times_{S}T\ar[r]\ar[d] & X\ar[d]\\
Y\times_{S}T'\ar[r]\ar@{-->}[ur] & W,
}
\]
there exists at most one dashed arrow, since $X\to W$ is unramified.

Universally closed: Let $R$ be a valuation ring and let $K$ be its
fraction field. Suppose that we have the commutative diagram below
of solid arrows.
\[
\xymatrix{\Spec K\ar[r]\ar[d] & \ulHom_{S}(Y,X)\ar[d]\\
\Spec R\ar[r]\ar@{-->}[ur] & \ulHom_{S}(Y,W)
}
\]
We need to show the existence of the solid arrow making the whole
diagram commutative. Let $f:Y\times_{S}\Spec R\to W$ be the morphism
corresponding to the bottom arrow. The induced morphism $g\colon Y\times_{S}\Spec K\to W$
factors through $X$. It means $g^{-1}\cI_{X}=0$ where $\cI_{X}\subset\cO_{W}$
is the defining ideal sheaf of $X\subset W$. Since $Y\times_{S}\Spec R$
is flat over $\Spec R$, the structure sheaf of $Y\times_{S}\Spec R$
has no $R$-torsion. This implies that $f^{-1}\cI_{X}=0$ and that
$f$ factors through $X$. This means the existence of the dashed
arrow.
\end{proof}
\begin{lem}
\label{lem:Hom-finite-space}Let $S$ be an algebraic space of finite
type, let $Y,X$ be algebraic spaces of finite type over $S$. Suppose
that $Y$ is flat and finite over $S$. Then $\ulHom_{S}(Y,X)$ is
an algebraic space of finite type.
\end{lem}

\begin{proof}
First consider the case where $S,Y,X$ are all affine schemes, say
$S=\Spec R$, $X=\Spec A$ and $Y=\Spec B$ and $B$ is a free $R$-module
of rank $m$. Suppose that $X$ is embedded in $\AA_{k}^{n}$. From
\ref{lem:Hom-closed-imm}, $\ulHom_{S}(Y,X)$ is a closed subspace
of $\hookrightarrow\ulHom_{S}(Y,\AA^{n})$. An element $f\in\ulHom_{S}(Y,\AA^{n})(\Spec B)$
is given by $f^{*}\colon B[x_{1},\dots,x_{n}]\to B^{\oplus m}$. Therefore
we have a universally injective morphism
\[
\ulHom_{S}(Y,\AA^{n})\to\AA_{S}^{nm}.
\]
As a consequence, $\ulHom_{S}(Y,X)$ is quasi-compact and of finite
type in this case.

We reduce the general case to the above special case step by step.
Firstly, taking the base change by a surjective morphism $S'\to S$,
we may suppose that $S$ is an affine scheme, which implies that $Y$
is also affine. By a further base change, we may suppose that, if
$S=\Spec R$ and $Y=\Spec B$, then $B$ is free over $R$. There
exists a surjective finite-type morphism $T\to S$ such that each
connected component of $(T\times_{S}Y)_{\red}$ maps isomorphically
onto $T$. This allows us to reduce to the case where $Y=\coprod_{i=1}^{n}Y_{i}$
and $Y_{i,\red}\to S$ is an isomorphism. Then $\ulHom_{S}(Y,X)=\prod_{i=1}^{n}\ulHom_{S}(Y_{i},X)$.
Thus we may further suppose that $Y_{\red}\to S$ is an isomorphism.
The morphism $S\cong Y_{\red}\hookrightarrow Y$ induces a morphism
$\ulHom_{S}(Y,X)\to\ulHom_{S}(S,X)=X$. For an etale morphism $U\to X$,
we have
\[
\ulHom_{S}(Y,X)\times_{X}U=\ulHom_{S}(Y,U).
\]
We can see this by looking at the diagram
\[
\xymatrix{T\ar[r]\ar@{^{(}->}[d] & U\ar[d]\\
Y_{T}\ar[r]\ar@{-->}[ur] & X
}
\]
induced from a morphism $T\to S$. Using this, we reduce the problem
to the case where $X$ is an affine scheme. This completes the proof.
\end{proof}
\begin{lem}
\label{lem:Hom-finite-stack}Let $\cS$ be a DM stack of finite type,
let $\cY,\cX$ be DM stacks of finite type over $\cS$. Suppose that
$\cY$ is flat and finite over $\cS$ and that there exists a finite,
etale and surjective morphism $\cU\to\cY$ such that $\cU\to\cS$
is representable. Then $\ulHom_{\cS}(\cY,\cX)$ and $\ulHom_{\cS}^{\rep}(\cY,\cX)$
are DM stacks of finite type.
\end{lem}

\begin{proof}
Let $T\to\cS$ be an atlas. Then 
\begin{align*}
\ulHom_{\cS}(\cY,\cX)\times_{\cS}T & \cong\ulHom_{T}(\cY\times_{\cS}T,\cX\times_{\cX}T),\\
\ulHom_{\cS}^{\rep}(\cY,\cX)\times_{\cS}T & \cong\ulHom_{T}^{\rep}(\cY\times_{\cS}T,\cX\times_{\cX}T)
\end{align*}
are DM stacks locally of finite type. From \ref{lem:DM-base-ch},
$\ulHom_{\cS}(\cY,\cX)$ and $\ulHom_{\cS}^{\rep}(\cY,\cX)$ are also
DM stacks locally of finite type. To show that they are of finite
type, it suffices to show that $\ulHom_{\cS}(\cY,\cX)$ is of finite
type in the case where $\cS=S$ is a scheme. Let $X$ is the coarse
moduli space of $\cX$. Then the morphism $\ulHom_{S}(\cY,\cX)\to\ulHom_{S}(\cY,X)$
is of finite type from \cite{MR2357471}. From the assumption, there
exists a finite, etale and surjective morphism $U\to\cY$ from an
algebraic space. Let $V:=U\times_{\cY}U$. We have a short exact sequence
\[
\ulHom_{S}(\cY,X)\to\ulHom_{S}(U,X)\rightrightarrows\ulHom_{S}(V,X).
\]
From \ref{lem:Hom-finite-space}, the right two Hom spaces are of
finite type. It follows that $\ulHom_{S}(\cY,X)$ is of finite type.
Therefore $\ulHom_{S}(\cY,\cX)$ is also of finite type.
\end{proof}
\begin{lem}
\label{lem:fiber-product}Let $\cC$ be a DM stack and let $\cE,\cX,\cY,\cZ$
be DM stacks over $\cC$. Let $f\colon\cY\to\cX$ and $g\colon\cZ\to\cX$
be a $\cC$-morphism. Then there exists an isomorphism 
\[
\ulHom_{\cC}(\cE,\cY\times_{\cX}\cZ)\to\ulHom_{\cC}(\cE,\cY)\times_{\ulHom_{\cC}(\cE,\cX)}\ulHom_{\cC}(\cE,\cZ).
\]
If $f$ and $g$ are representable morphisms of formal DM stacks,
then 
\[
\ulHom_{\cC}^{\rep}(\cE,\cY\times_{\cX}\cZ)\to\ulHom_{\cC}^{\rep}(\cE,\cY)\times_{\ulHom_{\cC}^{\rep}(\cE,\cX)}\ulHom_{\cC}^{\rep}(\cE,\cZ).
\]
\end{lem}

\begin{proof}
Let $S\to\cC$ be a morphism from an affine scheme. We define $\cE_{S}$
to be the base change of $\cE$ by $S\to\cC$. Similarly for $\cX_{S}$,
$\cY_{S}$ and $\cZ_{S}$. Let $\gamma\colon\cE_{S}\to\cY\times_{\cX}\cZ$
be a $\cC$-morphism. We get the induced morphisms $p_{\cY}\circ\gamma\colon\cE_{S}\to\cY$,
$p_{\cZ}\circ\gamma\colon\cE_{S}\to\cZ$ and the isomorphism $f\circ p_{\cY}\circ\gamma\to g\circ p_{\cZ}\circ\gamma$.
This defines a morphism
\[
\ulHom_{\cC}(\cE,\cY\times_{\cX}\cZ)\to\ulHom_{\cC}(\cE,\cY)\times_{\ulHom_{\cC}(\cE,\cX)}\ulHom_{\cC}(\cE,\cZ).
\]
We need to show that 
\[
\Hom_{S}(\cE_{S},(\cY\times_{\cX}\cZ)_{S})\to\Hom_{S}(\cE_{S},\cY_{S})\times_{\Hom_{S}(\cE_{S},\cX_{S})}\Hom_{S}(\cE_{S},\cZ_{S})
\]
is an isomorphism. This follows from 
\[
\cY_{S}\times_{\cX_{S}}\cZ_{S}\cong(\cY\times_{\cX}\cX_{S})\times_{\cX_{S}}(\cZ\times_{\cX}\cX_{S})\cong\cY\times_{\cX}\cZ\times_{\cX}\cX_{S}\cong\cY\times_{\cX}\cZ\times_{\cC}S
\]
and
\[
\Hom_{S}(\cE_{S},\cY_{S}\times_{\cX_{S}}\cZ_{S})\to\Hom_{S}(\cE_{S},\cY_{S})\times_{\Hom_{S}(\cE_{S},\cX_{S})}\Hom_{S}(\cE_{S},\cZ_{S})
\]
is an isomorphism from \cite[pp. 82--83]{MR3495343}.

For the second assertion, since $f$ and $g$ are representable, $\ulHom_{\cC}^{\rep}(\cE,\cY)$
and $\ulHom_{\cC}^{\rep}(\cE,\cZ)$ maps into $\ulHom_{\cC}^{\rep}(\cE,\cX)$.
Let $\beta\colon\cE_{S}\to\cY_{S}\times_{\cX_{S}}\cZ_{S}$ be a morphism
over $S$. It suffices to show that $\beta$ is representable if and
only if the induced morphisms $\cE_{S}\to\cY_{S}$ and $\cE_{S}\to\cZ_{S}$
are representable. This reduces to proving the following fact: Suppose
that $B\to A$ and $C\to A$ are injective homomorphisms of groups
and that $D$ is another group. Then a homomorphism $D\to B\times_{A}C$
is injective if and only if $D\to B$ and $D\to C$ are injective.
But this is obvious.
\end{proof}

\subsection{Untwisting stacks}
\begin{defn}
\label{def:untwg-stack}Let $\cX$ be a formal DM stack of finite
type over $\Df$. We fix a stack $\Theta$ as in \ref{def:Theta}.
We define the \emph{total untwisting} \emph{stack }$\Utg_{\Theta}(\cX)$
to be the Hom stack $\ulHom_{\Df_{\Theta}}^{\rep}(\cE_{\Theta},\cX_{\Theta})$,
where $\cX_{\Theta}:=\cX\times_{k}\Theta$. This is a fibered category
over $\Df_{\Theta}$ such that for a morphism $U\to\Df_{\Theta}$
from an affine scheme, the fiber category $\ulHom_{\Df_{\Theta}}^{\rep}(\cE_{\Theta},\cX_{\Theta})(U)$,
is the category $\Hom_{U}^{\rep}(\cE_{\Theta}\times_{\Df_{\Theta}}U,\cX_{\Theta}\times_{\Df_{\Theta}}U)$
of representable morphisms $\cE_{\Theta}\times_{\Df_{\Theta}}U\to\cX_{\Theta}\times_{\Df_{\Theta}}U$
over $U$. For any morphism $\sigma\colon\Sigma\to\Theta$ of stacks,
we define
\[
\Utg_{\Sigma}(\cX)=\Utg_{\sigma}(\cX):=\ulHom_{\Df_{\Theta}}^{\rep}(\cE_{\Theta},\cX_{\Theta})\times_{\Theta}\Sigma\left(\cong\ulHom_{\Df_{\Sigma}}^{\rep}(\cE_{\Sigma},\cX_{\Sigma})\right),
\]
where $\cE_{\Sigma}:=\cE_{\Theta}\times_{\Theta}\Sigma$.
\end{defn}

\begin{lem}
For a Galoisian group $G$ and a ramification datum $\br$ for $G$,
$\Utg_{\Theta_{[G]}^{[\br]}}(\cX)$ is a formal DM stack of finite
type over $\Df_{\Theta_{[G]}^{[\br]}}$.
\end{lem}

\begin{proof}
First note that $\Theta_{[G]}^{[\br]}$ is a DM stack of finite type.
Let $V\to\Theta_{[G]}^{[\br]}$ be an atlas. For each $n\in\NN$,
\[
\Utg_{V}(\cX)_{n}=\ulHom_{\D_{V,n}}^{\rep}(\cE_{V,n},\cX_{V,n})=\Utg_{\Theta_{[G]}^{[\br]}}(\cX)_{n}\times_{\Theta_{[G]}^{[\br]}}V
\]
is a DM stack of finite type from \ref{lem:Hom-finite-stack}. From
\ref{lem:DM-base-ch}, $\Utg_{\Theta_{[G]}^{[\br]}}(\cX)_{n}$ is
a DM stack of finite type. The lemma follows from \ref{lem:formal-finite-type}.
\end{proof}
From \ref{lem:tw-disk-iso}, for an algebraically closed field $K$,
the set of $K$-points
\[
\Utg_{\Theta}(\cX)[K]=\Utg_{\Theta}(\cX)_{0}[K]
\]
is identified with the isomorphism classes of pairs $(\cE,\cE_{0}\to\cX)$
of a twisted formal disk $\cE$ over $K$ and a representable morphism
$\cE_{0}\to\cX$. Here two pairs $(\cE,a\colon\cE_{0}\to\cX)$ and
$(\cE',b\colon\cE_{0}'\to\cX)$ are \emph{isomorphic} if there exists
an isomorphism $h\colon\cE\to\cE'$ such that $a$ is isomorphic to
\[
\cE_{0}\xrightarrow{h\times_{\Df}\D_{0}}\cE_{0}'\xrightarrow{b}\cX.
\]

\begin{lem}
\label{lem:Hom-completion}Let $S=\Spec R$ be an affine scheme and
let $\cE$ and $\cX$ be DM stacks over $\D_{S}=\Spec R\llbracket t\rrbracket$.
Let $\hat{\cE}$ and $\hat{\cX}$ be their $t$-adic completions,
which are formal DM stacks over $\Df_{S}=\Spf R\llbracket t\rrbracket$.
Then we have an isomorphism of stacks over $\Df_{S}$,
\[
\ulHom_{\Df_{S}}^{\rep}(\hat{\cE},\hat{\cX})\cong\widehat{\ulHom_{\D_{S}}^{\rep}(\cE,\cX)}.
\]
\end{lem}

\begin{proof}
For $n\in\NN$, we put $\cE_{n}:=\cE\times_{\D_{S}}\D_{S,n}=\hat{\cE}\times_{\Df_{S}}\D_{S,n}$
and similarly for $\cX_{n}$. Any morphism $U\to\Df_{S}$ from an
affine scheme factors through $\D_{S,n}$ for some $n\in\NN$. Therefore
we have natural isomorphisms
\begin{align*}
\ulHom_{\Df_{S}}^{\rep}(\hat{\cE},\hat{\cX})(U) & \cong\left(\ulHom_{\Df_{S}}^{\rep}(\hat{\cE},\hat{\cX})\times_{\Df_{S}}\D_{S,n}\right)(U)\\
 & \cong\ulHom_{\D_{S,n}}^{\rep}(\cE_{n},\cX_{n})(U)\\
 & \cong\left(\ulHom_{\D_{S}}^{\rep}(\cE,\cX)\times_{\D_{S}}\D_{S,n}\right)(U)\\
 & \cong\widehat{\ulHom_{\D_{S}}^{\rep}(\cE,\cX)}(U).
\end{align*}
\end{proof}
\begin{lem}
\label{lem:unt-rep-etale}If $\cY\to\cX$ is representable and etale,
then $\Utg_{\Theta}(\cY)\to\Utg_{\Theta}(\cX)$ is also representable
and etale. If $\cY\to\cX$ is stabilizer-preserving, etale and surjective,
then $\Utg_{\Theta}(\cY)\to\Utg_{\Theta}(\cX)$ is surjective.
\end{lem}

\begin{proof}
Let $\Spec K\to\Df_{\Theta}$ be a geometric point and let $r\colon\Spec K\to\Utg_{\Theta}(\cY)$
be a lift of it. The latter corresponds to a representable morphism
\[
\cE_{\Spec K}\to\cY.
\]
The automorphism group of this morphism is equal to $\Aut_{\Utg_{\Theta}(\cY)/\Df_{\Theta}}(r)$,
the group of those automorphisms of the object $r$ that map to the
identity morphism in $\Df_{\Theta}$. From \ref{lem:rep-stab-preserving},
the map $\Aut_{\Utg_{\Theta}(\cY)/\Df_{\Theta}}(r)\to\Aut_{\Utg_{\Theta}(\cX)/\Df_{\Theta}}(r')$
is injective, where $r'\colon\Spec K\to\Utg_{\Theta}(\cX)$ is the
image of $r$ and corresponds to the composition $\cE_{\Spec K}\to\cY\to\cX$.
From \ref{lem:Aut-inj}, $\Aut_{\Utg_{\Theta}(\cY)}(r)\to\Aut_{\Utg_{\Theta}(\cX)}(r')$
is also injective. Therefore, for each $n$, the morphism $\Utg_{\Theta}(\cY)_{n}\to\Utg_{\Theta}(\cX)_{n}$
of DM stacks is representable. We conclude that the morphism $\Utg_{\Theta}(\cY)\to\Utg_{\Theta}(\cX)$
is also representable.

Consider a 2-commutative diagram
\[
\xymatrix{S\ar[rr]\ar@{^{(}->}[d]_{\iota} &  & \Utg_{\Theta}(\cY)\ar[d]\\
S'\ar[r]\ar@{=>}[urr] & U\ar[r] & \Utg_{\Theta}(\cX)
}
\]
such that $S,S',U$ are algebraic spaces and $\iota$ is a thickening.
Ignoring $U$, we obtain the corresponding diagram: 
\begin{equation}
\xymatrix{\cE_{S}\ar@{^{(}->}[d]\ar[r] & \cY\ar[d]\\
\cE_{S'}\ar@{=>}[ur]\ar[r] & \cX
}
\label{diag2}
\end{equation}
Since $\cY\to\cX$ is representable and etale, from \ref{lem:etale},
there exists a morphism $\cE_{S'}\to\cY$ and two 2-morphisms forming
the 2-commutative diagram
\[
\xymatrix{\cE_{S}\ar@{^{(}->}[d]\ar[rr] &  & \cY\ar[d]\\
\cE_{S'}\ar[urr]\ar[rr]\ar@{=>}[ur] & {}\ar@{=>}[ur] & \cX
}
\]
which induces (\ref{diag2}). Moreover such a triple of a morphism
$\cE_{S'}\to\cY$ and two 2-morphisms is unique up to unique isomorphism
(see the uniqueness assertion in \ref{lem:etale}). We have the corresponding
diagram:
\[
\xymatrix{S\ar[rr]\ar@{^{(}->}[d]_{\iota} &  & \Utg_{\Theta}(\cY)\ar[d]\\
S'\ar[r]\ar[urr]\ar@{=>}[ur] & U\ar[r]\ar@{=>}[ur] & \Utg_{\Theta}(\cX)
}
\]
This shows that there exists a unique dashed arrow fitting into the
commutative diagram of algebraic spaces:
\[
\xymatrix{S\ar[r]\ar@{^{(}->}[d] & U\times_{\Utg_{\Theta}(\cX)}\Utg_{\Theta}(\cY)\ar[d]\\
S'\ar[r]\ar@{-->}[ur] & U
}
\]
Therefore $U\times_{\Utg_{\Theta}(\cX)}\Utg_{\Theta}(\cY)\to U$ is
etale for any morphism $U\to\Utg_{\Theta}(\cX)$ from an algebraic
space. It follows that $\Utg_{\Theta}(\cY)\to\Utg_{\Theta}(\cX)$
is etale.

Suppose now that $\cY\to\cX$ is stabilizer-preserving, etale and
surjective. Let $\Spec K\to\Utg_{\Theta}(\cX)$ be any geometric point.
This induces a geometric point 
\[
x\colon\Spec K\to\cE_{\Spec K}\to\cX.
\]
Let $y\colon\Spec K\to\cY$ be a lift of it. For some finite group
$G$ and a scheme $E$ with $E_{\red}=\Spec K$, we have $\cE_{\Spec K}\cong[E/G]$.
The representable morphism $\cE_{\Spec K}\to\cX$ induces an injection
$G\to\Aut_{\cX}(x)$. Since $\cY\to\cX$ is stabilizer-preserving,
the composition 
\[
\B G=[E_{\red}/G]\hookrightarrow\cE_{\Spec K}\to\cX
\]
lifts to $\B G\to\cY$. We obtain a 2-commutative diagram:
\begin{equation}
\xymatrix{\Spec K\ar[r] & \B G\ar[d]\ar[r] & \cY\ar[d]\\
 & \cE_{\Spec K}\ar[r] & \cX
}
\label{eq:E}
\end{equation}
The morphism $\B G\to\cE_{\Spec K}$ is a thickening. From \ref{lem:etale},
there exists a morphism $\cE_{\Spec K}\to\cY$ fitting into (\ref{eq:E}).
It gives a $K$-point of $\Utg_{\Theta}(\cY)$ lying over the chosen
$K$-point of $\Utg_{\Theta}(\cX)$. This shows the desired surjectivity.
\end{proof}

\subsection{The case of quotient stacks}

For a formal scheme $X$ of finite type over $\Df$ with an automorphism
$\alpha$, the $\alpha$-fixed locus $X^{\alpha}$ is defined to be
\[
X\times_{(\id_{X},\alpha),X\times_{\Df}X,\Delta}X.
\]
This becomes a closed subscheme of $X$ by the first projection. When
a finite group $H$ acts on $X$, we define the \emph{$H$-fixed locus
}$X^{H}$ to be $\bigcup_{h\in H}X^{h}$.
\begin{defn}
Let $G$ be a Galoisian group, let $V$ be a formal scheme over $\Df$
with an action of a finite group $H$ and let $\iota\colon G\to H$
be an embedding of finite groups. We define $\Utg_{\Theta_{G},\iota}(V)$
to be the $G$-fixed locus of 
\[
\ulHom_{\D_{\Theta_{G}}}(E_{\Theta_{G}},V_{\Theta_{G}})
\]
by the action $g(f):=\iota(g)fg^{-1}$.
\end{defn}

Note that the Hom stack $\ulHom_{\D_{\Theta_{G}}}(E_{\Theta_{G}},V_{\Theta_{G}})$
is identical to the Weil restriction $\R_{E_{\Theta_{G}}/\Df_{\Theta_{G}}}(V\times_{\Df_{\Theta_{G}}}E_{\Theta_{G}})$.

An object of $\Utg_{\Theta_{G},\iota}(V)$ over an $S$-point $S\to\Df_{\Theta_{G}}$
is an $\iota$-equivariant morphism $E_{\Theta_{G}}\times_{\Df_{\Theta_{G}}}S\to V.$
This induces a representable morphism 
\[
[(E_{\Theta_{G}}\times_{\Df_{\Theta_{G}}}S)/G]\cong\cE_{\Theta_{G}}\times_{\Df_{\Theta_{G}}}S\to[V/\iota(G)]\to[V/H].
\]
In turn, this leads to a morphism 
\[
\coprod_{\iota\in\Emb(G,H)}\Utg_{\Theta_{G},\iota}(V)\to\Utg_{\Theta_{G}}(\cX),
\]
where $\Emb(G,H)$ is the set of embeddings $G\hookrightarrow H$.

We define actions of $H$ and $\Aut(G)$ on $\coprod_{\iota}\Utg_{\Theta_{G},\iota}(V)$.
For $h\in H$, let $c_{h}\in\Aut(H)$ be the automorphism $f\mapsto hfh^{-1}$.
Given an $\iota$-equivariant morphism $\psi\colon F\to V$, the composition
$F\xrightarrow{\psi}V\xrightarrow{h}V$ is $c_{h}\iota$-equivariant.
Sending $F\to V$ to $F\xrightarrow{}V\xrightarrow{h}V$ gives a morphism
\[
\Utg_{\Theta_{G},\iota}(V)\to\Utg_{\Theta_{G},c_{h}\iota}(V)
\]
and defines an $H$-action on $\coprod_{\iota}\Utg_{\Theta_{G},\iota}(V)$.

For $c\in\Aut(G)$, let $F^{(c)}$ be the scheme $F$ with the new
$G$-action such that the automorphism $c^{-1}(g)\colon F\to F$ for
the old action is the new $g$-action. Then the same scheme morphism
$\psi\colon F^{(c)}\to V$ is now $\iota c^{-1}$-equivariant. Sending
$F\to V$ to $F^{(c)}\to V$ gives a morphism 
\[
\Utg_{\Theta_{G},\iota}(V)\to\Utg_{\Theta_{G},\iota c^{-1}}(V)
\]
and defines an $\Aut(G)$-action on $\coprod_{\iota}\Utg_{\Theta_{G},\iota}(V).$

The two actions commute and we obtain the $\Aut(G)\times H$-action.
The morphism $\coprod_{\iota}\Utg_{\Theta_{G},\iota}(V)\to\Utg_{\Theta_{G}}(\cX)$
is invariant for the $H$-action and equivariant for the $\Aut(G)$-
and $\Aut(G)\times H$-actions. From the universality of quotient
stacks \cite{MR2125542} (similar to the one of quotient schemes),
we obtain morphisms
\begin{gather}
\left[\left(\coprod_{\iota\in\Emb(G,H)}\Utg_{\Theta_{G},\iota}(V)\right)/H\right]\to\Utg_{\Theta_{G}}(\cX),\label{eq:UnG1}\\
\left[\left(\coprod_{\iota\in\Emb(G,H)}\Utg_{\Theta_{G},\iota}(V)\right)/\Aut(G)\times H\right]\to\Utg_{\Theta_{[G]}}(\cX).\label{eq:UnG2}
\end{gather}
Note that $\Theta_{G}\to\Theta_{[G]}$ as well as $\Utg_{\Theta_{G}}(\cX)\to\Utg_{\Theta_{[G]}}(\cX)$
is an $\Aut(G)$-torsor, since it is a base change of $\Spec k\to\cA_{[G]}\cong\B\Aut(G)$.
Since $\Aut(G)$ acts freely on $\Emb(G,H)$,  if $\Emb(G,H)/\Aut(G)$
denotes a set of representatives of $\Aut(G)$-orbits, we have 
\[
\left[\left(\coprod_{\iota\in\Emb(G,H)}\Utg_{\Theta_{G},\iota}(V)\right)/\Aut(G)\right]\cong\coprod_{\iota\in\Emb(G,H)/\Aut(G)}\Utg_{\Theta_{G},\iota}(V).
\]
The $\Aut(G)\times H$-action on $\Emb(G,H)$ is not generally free.
The stabilizer of $\iota$ is 
\[
\{(\iota^{*}c_{h},h)\mid h\in\N_{H}(\iota(G))\}\;(\cong\N_{H}(\iota(G))).
\]
Here $\N_{H}(\iota(G))$ is the normalizer of $\iota(G)$ in $H$
and $\iota^{*}c_{h}$ is the automorphism of $G$ corresponding to
$c_{h}|_{\iota(G)}$ via $\iota$. We have the induced $\N_{H}(\iota(G))$-action
on $\Utg_{\Theta_{G},\iota}(V)$; for $h\in\N_{H}(\iota(G))$, 
\[
h\colon\Utg_{\Theta_{G},\iota}(V)\to\Utg_{\Theta_{G},c_{h}\iota}(V)\to\Utg_{\Theta_{G},c_{h}\iota(\iota^{*}c_{h})^{-1}}(V)=\Utg_{\Theta_{G},\iota}(V),
\]
being the composition of morphisms given above. Therefore 
\begin{multline}
\left[\left(\coprod_{\iota\in\Emb(G,H)}\Utg_{\Theta_{G},\iota}(V)\right)/(\Aut(G)\times H)\right]\\
\cong\coprod_{\iota\in\Emb(G,H)/(\Aut(G)\times H)}[\Utg_{\Theta_{G},\iota}(V)/\N_{H}(\iota(G))].\label{eq:utg-N_H}
\end{multline}
Note that $\Emb(G,H)/(\Aut(G)\times H)$ is in one-to-one correspondence
with the set of $H$-conjugacy classes of subgroups $G'\subset H$
such that $G'\cong G$. In particular, if $G$ is a cyclic group of
order $l$, then it is in one-to-one correspondence with the set of
conjugacy classes of $H$ that have order $l$. We also have
\[
\left[\left(\coprod_{\iota\in\Emb(G,H)}\Utg_{\Theta_{G},\iota}(V)\right)/H\right]\cong\coprod_{\iota\in\Emb(G,H)/H}[\Utg_{\Theta_{G},\iota}(V)/\C_{H}(\iota(G))].
\]
Here $\C_{H}(\iota(G))$ denotes the centralizer of $\iota(G)$ in
$H$.
\begin{lem}
\label{lem:Unt-quot}Morphisms (\ref{eq:UnG1}) and (\ref{eq:UnG2})
are isomorphisms.
\end{lem}

\begin{proof}
It suffices to show that $\coprod_{\iota\in\Emb(G,H)}\Utg_{\Theta_{G},\iota}(V)\to\Utg_{\Theta_{G}}(\cX)$
is an $H$-torsor. This is true because the base change by $E_{\Theta_{G}}\to\Df_{\Theta_{G}}$
is an $H$-torsor as proved in \ref{lem:quot-unt}.
\end{proof}
\begin{lem}
\label{lem:quot-unt}Let $\cX:=[V/H]$. The diagram
\[
\xymatrix{\coprod_{\iota\in\Emb(G,H)}\Utg_{\Theta_{G},\iota}(V)\times_{\Df_{\Theta_{G}}}E_{\Theta_{G}}\ar[r]\ar[d] & V\ar[d]\\
\Utg_{\Theta_{G}}(\cX)\times_{\Df_{\Theta_{G}}}E_{\Theta_{G}}\ar[r] & \cX
}
\]
is 2-Cartesian. In particular, the left vertical arrow is an $H$-torsor.
\end{lem}

\begin{proof}
Since $V\to\cX=[V/H]$ is an $H$-torsor, the second assertion follows
from the first. Let $s\colon S\to\Df_{\Theta_{G}}$ be a morphism
from an affine scheme, let $E_{S}:=E_{\Theta_{G}}\times_{\Theta_{G},s}S$
and $\cE_{S}:=\cE_{\Theta_{G}}\times_{\Theta_{G},s}S$ and let $V_{S}:=V\times_{\Df}S$
and $\cX_{S}:=\cX\times_{\Df}S$. The fiber over $s$ of the fibered
category $(\Utg_{\Theta_{G}}(\cX)\times_{\Df_{\Theta_{G}}}E_{\Theta_{G}})\times_{\cX}V\to\Df_{\Theta_{G}}$
is regarded the category of 2-commutative diagrams of the form
\begin{equation}
\xymatrix{S\ar[rr]\ar[d] &  & V_{S}\ar@{..>}[d]\\
E_{S}\ar@{..>}[r]\ar@{=>}[urr] & \cE_{S}\ar[r] & \cX_{S}
}
\label{eq:diag1}
\end{equation}
where the dotted arrows are the canonical ones, the normal arrows
are $S$-morphisms and the thick arrow is a witnessing 2-isomorphism.
On the other hand, the fiber over $s$ of the fibered category $\Utg_{\Theta_{G},\iota}(V)\times_{\Df_{\Theta_{G}}}E_{\Theta_{G}}\to\Df_{\Theta_{G}}$
is identified with the category of diagrams of $S$-morphisms
\begin{equation}
\xymatrix{S\ar[d] & V_{S}\\
E_{S}\ar[ur]_{\iota\text{-eq.}}
}
\label{eq:diag2}
\end{equation}
where $E_{S}\to V_{S}$ is $\iota$-equivariant.

Given a diagram of form (\ref{eq:diag2}), we obtain dashed arrows
and thick arrows forming the 2-commutative diagram:
\[
\xymatrix{S\ar[d]\ar@{-->}[rr] & {} & V_{S}\ar@{..>}[d]\\
E_{S}\ar[urr]\ar@{..>}[r]\ar@{=>}[ur] & \cE_{S}\ar@{-->}[r]\ar@{=>}[ur] & [V_{S}/\iota(H)]\ar@{..>}[r] & \cX_{S}
}
\]
Again the dotted arrows are the canonical ones. Thus we obtain a diagram
of form (\ref{eq:diag1}) and a morphism
\[
\Utg_{\Theta_{G},\iota}(V)\times_{\Df_{\Theta_{G}}}E_{\Theta_{G}}\to(\Utg_{\Theta_{G}}(\cX)\times_{\Df_{\Theta_{G}}}E_{\Theta_{G}})\times_{\cX}V.
\]

Conversely, given a diagram of form (\ref{eq:diag1}), from \ref{lem:etale},
we obtain a dashed arrow and a 2-isomorphism for each triangle in
\[
\xymatrix{S\ar[rr]\ar[d] & {} & V_{S}\ar@{..>}[d]\\
E_{S}\ar@{..>}[r]\ar@{-->}[urr]\ar@{=>}[ur] & \cE_{S}\ar[r]\ar@{=>}[ur] & \cX_{S}
}
\]
which induces the given 2-isomorphism for the rectangle. We claim
that there exists a unique decomposition $S=\coprod_{\iota\in\Emb(G,H)}S_{\iota}$
into open and closed subschemes such that the morphism $E_{S}\to V_{S}$
restricts to $\iota$-equivariant morphisms $E_{S_{\iota}}\to V_{S_{\iota}}$.
The subdiagram 
\[
\xymatrix{E_{S}\ar[r]\ar[d] & V_{S}\ar[d]\\
\cE_{S}=[E_{S}/H]\ar[r] & \cX_{S}=[V_{S}/G]
}
\]
gives a morphism of groupoids $(E_{S}\times H\rightrightarrows E_{S})\to(V_{S}\times G\rightrightarrows V_{S})$.
Therefore there exists a unique decomposition $S=\coprod_{\iota\in\Emb(G,H)}S_{\iota}$
into open and closed subschemes such that for each $\iota$, the morphism
$E_{S}\times H\to V_{S}\times G$ restricts to the morphism $E_{S_{\iota}}\times H\to V_{S_{\iota}}\times G$
compatible with $\iota$. This shows the claim. Thus we have obtained
a morphism 
\[
(\Utg_{\Theta_{G}}(\cX)\times_{\D_{\Theta_{G}}}\cE_{\Theta_{G}})\times_{\cX}V\to\coprod_{\iota\in\Emb(G,H)}\Utg_{\Theta_{G},\iota}(V)\times_{\D_{\Theta_{G}}}E_{\Theta_{G}}.
\]
We can check that thus obtained morphisms between $(\Utg_{\Theta_{G}}(\cX)\times_{\D_{\Theta_{G}}}\cE_{\Theta_{G}})\times_{\cX}V$
and $\coprod_{\iota\in\Emb(G,H)}\Utg_{\Theta_{G},\iota}(V)\times_{\D_{\Theta_{G}}}E_{\Theta_{G}}$
are quasi-inverses to each other, by using the uniqueness in \ref{lem:etale}.
\end{proof}

\subsection{Variants of inertia stacks}

Recall from Section \ref{sec:Galoisian-group-schemes} that $\cA$
denotes the moduli stack of Galoisian group schemes and $\cG\to\cA$
is the universal group scheme.
\begin{defn}
We define a fibered category $\B\cG\to\cA$ as follows: for an object
$G\to S$ of $\cA$, which is by definition a Galoisian group scheme
over an affine scheme, the fiber category $(\B\cG)(G\to S)$ is $(\B G)(S)$,
that is, the groupoid of $G$-torsors over $S$. For a Galoisian group
$G$, we define $\B\cG_{[G]}:=\B\cG\times_{\cA}\cA_{[G]}$.
\end{defn}

\begin{defn}
Let $\cZ$ be a DM stack of finite type over $k$. For a Galoisian
group scheme $G\to\Spec k$, we define $\I^{(G)}\cZ:=\ulHom_{\Spec k}^{\rep}(\B G,\cZ)$.
We define $\I^{(\bullet)}\cZ:=\coprod_{G\in\GG}\I^{(G)}\cZ$.

For a Galoisian group $G$, we define $\I^{[G]}\cZ:=\ulHom_{\cA_{[G]}}^{\rep}(\B\cG_{[G]},\cZ)$.
We also define $\I^{[\bullet]}\cZ:=\coprod_{G\in\GG}\I^{[G]}\cZ$.
\end{defn}

These are variants of inertia stack and their relation is closer in
the tame case, as we will see in Section \ref{sec:The-tame-case}.
For the morphism $\Spec k\to\cA_{[G]}$ given by $G\to\Spec k$, we
have $\I^{(G)}\cZ\cong\I^{[G]}\cZ\times_{\cA_{[G]}}\Spec k$.
\begin{lem}
\label{lem:IG}Let $\cZ$ be a quotient stack $[V/H]$ associated
to a finite group action $H\curvearrowright V$ on an algebraic space.
For a Galoisian group $G$, we have
\begin{align*}
\I^{(G)}\cZ & \cong\left[\left(\coprod_{\iota\in\Emb(G,H)}V^{\iota(G)}\right)/H\right]\\
 & \cong\coprod_{\iota\in\Emb(G,H)/H}[V^{\iota(G)}/\C_{H}(\iota(G))]
\end{align*}
and
\begin{align*}
\I^{[G]}\cZ & \cong\left[\left(\coprod_{\iota\in\Emb(G,H)}V^{\iota(G)}\right)/(\Aut(G)\times H)\right]\\
 & \cong\coprod_{\iota\in\Emb(G,H)/(\Aut(G)\times H)}[V^{\iota(G)}/\N_{H}(\iota(G))].
\end{align*}
\end{lem}

\begin{proof}
We apply the same arguments as in the proof of \ref{lem:Unt-quot}
to $\B G=[\Spec k/G]$, $V^{\iota(G)}=\ulHom_{\Spec k}^{\iota}(\Spec k,V)$
(the right side is the scheme of $\iota$-equivariant morphisms) and
$\I^{[G]}\cZ$ .
\end{proof}
For a formal DM stack $\cX$ of finite type over $\Df$, we have a
morphism 
\[
\Utg_{\Theta_{[G]}}(\cX)_{0}\to\I^{[G]}(\cX_{0})
\]
as follows. Let $S\to\D_{0,\Theta_{[G]}}$ be a morphism with $S\in\Aff$
and $G_{S}\to S$ the associated group scheme and 
\[
\cE_{S}:=\cE_{\Theta_{[G]}}\times_{\Df_{\Theta_{[G]}}}S\to\cX
\]
a representable morphism, which give an $S$-point of $\Utg_{\Theta_{[G]}}(\cX)_{0}$.
The given section $S\to\cE_{S}$ has automorphism group scheme $G_{S}$.
Therefore we have a closed immersion $\B(G_{S}/S)\hookrightarrow\cE_{S}$.
We send the above $S$-point of $\Utg_{\Theta}(\cX)_{0}$ to the $S$-point
of $\I^{[G]}(\cX)$ defined by the induced $\B(G_{S}/S)\to\cX_{0}$
and $G_{S}\colon S\to\cA_{[G]}$.

By base change along the morphism $\Spec k\to\cA_{[G]}$, given by
a Galoisian group $G$, we get a morphism 
\[
\Utg_{\Theta_{G}}(\cX)_{0}\to\I^{(G)}(\cX_{0}).
\]

\subsection{More on untwisting stacks}

In this and next subsection, we use a few notions from the rigid geometry,
for which we refer the reader to \cite{MR2815110}. To a formal scheme
$X$ of finite type over $\Df$, we can associate a coherent rigid
space $X^{\rig}$, which may be regarded as the ``generic fiber''
of $X\to\Df$.
\begin{defn}[{\cite[Prop. 6.4.12]{MR2815110}}]
Let $U,V$ be formal affine schemes of finite type over $\Df$ and
let $f\colon V\to U$ be a $\Df$-morphism. We say that $f$ is \emph{rig-etale
}if the induced morphism $f^{\rig}\colon V^{\rig}\to U^{\rig}$ of
coherent rigid spaces is etale.
\end{defn}

\begin{lem}
\label{lem:rig-etale-compos}Let $g\colon Z\to Y$ and $f\colon Y\to X$
be morphisms of coherent rigid spaces.
\begin{enumerate}
\item If $f\circ g$ is etale and $f$ is unramified, then $f$ is etale.
\item If $f\circ g$ is etale, $g$ is etale and the map $\langle Z\rangle\to\langle Y\rangle$
of sets of rigid points is surjective, then $f$ is etale.
\end{enumerate}
\end{lem}

\begin{proof}
(1) \cite[Prop. 6.4.3.(vi).]{MR2815110}.

(2) Let $y\in\langle Y\rangle$ be any rigid point, let $x\in\langle X\rangle$
be its image and let $z\in\langle Z\rangle$ be a lift of $y$. Let
$Z_{y}$ denote the fiber of $Z\to Y$ over $y$. Similarly for $Z_{x}$
and $Y_{x}$. From \cite[Prop. 6.4.24.(v)]{MR2815110}, $\cO_{Z_{y},z}/\kappa(y)$
and $\cO_{Z_{x},z}/\kappa(x)$ are finite separable field extensions.
If $\fm$ denotes the maximal ideal of $\cO_{Y_{x},y}$, then since
$\cO_{Z_{x},z}$ is a field, we have $\fm\cO_{Z_{x},z}=0$. Since
$\cO_{Y_{x},y}\to\cO_{Z_{x},z}$ is faithfully flat, it is injective.
Therefore $\fm=0$ and $\cO_{Y_{x},y}$ is a field. This is a finite
separable extension of $\kappa(x)$, since $\cO_{Z_{x},z}$ is so.
Therefore $f$ is etale.
\end{proof}
\begin{lem}
\label{lem:rig-etale-compos-2}Let 
\[
\xymatrix{\tilde{V}\ar[r]^{\tilde{f}}\ar[d] & \tilde{U}\ar[d]\\
V\ar[r]_{f} & U
}
\]
be a commutative diagram of  formal schemes of finite type over $\Df$.
Suppose that the vertical arrows are etale and surjective. Then $f$
is rig-etale if and only if $\tilde{f}$ is rig-etale.
\end{lem}

\begin{proof}
If either $f$ or $\tilde{f}$ is rig-etale, then $\tilde{V}\to U$
is rig-etale. Since $\tilde{V}\to V$ is etale and surjective, the
map $\langle\tilde{V}\rangle\to\langle V\rangle$ is surjective. Now
the ``if'' and ``only if'' parts follow from (2) and (1) of \ref{lem:rig-etale-compos}
respectively.
\end{proof}
\begin{defn}
Let $\cY,\cX$ be formal DM stacks of finite type over $\Df$. We
say that a $\Df$-morphism $\cY\to\cX$ is \emph{rig-etale }if it
fits into some 2-commutative diagram
\[
\xymatrix{V\ar[r]\ar[d] & U\ar[d]\\
\cY\ar[r] & \cX
}
\]
such that vertical arrows are atlases from formal affine schemes and
the upper arrow is rig-etale.
\end{defn}

Form \ref{lem:rig-etale-compos-2}, ``some'' in the last condition
can be replaced with ``any.''
\begin{lem}
\label{lem:rig-etale}Let $G$ be a finite group and let $V$ be a
formal scheme of finite type over $\Df$ with an action of $G$. Then
the natural morphism 
\[
\Utg_{\Theta_{G},\iota}(V)\times_{\Df_{\Theta_{G}}}E_{\Theta_{G}}\to V\times\Theta_{G}
\]
is rig-etale.
\end{lem}

\begin{proof}
From the construction of $\Theta_{G}$ (see \ref{def:Theta}), there
exists an atlas $T=\Spec R\to\Theta_{G}$ from a formal affine scheme
which corresponds to a uniformizable torsor 
\[
P_{T}=\Spec A\to\D_{T}^{*}=\Spec R\tpars.
\]
Let $E_{T}=\Spf O_{A}\to\Df_{T}$ be its integral model and let $E^{\rig}\to\Df_{T}^{\rig}$
be the induced morphism of rigid spaces, which is an etale $G$-torsor.

Let $V_{T}:=V\times T$. It suffices to show that the morphism of
formal schemes $\Utg_{T,\iota}(V)\times_{\Df_{T}}E_{T}\to V_{T}$
is rig-etale. Here $\Utg_{T,\iota}(V):=\Utg_{\Theta_{G},\iota}(V)\times_{\Theta_{G}}T$.
By definition, $\Utg_{T,\iota}(V)$ is the $G$-fixed subscheme of
the Weil restriction $\R_{E_{T}/\Df_{T}}(V_{T}\times_{\Df_{T}}E_{T})$.
From \cite[Prop. 1.22]{MR1752779}, the rigid space $\Utg_{T,\iota}(V)^{\rig}$
is an open subspace of $\R_{E^{\rig}/\Df_{T}^{\rig}}(V_{T}^{\rig}\times_{\Df_{T}^{\rig}}E^{\rig})^{G}$.
It suffices to show that 
\[
\theta\colon\R_{E^{\rig}/\Df_{T}^{\rig}}(V_{T}^{\rig}\times_{\Df_{T}^{\rig}}E^{\rig})^{G}\times_{\Df_{T}^{\rig}}E^{\rig}\to V_{T}^{\rig}
\]
is formally etale. Let us consider the commutative diagram of solid
arrows
\begin{equation}
\xymatrix{S\ar@{^{(}->}[d]\ar[rr]\sp(0.2){\phi} &  & \R_{E^{\rig}/\Df_{T}^{\rig}}(V_{T}^{\rig}\times_{\Df_{T}^{\rig}}E^{\rig})^{G}\times_{\Df_{T}^{\rig}}E^{\rig}\ar[d]^{\theta}\\
S'\ar[rr]_{\phi'}\ar@{-->}[urr] &  & V_{T}^{\rig}
}
\label{diag-A}
\end{equation}
where $S\hookrightarrow S'$ is a thickening of coherent rigid spaces
over $\Df_{T}^{\rig}$. The $S$-point $\phi$ is given by the pair
$(f,s)$ of a $G$-equivariant morphism $f\colon S\times_{\Df_{T}^{\rig}}E^{\rig}\to V_{T}^{\rig}$
and a $\Df_{T}^{\rig}$-morphism $s\colon S\to E^{\rig}$. Since the
$G$-torsor $S\times_{\Df_{T}^{\rig}}E^{\rig}\to S$ has the section
$\sigma:=(\id_{S},s)$, we have the decomposition 
\begin{equation}
S\times_{\Df_{T}^{\rig}}E^{\rig}=\coprod_{g\in G}g\sigma(S).\label{eq:decomp}
\end{equation}
The morphism $\theta$ sends the $S$-point $\phi$ to the $S$-point
\[
\psi\colon S\xrightarrow{\sigma}S\times_{\Df_{T}^{\rig}}E^{\rig}\xrightarrow{f}V_{T}^{\rig}.
\]
This $\psi$ extends an $S'$-point $\phi'\colon S'\to V_{T}^{\rig}$.
Thanks to the above decomposition (\ref{eq:decomp}), $\phi'$ uniquely
extends to the $G$-equivariant morphism $f'\colon S'\times_{\Df_{T}^{\rig}}E^{\rig}\to V_{T}^{\rig}$,
which restricts to $f$ on $S\times_{\Df_{T}^{\rig}}E^{\rig}$. On
the other hand, since $E^{\rig}\to\Df_{T}^{\rig}$ is etale, the morphism
$s$ also uniquely extends to $s'\colon S'\to E^{\rig}$. The pair
$(f',s')$ gives the unique dashed arrow fitting in diagram (\ref{diag-A}).
Thus $\theta$ is formally etale.
\end{proof}
\begin{lem}
\label{lem:rig-etale-2}The morphism $\Utg_{\Theta}(\cX)\times_{\Df_{\Theta}}\cE_{\Theta}\to\cX\times\Theta$
is rig-etale.
\end{lem}

\begin{proof}
It suffices to show that for each Galoisian group $G$, the morphism
\[
\Utg_{\Theta_{G}}(\cX)\times_{\Df_{\Theta_{G}}}\cE_{\Theta_{G}}\to\cX\times\Theta_{G}
\]
is rig-etale. Moreover, from \ref{lem:formal-locally-quot} and \ref{lem:unt-rep-etale},
we can reduce to the case where $\cX$ is a quotient stack $[V/H]$
associated to a formal scheme $V$ with an action of a finite group
$H$. Consider the following 2-commutative diagram:
\[
\xymatrix{\coprod_{\iota\in\Emb(G,H)}\Utg_{\Theta_{G},\iota}(V)\times_{\Df_{\Theta_{G}}}E_{\Theta_{G}}\ar[r]\ar[d] & V\times\Theta_{G}\ar[d]\\
\Utg_{\Theta_{G}}(\cX)\times_{\Df_{\Theta_{G}}}\cE_{\Theta_{G}}\ar[r] & \cX\times\Theta_{G}
}
\]
Since the vertical arrows are etale and the upper arrow is rig-etale,
we conclude that the bottom arrow is also rig-etale.
\end{proof}
\begin{rem}
\label{rem:unt-generic-fib}In the non-formal setting, the ``untwisting
stack'' for a uniformizable $G$-torsor does not change the generic
fiber. Lemma \ref{lem:rig-etale-2} is an analogue of this fact. More
precisely, let $\cX$ be a (non-formal) DM stack of finite type over
the genuine scheme $\D=\Spec k\tbrats$, let $T$ be an affine scheme
and let $\Spec A\to\D_{T}^{*}$ is a uniformizable $G$-torsor. Consider
the non-formal counterpart 
\[
\cU_{T}:=\ulHom_{\D_{T}}^{\rep}([\Spec O_{A}/G],\cX\times_{\D}\D_{T})
\]
of the untwisting stack $\Utg_{T}(\cX)$. We have
\begin{align*}
\cU_{T}\times_{\D}\D^{*} & \cong\ulHom_{\D_{T}}^{\rep}([\Spec O_{A}/G],\cX\times_{\D}\D_{T})\times_{\D_{T}}\D_{T}^{*}\\
 & \cong\ulHom_{\D_{T}^{*}}^{\rep}([\Spec O_{A}/G]\times_{\D_{T}}\D_{T}^{*},\cX\times_{\D}\D_{T}^{*})\\
 & \cong\ulHom_{\D_{T}^{*}}^{\rep}([\Spec A/G],\cX\times_{\D}\D_{T}^{*})\\
 & \cong\ulHom_{\D_{T}^{*}}^{\rep}(\D_{T}^{*},\cX\times_{\D}\D_{T}^{*})\\
 & \cong\cX\times_{\D}\D_{T}^{*}.
\end{align*}
In particular, if $T=\Spec k$, the two $\D$-stacks $\cU_{\Spec k}$
and $\cX$ share the generic fiber.
\end{rem}

\begin{lem}
\label{lem:closed-immersion-Unt}Let $\cY\hookrightarrow\cX$ be a
closed immersion of formal DM stacks of finite type over $\Df$. The
induced morphism $\Utg_{\Theta}(\cY)\to\Utg_{\Theta}(\cX)$ is a closed
immersion.
\end{lem}

\begin{proof}
For each $G\in\GG$, we prove that $\Utg_{\Theta_{[G]}}(\cY)\to\Utg_{\Theta_{[G]}}(\cX)$
is a closed immersion. By base change, it suffices to show that $\Utg_{\Theta_{G}}(\cY)\to\Utg_{\Theta_{G}}(\cX)$
is a closed immersion. Let $\cX':=\coprod_{\gamma}[W_{\gamma}/H_{\gamma}]\to\cX$
be as in \ref{lem:formal-locally-quot}, which is a stabilizer-preserving,
etale and surjective morphism, and let $\cY':=\cY\times_{\cX}\cX'$.
From \ref{lem:fiber-product}, the diagram
\[
\xymatrix{\Utg_{\Theta_{G}}(\cY')\ar[d]\ar[r] & \Utg_{\Theta_{G}}(\cX')\ar[d]\\
\Utg_{\Theta_{G}}(\cY)\ar[r] & \Utg_{\Theta_{G}}(\cX)
}
\]
is 2-Cartesian. The upper arrow is a closed immersion from \ref{lem:quot-unt}
and \ref{lem:Hom-closed-imm}. From \ref{lem:unt-rep-etale}, the
vertical ones are etale and surjective. Therefore the bottom arrow
is also a closed immersion.
\end{proof}
Let $\cX$ be a formal DM stack of finite type and $\cY$ a closed
substack of $\cX$ defined by an ideal sheaf $\cI\subset\cO_{\cX}$.
Let $\cY_{(n)}$ denote the closed substack defined by $\cI^{n+1}$.
We get an inductive system of closed substacks 
\[
\cY=\cY_{(0)}\hookrightarrow\cY_{(1)}\hookrightarrow\cdots.
\]
We define the ind-DM stack $\widehat{\cX}$ to be the limit $\varinjlim\cY_{(n)}$
and call it the \emph{completion} of $\cX$ along $\cY$.
\begin{lem}
\label{lem:Utg-completion}Let $\Gamma$ be a DM stack of finite type
and let $\Gamma\to\Theta_{[G]}$ be any morphism. Suppose that $\cY\subset\cX_{m}:=\cX\times_{\Df}\Df_{m}$
for some $m\in\NN$. Then $\Utg_{\Gamma}(\widehat{\cX})$ is the completion
of $\Utg_{\Gamma}(\cX)$ along $\Utg_{\Gamma}(\cY_{(i)})$ for any
$i\ge\sharp G-1$.
\end{lem}

\begin{proof}
The stack $\Utg_{\Gamma}(\widehat{\cX})$ is the limit of $\Utg_{\Gamma}(\cY_{(n)})$,
$n\in\NN$, while the completion of $\Utg_{\Gamma}(\cX)$ along $\Utg_{\Gamma}(\cY_{(i)})$
is the limit of $\Utg_{\Gamma}(\cY_{(i)})_{(n)}$, $n\in\NN$. It
suffices to show that for every $n$, there exists $n'$ such that
\begin{gather}
\Utg_{\Gamma}(\cY_{(n)})\subset\Utg_{\Gamma}(\cY_{(i)})_{(n')}\text{ and}\label{eq:two-inclu}\\
\Utg_{\Gamma}(\cY_{(i)})_{(n)}\subset\Utg_{\Gamma}(\cY_{(n')}).\label{eq:two-inclu2}
\end{gather}

To show these, we may suppose that $\Gamma=\Theta_{[G]}$, in particular,
we may suppose that $\Gamma$ is reduced. Since $\cY\subset\cX_{m}$,
the morphism $\Utg_{\Gamma}(\cY_{(n)})_{\red}\to\Df_{\Gamma}$ factors
through $\D_{\Gamma,0}$ and
\[
\Utg_{\Gamma}(\cY_{(n)})_{\red}\times_{\Df_{\Gamma}}\cE_{\Gamma}\cong\Utg_{\Gamma}(\cY_{(n)})_{\red}\times_{\D_{\Gamma,0}}\cE_{\Gamma,0}.
\]
Let $\cN$ be the nilradical of the structure sheaf of $\Utg_{\Gamma}(\cY_{(n)})_{\red}\times_{\D_{\Gamma,0}}\cE_{\Gamma,0}$,
which is the pullback of the nilradical $\cN'$ of the structure sheaf
of $\cE_{\Gamma,0}$. Since $(\cN')^{\sharp G}=0$, $\cN^{\sharp G}=0$.
Consider the universal morphisms 
\begin{align*}
u\colon\Utg_{\Gamma}(\cY_{(n)})_{\red}\times_{\Df_{\Gamma}}\cE_{\Gamma} & \to\cY_{(n)}.
\end{align*}
If $\cI$ denotes the defining ideal of $\cY\subset\cY_{(n)}$, which
is nilpotent, then $u^{-1}\cI\subset\cN$. It follows that $u^{-1}\cI^{\sharp G}=0$.
This means that $u$ factors through $\cY_{(i)}$ as far as $i\ge\sharp G-1$.
In turn, this shows that 
\[
\Utg_{\Gamma}(\cY_{(n)})_{\red}\subset\Utg_{\Gamma}(\cY_{(i)}).
\]
Inclusion (\ref{eq:two-inclu}) follows.

Inclusion (\ref{eq:two-inclu2}) holds in a more general situation.
For any closed substack $\cZ\subset\cX$, consider the 2-commutative
diagram:
\[
\xymatrix{\Utg_{\Gamma}(\cZ)\times_{\Df_{\Gamma}}\cE_{\Gamma}\ar[r]\ar@{^{(}->}[d]\ar[dr]\sp(0.7){b} & \cZ\ar@{^{(}->}[d]\\
\Utg_{\Gamma}(\cZ)_{(n)}\times_{\Df_{\Gamma}}\cE_{\Gamma}\ar[r]\sb(0.8){c} & \cX
}
\]
Let $\cJ$ and $\cI_{\cZ}$ be the defining ideals of the left and
right vertical arrows respectively. We have $b^{-1}\cI_{\cZ}=0$,
$c^{-1}\cI_{\cZ}\subset\cJ$ and $\cJ^{n+1}=0$. Therefore $c^{-1}\cI_{\cZ}^{n+1}=0$.
This shows that $c$ factors through $\cZ_{(n)}$. Therefore $\Utg_{\Gamma}(\cZ)_{(n)}\subset\Utg_{\Gamma}(\cZ_{(n)})$.
In particular, 
\[
\Utg_{\Gamma}(\cY_{(i)})_{(n)}\subset\Utg_{\Gamma}\left((\cY_{(i)})_{(n)}\right)=\Utg_{\Gamma}(\cY_{((i+1)(n+1)-1)}).
\]
\end{proof}

\subsection{Rig-pure formal stacks and flattening stratification}
\begin{defn}[{\cite[page 43]{MR2815110}}]
Let $A$ be a $k\tbrats$-algebra. We say that $A$ is \emph{rig-pure
}if one of the following equivalent conditions holds:
\begin{enumerate}
\item The element $t\in A$ is not a zerodivisor.
\item The localization map $A\to A_{t}$ is injective.
\item The scheme-theoretic closure of the open subscheme $\Spec A_{t}\subset\Spec A$
in $\Spec A$ is $\Spec A$.
\end{enumerate}
For a $k\llbracket t\rrbracket$-algebra $A$, the \emph{associated
rig-pure algebra} of $A$, denoted by $A^{\pur}$, is defined to be
the image of the map $A\to A_{t}$.
\end{defn}

This construction is characterized by the universality: the map $A\to A^{\pur}$
is the universal map from $A$ to a rig-pure algebra.
\begin{lem}
\label{lem:calm-flat}Let $A$ be an $R\llbracket t\rrbracket$-algebra
and $A\to B$ a flat homomorphism of rings. Then there exists a natural
isomorphism $B^{\pur}\cong A^{\pur}\otimes_{A}B$. In particular,
$B$ is rig-pure if $A$ is rig-pure, and the converse is also true
if $A\to B$ is faithfully flat.
\end{lem}

\begin{proof}
This is \cite[5.1.13 and 5.11.9]{MR2815110}. For the sake of completeness,
we give a proof. Let $C$ be an $A$-algebra which is rig-pure. Then
the map $C\to C_{t}$ is injective. Applying $\otimes_{A}B$, we get
a map $C\otimes_{A}B\to C_{t}\otimes_{A}B=(C\otimes_{A}B)_{t}$. Since
$A\to B$ is flat, this is injective, which means that $C\otimes_{A}B$
is rig-pure. We apply this to $C=A^{\pur}$ and get that $A^{\pur}\otimes_{A}B$
is rig-pure. From the commutative diagram 
\[
\xymatrix{A\ar[r]\ar[d] & B\ar[d]\\
A^{\pur}\ar[r] & B^{\pur}
}
\]
we get a map $A^{\pur}\otimes_{A}B\to B^{\pur}$. From the universality
of $B^{\pur}$, this map has the inverse and the first assertion follows.

If $A$ is rig-pure, then we have $A\cong A^{\pur}$, which implies
$B\cong B^{\pur}$, that is, $B$ is rig-pure. Conversely, if $B$
is rig-pure, then $B=A\otimes_{A}B\to A^{\pur}\otimes_{A}B\cong B^{\pur}$
is an isomorphism. Now, if $A\to B$ is faithfully flat, this implies
that $A\to A^{\pur}$ is also an isomorphism and $A$ is rig-pure.
\end{proof}
\begin{lem}
\label{lem:flat-calm-base-change-1}Let $A$ be a $R\tbrats$-algebra.
Suppose that $A^{\pur}$ is flat over $R\tbrats$. Then, for any ring
map $R\to S$, $A^{\pur}\otimes_{R\tbrats}S\tbrats\cong(A\otimes_{R\tbrats}S\tbrats)^{\pur}$.
\end{lem}

\begin{proof}
As a base change of $A\to A^{\pur}\to A_{t},$ we get $A_{S}\to(A^{\pur})_{S}\to(A_{S})_{t},$
where the subscript $S$ means the base change by $R\tbrats\to S\tbrats$.
Therefore we get surjective maps $A_{S}\to(A^{\pur})_{S}\to(A_{S})^{\pur}$.
If $(A_{S})_{t\textrm{-tor}}$ denotes the submodule of $t$-torsions
of $A_{S}$, we have
\[
(A_{S})_{t\textrm{-tor}}=\Ker(A_{S}\to(A_{S})^{\pur})\supset\Ker(A_{S}\to(A^{\pur})_{S})\supset(A_{S})_{t\textrm{-tor}}.
\]
The last inclusion holds since $(A^{\pur})_{S}$ is flat over $S\tbrats$
as well as over $k\tbrats$ and has no $t$-torsion. Thus the two
kernels are equal. It follows that $(A^{\pur})_{S}\to(A_{S})^{\pur}$
is an isomorphism. We have proved the lemma.
\end{proof}
\begin{cor}
Let $\cX$ be a formal DM stack of finite type over $\Df$. Then there
exists a unique closed substack $\cX^{\pur}\subset\cX$ such that
for any flat finite-type morphism $\Spf A\to\cX$, the scheme-theoretic
preimage of $\cX^{\pur}$ in $\Spf A$ is $\Spf A^{\pur}$.
\end{cor}

\begin{proof}
Let $\Spf B\to\cX$ be an atlas. From \ref{lem:formal-finite-type},
\[
\Spf B\times_{\cX}\Spf B\left(\cong\varinjlim(\Spf B)_{n}\times_{\cX_{n}}(\Spf B)_{n}\right)
\]
is a formal affine scheme of finite type; we write this as $\Spf C$.
The groupoid $\Spf C\rightrightarrows\Spf B$ induces the groupoid
$\Spf C^{\pur}\rightrightarrows\Spf B^{\pur}$. Let $\cX^{\pur}$
be the stack associated to the last groupoid, which is regarded as
a closed substack of $\cX$. We show that this satisfies the desired
property. Let $\Spf A\to\cX$ be a flat finite-type morphism. Let
$\Spf R:=\Spf A\times_{\cX}\Spf B$. The preimage of $\cX^{\pur}$
in $\Spf B$ is $\Spf B^{\pur}$ by construction. Since $B\to R$
is flat, from \ref{lem:calm-flat}, the preimage of $\cX^{\pur}$
in $\Spf R$ is $\Spf R^{\pur}$. Since $\Spf R\to\Spf A$ is etale
and surjective, a closed subscheme of $\Spf A$ is uniquely determined
by its preimage in $\Spf R$. The preimage of $\cX^{\pur}$ in $\Spf A$
and the closed subscheme $\Spf A^{\pur}\subset\Spf A$ are both pulled
back to $\Spf R^{\pur}$. Therefore the two closed subschemes are
the same. The uniqueness of $\cX^{\pur}$ is clear.
\end{proof}
\begin{defn}
\label{def:pure-stack}We call $\cX^{\pur}$ in the last corollary
the \emph{associated rig-pure stack }of $\cX$.
\end{defn}

\begin{lem}
\label{lem:flat-calm-base-change}Let $\Sigma$ be a DM stack of finite
type and let $\cZ\to\Df_{\Sigma}$ be a morphism of formal DM stacks
of finite type over $\Df_{\Sigma}$. Suppose that $\cZ^{\pur}$ is
flat over $\Df_{\Sigma}$. Then, for any morphism of DM stacks $\Sigma'\to\Sigma$,
we have $(\cZ_{\Sigma'})^{\pur}=(\cZ^{\pur})_{\Sigma'}$.
\end{lem}

\begin{proof}
This is a direct consequence of \ref{lem:flat-calm-base-change-1}.
\end{proof}
\begin{lem}
\label{lem:flat-locus}Let $\Sigma$ be a reduced DM stack of finite
type and let $\cX$ be a rig-pure formal stack of finite type over
$\Df_{\Sigma}$. Then there exists an open dense substack $\Sigma^{\circ}\subset\Sigma$
such that the induced morphism $\cX\times_{\Sigma}\Sigma^{\circ}\to\Df_{\Sigma^{\circ}}$
is flat.
\end{lem}

\begin{proof}
By the base change along an atlas $S\to\Sigma$, we can reduce the
lemma to the case where $\Sigma$ is a scheme. Restricting to an irreducible
component of $\Sigma$, we may further assume that $\Sigma$ is integral.
Let $\eta\to\Sigma$ be its generic point. Since $\cX$ is rig-pure
over $\Df_{\Sigma}$, $\cX\times_{\Sigma}\eta\to\Df_{\eta}$ is flat.
Let $U\subset|\cX|$ be the flat locus of the morphism $\cX\to\Df_{\Sigma}$,
which is an open subset from \cite[5.3.10]{MR2815110}. The image
of $|\cX|\setminus U$ in $|\Sigma|=|\Df_{\Sigma}|$ is a constructible
subset which does not contain the generic point. Removing the closure
of this image, we obtain an open dense subscheme $\Sigma^{\circ}\subset\Sigma$
having the desired property.
\end{proof}
\begin{cor}
\label{cor: flattening-strat}Let $\Sigma$ be a DM stack of finite
type and let $\cZ$ be a formal DM stack of finite type over $\Df_{\Sigma}$.
Then there exist finitely many locally closed reduced substacks $\Sigma_{i}\hookrightarrow\Sigma$,
$i\in I$ such that $|\Sigma|=\bigsqcup_{i\in I}|\Sigma_{i}|$ and
for every $i\in I$, the associated rig-pure stack $(\cZ_{\Sigma_{i}})^{\pur}$
of $\cZ_{\Sigma_{i}}:=\cZ\times_{\Sigma}\Sigma_{i}$ is flat over
$\Df_{\Sigma_{i}}$.
\end{cor}

\begin{proof}
This follows from \ref{lem:flat-locus} and the fact that $|\Sigma|$
is a Noetherian topological space.
\end{proof}
\begin{defn}
\label{nota:Gamma}For a formal DM stack $\cX$ of finite type over
$\Df$ and for each $\Theta_{[G]}^{[\br]}$, we take locally closed
reduced substacks $\Theta_{[G],i}^{[\br]}\subset\Theta_{[G]}^{[\br]}$
such that $\left|\Theta_{[G]}^{[\br]}\right|=\bigsqcup_{i}\left|\Theta_{[G],i}^{[\br]}\right|$
and $\Utg_{\Theta_{[G]}^{[\br]},i}(\cX)^{\pur}$ is flat over $\Df_{\Theta_{[G]}^{[\br]},i}$.
We write $\Gamma_{\cX}=\Gamma:=\coprod_{[G],[\br],i}\Theta_{[G],i}^{[\br]}$.
For each $G$, we write $\Gamma_{[G]}:=\coprod_{[\br],i}\Theta_{[G],i}^{[\br]}$
and $\Gamma_{G}:=\Gamma_{[G]}\times_{\cA_{[G]}}\Spec k$.
\end{defn}

Thus $\Gamma=\Gamma_{\cX}$ is a DM stack of finite type given with
a morphism $\Gamma\to\Theta$ such that the map $|\Gamma|\to|\Theta|$
as well as maps $\Gamma[K]\to\Theta[K]$ for algebraically closed
fields $K$ is bijective and $\Utg_{\Gamma}(\cX)^{\pur}$ is flat
over $\Df_{\Gamma}$.

For a tame cyclic group $C_{l}$, we have 
\begin{gather*}
\Gamma_{[C_{l}]}=\Theta_{[C_{l}]}=\Lambda_{[C_{l}]}\text{ and}\\
\Gamma_{C_{l}}=\Theta_{C_{l}}=\Lambda_{C_{l}}.
\end{gather*}
In particular, $\Gamma$ contains $\Lambda_{\tame}$ as an open and
closed substack. 
\begin{lem}
\label{lem:rel-dim}Suppose that $\cX$ is flat over $\Df$ and has
pure relative dimension. Let $\Gamma=\Gamma_{\cX}$ be as above. Then
$\Utg_{\Gamma}(\cX)^{\pur}$ has pure relative dimension over $\Df_{\Gamma}$
equal to the relative dimension of $\cX$ over $\Df$.
\end{lem}

\begin{proof}
This follows from from \ref{lem:rig-etale-2}.
\end{proof}
\begin{prop}
\label{prop:Unt-morphism}Let $\cY$ be a formal DM stack of finite
type over $\Df$ and $X$ a formal algebraic space over $\Df$. Let
$f\colon\cY\to X$ be a $\Df$-morphism and let $\Gamma=\Gamma_{\cY}$
be as in \ref{nota:Gamma}. For each locally finite-type morphism
of DM stacks $\Sigma\to\Gamma$, there exists a $\Df$-morphism 
\[
f_{\Sigma}^{\utg}\colon\Utg_{\Sigma}(\cY)^{\pur}\to X
\]
satisfying the following property:
\begin{enumerate}
\item If $\Spf A\to\Utg_{\Sigma}(\cY)^{\pur}$ is a flat morphism corresponding
to a representable morphism $\cE_{\Sigma}\times_{\Df_{\Sigma}}\Spf A\to\cY$,
then the morphism $\Spf A\to\overline{\cY}\to X$ induced from this
by taking coarse moduli spaces is the same as the composition $\Spf A\to\Utg_{\Sigma}(\cY)^{\pur}\to X$.
Here $\overline{\cY}$ is the coarse moduli space of $\cY$. (Note
that since $\Spf A\to\Df_{\Sigma}$ is flat, from \ref{lem:coarse-flat-base-change},
the coarse moduli space of $\cE_{\Sigma}\times_{\Df_{\Sigma}}\Spf A$
is $\Spf A$.)
\item For a locally finite-type morphism $\Sigma'\to\Sigma$ of DM stacks,
the composition 
\[
\Utg_{\Sigma'}(\cY)^{\pur}\to\Utg_{\Sigma}(\cY)^{\pur}\xrightarrow{f_{\Sigma}^{\utg}}X
\]
is the same as $f_{\Sigma'}^{\utg}$.
\end{enumerate}
\end{prop}

\begin{proof}
Let $W\to\Gamma$ be an atlas, let $\cU:=\Utg_{\Gamma}(\cY)^{\pur}$
and $\cU_{W}:=\Utg_{W}(\cY)^{\pur}$ and let $V_{0}\to\cU_{W}$ be
an atlas from a formal scheme and $V_{1}:=V_{0}\times_{\cU}V_{0}$
so that we have the groupoid in formal schemes $V_{1}\rightrightarrows V_{0}$.
The morphisms $V_{i}\to\cU$, $i=0,1$, correspond to representable
$\Df$-morphisms $\gamma_{i}\colon\cE_{\Gamma}\times_{\Df_{\Gamma}}V_{i}=\cE_{W}\times_{\Df_{W}}V_{i}\to\cY$
respectively. The coarse moduli space of $\cE_{W}$ is $\Df_{W}$.
Since $V_{i}\to\Df_{W}$ is flat, from \ref{lem:coarse-flat-base-change},
the moduli space of $\cE_{\Gamma}\times_{\Df_{\Gamma}}V_{i}$ is $V_{i}$.
Therefore the morphisms $f\circ\gamma_{i}\colon\cE_{\Gamma}\times_{\Df_{\Gamma}}V_{i}\to X$
induce morphisms $\beta_{i}\colon V_{i}\to X$. The morphisms $\beta_{i}$
define a morphism $(V_{1}\rightrightarrows V_{0})\to(X\rightrightarrows X)$
of groupoids. Finally this induces the morphism of the associated
stacks,
\[
\cU=[V_{1}/V_{0}]\to X=[X/X].
\]
For any morphism $\Sigma\to\Gamma$, by the base change, the groupoid
$V_{1}\rightrightarrows V_{0}$ induces a groupoid $V_{1,\Sigma}\rightrightarrows V_{0,\Sigma}$
whose associated stack is $\Utg_{\Sigma}(\cY)^{\pur}$. We define
$\Utg_{\Sigma}(f)$ to be the one associated to the morphism of groupoids,
$(V_{1,\Sigma}\rightrightarrows V_{0,\Sigma})\to(X\rightrightarrows X)$.
We easily see that these morphisms satisfy the desired properties.
\end{proof}
\begin{rem}
The reason why we need a flattening stratification as in \ref{cor: flattening-strat}
and \ref{nota:Gamma} is that the coarse moduli space does not commute
with non-flat base changes in the case of wild DM stacks. This prevents
from having a natural morphism $\Utg_{\Theta}(\cY)\to\Utg_{\Theta}(X)$.
For instance, a $\D_{R,n}$-point of $\Utg_{\Theta}(\cY)$ over $\D_{R,n}\hookrightarrow\Df_{R}$,
a typical nonflat morphism, corresponds to a representable morphism
\[
\cE_{n}=[E_{n}/G]\to\cY
\]
for some $G$-covering $E\to\Df_{R}$. But the $G$-action on $E_{n}$
may be trivial and the coarse moduli space of $\cE_{n}$ may be $E_{n}$
rather than $\D_{R,n}$. Thus we would get a morphism $E_{n}\to X$
rather than a desired $\D_{R,n}\to X$.
\end{rem}

\begin{rem}
\label{rem:utg-map-general}For a morphism $\cY\to\cX$ of formal
DM stacks of finite type over $\Df$ and for a suitable choice of
$\Gamma_{\cY}$ and $\Gamma_{\cX}$, we would like to have a morphism
\[
\Utg_{\Gamma_{\cY}}(\cY)\to\Utg_{\Gamma_{\cX}}(\cX).
\]
However it was technically difficult and we are content with restricting
ourselves to the case where $\cX$ is a formal algebraic space.
\end{rem}

\section{Stacks of jets and arcs\label{sec:jets-arcs}}

In this section, we first define stacks of untwisted jets and arcs
on formal DM stacks. Next we define stacks of twisted jets and arcs,
using untwisting stacks and stacks of untwisted jets and arcs.
\begin{defn}
\label{def:untwisted-jets}Let $\Phi$ be a DM stack almost of finite
type and $\cX$ a formal DM stack of finite type over $\Df_{\Phi}$.
For $n\in\NN$, we define the \emph{stack of (untwisted) $n$-jets
on $\cX$, }denoted by $\J_{\Phi,n}\cX$ or $\J_{n}\cX$, to be the
Weil restriction $\R_{\D_{\Phi,n}/\Phi}(\cX_{n})$. Namely $\J_{n}\cX$
is a fibered category over $\Phi$ such that the fiber category over
an $S$-point $s\colon S\to\Phi$ is
\[
(\J_{n}\cX)(s)=\Hom_{\D_{S,n}}(\D_{S,n},\cX_{n}\times_{\Phi}S)=\Hom_{\Df_{\Phi}}(\D_{S,n},\cX).
\]
\end{defn}

\begin{lem}
The $\J_{n}\cX$ is a DM stack almost of finite type and the morphism
$\J_{n}\cX\to\Phi$ is of finite type.
\end{lem}

\begin{proof}
We claim 
\[
\J_{n}\cX\cong\ulHom_{\Phi}(\D_{\Phi,n},\cX_{n})\times_{\ulHom_{\Phi}(\D_{\Phi,n},\D_{\Phi,n})}\Phi,
\]
where $\Phi\to\ulHom_{\Phi}(\D_{\Phi,n},\D_{\Phi,n})$ is the morphism
induced by the identity morphism $\Df_{\Phi,n}\to\Df_{\Phi,n}$. For
an object $s\in\D_{\Phi,n}(S)$, if $\{s\}$ denotes the category
consisting of $s$ and its identity morphism, then
\begin{align*}
 & (\ulHom_{\Phi}(\D_{\Phi,n},\cX_{n})\times_{\ulHom_{\Phi}(\D_{\Phi,n},\D_{\Phi,n})}\Phi)(s)\\
 & \cong\Hom_{S}(\D_{S,n},\cX_{n}\times_{\Phi}S)\times_{\Hom_{S}(\D_{S,n},\D_{S,n})}\{s\}\\
 & \cong\Hom_{\D_{S,n}}(\D_{S,n},\cX_{n}\times_{\Phi}S)\\
 & \cong(\J_{n}\cX)(s).
\end{align*}
Thus the claim holds. Applying \ref{lem:Hom-finite-stack} to each
connected component of $\Phi$, we get the lemma.
\end{proof}
For $n=0$, we have $\J_{\Phi,0}\cX=\cX_{0}$. For $n\ge m$, the
closed immersion $\D_{m}\hookrightarrow\D_{n}$ induces a morphism
$\pi_{m}^{n}\colon\J_{n}\cX\to\J_{m}\cX$.
\begin{lem}
The morphism $\pi_{m}^{n}$ is representable and affine.
\end{lem}

\begin{proof}
For an atlas $U\to\Phi$, we have $(\J_{\Phi,n}\cX)\times_{\Phi}U\cong\J_{U,n}(\cX\times_{\Phi}U)$.
Therefore we may suppose that $\Phi=U$ is a scheme. For an atlas
$V\to\cX$, we have $\left(\J_{U,n}\cX\right)\times_{\cX}V\cong\J_{U,n}V$
as in the case of etale morphisms of schemes. Indeed, for the 2-commutative
diagram of solid arrows
\[
\xymatrix{\Spec R\ar@{^{(}->}[d]\ar[r] & V\ar[d]\\
\Spec R\tbrats/(t^{n+1})\ar[r]\ar@{-->}[ur] & \cX
}
\]
there exists a unique dash arrow fitting into the diagram. This shows
the last isomorphism. Thus we may also suppose that $\cX=V$ is a
formal scheme. The problem is now reduced to the case of formal schemes.
The lemma is well-known in this case (for instance, see \cite[page 249]{MR3838446}).
\end{proof}
From this lemma, we get a projective system $(\J_{n}\cX,\pi_{m}^{n})$
of DM stacks almost of finite type having affine morphisms as transition
morphisms. Therefore the projective limit 
\[
\J_{\Phi,\infty}\cX=\J_{\infty}\cX:=\varprojlim\J_{n}\cX
\]
 exists and is a DM stack.
\begin{defn}
We call $\J_{\infty}\cX$ the \emph{stack of (untwisted) arcs }on
$\cX$.
\end{defn}

A point of $|\J_{\infty}\cX|$ corresponds to an equivalence class
of pairs of a geometric point $\Spec K\to\Phi$ and a lift $\Df_{K}\to\cX$
of the morphism $\Df_{K}\to\Df_{\Phi}$.

Stacks of untwisted jets are functorial in $\cX$. Let $\Psi$ and
$\Phi$ be DM stacks almost of finite type and let $\cY$ and $\cX$
be formal DM stacks of finite type over $\Df_{\Psi}$ and $\Df_{\Phi}$
respectively. Let $\Psi\to\Phi$ be a morphism and let $f\colon\cY\to\cX$
be a morphism compatible with the induced morphism $\Df_{\Psi}\to\Df_{\Phi}$.
Then we obtain a morphism 
\[
f_{n}\colon\J_{\Psi,n}\cY\to\J_{\Phi,n}\cX
\]
compatible with $\Psi\to\Phi$: This sends an $S$-point of $\J_{\Psi,n}\cY$
given by a pair $(S\to\Psi,\D_{S,n}\to\cY)$ to the induced pair $(S\to\Psi\to\Phi,\D_{S,n}\to\cY\to\cX)$.
\begin{defn}
A \emph{twisted arc of $\cX$ }is a representable morphism $\cE\to\cX$
where $\cE$ is a twisted formal disk over some algebraically closed
field. We say that two twisted arcs $\cE\to\cX$ and $\cE'\to\cX$
say over $K$ and $K'$ are \emph{equivalent }if there exist morphisms
$\Spec K''\to\Spec K$ and $\Spec K''\to\Spec K'$ with $K''$ an
algebraically closed field and an isomorphism $\cE_{K''}\to\cE'_{K''}$
such that 
\[
\xymatrix{\cE_{K''}\ar[r]\ar[dr] & \cE_{K''}'\ar[d]\\
 & \cX
}
\]
is 2-commutative.
\end{defn}

\begin{defn}
Let $\cX$ be a formal DM stack of finite type over $\Df$ and let
$\Gamma=\Gamma_{\cX}$ be as in \ref{nota:Gamma}. For $n\in\NN$,
we define the \emph{stack of twisted $n$-jets }on $\cX$, denoted
by $\cJ_{\Gamma,n}\cX$ or $\cJ_{n}\cX$, to be $\J_{\Gamma,n}\Utg_{\Gamma}(\cX)^{\pur}$.\emph{
}We define the \emph{stack of twisted arcs }on $\cX$, denoted by
$\cJ_{\Gamma,\infty}\cX$ or $\cJ_{\infty}\cX$, to be $\J_{\Gamma,\infty}\Utg_{\Gamma}(\cX)^{\pur}$.
\end{defn}

Let $K$ be an algebraically closed field, let $E\to\Df_{K}$ a $G$-cover
for some Galoisian group $G$ corresponding to a $K$-point $e\colon\Spec K\to\Gamma$
and let $\cE=[E/G]$ be the induced twisted formal disk. We have
\begin{align*}
(\cJ_{\infty}\cX)(e) & =\varprojlim\Hom_{\Df}(\D_{K,n},\Utg_{\Spec K}(\cX)^{\pur})\\
 & =\Hom_{\Df}(\Df_{K},\Utg_{\Spec K}(\cX)^{\pur})\\
 & =\Hom_{\Df}(\Df_{K},\Utg_{\Spec K}(\cX))\\
 & =\varprojlim\Hom_{\Df}(\D_{K,n},\Utg_{\Spec K}(\cX))\\
 & =\varprojlim\Hom_{\Df}^{\rep}(\cE_{n},\cX)\\
 & =\Hom_{\Df}^{\rep}(\cE,\cX).
\end{align*}
It follows that the point set $|\cJ_{\infty}\cX|$ is in a one-to-one
correspondence with the set of equivalence classes of twisted arcs
of $\cX$. This also shows that $|\cJ_{\infty}\cX|$ is independent
of the choice of $\Gamma$.

Over the component $\Spec k\cong\Lambda_{[1]}\subset\Gamma$ corresponding
to the trivial Galoisian group $G=1$, we have 
\[
\Utg_{\Lambda_{[1]}}(\cX)=\ulHom_{\Df}^{\rep}(\Df,\cX)=\cX.
\]
Therefore the stack of untwisted arcs $\J_{\infty}\cX$ is regarded
as an open and closed substack of $\cJ_{\infty}\cX$.

If $f\colon\cY\to X$ is a morphism from a formal DM stack $\cY$
to a formal algebraic space $X$ which are both of finite type over
$\Df$, then we have the map 
\[
f_{\infty}\colon|\cJ_{\infty}\cY|\to|\J_{\infty}X|
\]
which sends a class of a twisted arc $\cE\to\cY$ to the arc $\Df_{K}\to X$
obtained from the composite $\cE\to\cY\to X$ and the universality
of the coarse moduli  space $\cE\to\Df_{K}$. This map is nothing
but the map 
\[
(f^{\utg})_{\infty}\colon|\J_{\infty}\Utg_{\Gamma}(\cY)|\to|\J_{\infty}X|
\]
which is associated to $f^{\utg}\colon\Utg_{\Gamma}(\cY)\to X$ obtained
in \ref{prop:Unt-morphism}.

\section{\label{sec:Jacobians}Jacobian order functions}

In this section, we define Jacobian order functions denoted by $\fj$
of formal DM stacks and of morphisms between them.

Let $\Phi$ be a DM stack of finite type and let $\cX$ be a formal
DM stack flat and of finite type over $\Df_{\Phi}$.
\begin{defn}
\label{def:order}For a field $K$, let 
\[
\ord\colon K\tpars\to\ZZ\cup\{\infty\}
\]
be the order function (the normalized additive valuation) with convention
$\ord0=\infty$. For a finite extension $L/K\tpars$ of degree $r$,
we continue to denote by $\ord$ the unique extension of $\ord$ to
$L$, which takes values in $\frac{1}{r}\ZZ\cup\{\infty\}$. Let $O_{L}$
be the integral closure of $K\tbrats$ in $L$. For an ideal $I$
of $O_{L}$ say generated by $a$, we define $\ord I:=\ord a$.
\end{defn}

\begin{defn}
\label{def:order-function}Let $\cI$ be a coherent ideal sheaf on
(the small etale site of) $\cX$. For a finite extension $L/K\tpars$
and a $\Df$-morphism $\beta\colon\Spf O_{L}\to\cX$, the pull-back
$\beta^{-1}\cI$ is an ideal of $O_{L}$ and we define 
\[
\ord_{\cI}(\beta):=\ord\beta^{-1}\cI\in\frac{1}{r}\ZZ\cup\{\infty\}.
\]
For a twisted arc $\gamma\colon\cE\to\cX$ with $\cE$ defined over
$K$, let us take a finite $\Df_{K}$-morphism $\Spf O_{L}\to\cE$
for some finite extension $L/K\tpars$ and let $\beta\colon\Spf O_{L}\to\cE\to\cX$
be the composition. We define 
\[
\ord_{\cI}(\gamma):=\ord_{\cI}(\beta).
\]
\end{defn}

The last definition is independent of the choice of $\Spf O_{L}\to\cE$.
This follows from the fact that if $\beta\colon\Spf O_{L}\to\cX$
is a morphism as above and $\beta'\colon\Spf O_{L'}\to\Spf O_{L}\to\cX$
is the one induced from $\beta$ and a finite extension $L'/L$, then
$\ord_{\cI}(\beta)=\ord_{\cI}(\beta')$.

We use sheaves of differentials. This well-known notion for schemes
is generalized to formal schemes (for instance, see \cite{MR2313672,MR3838446,MR2815110}).
We generalize it further to formal DM stacks.
\begin{defn}
We define the \emph{sheaf of differentials} of $\cX$ over $\Df_{\Phi}$,
denoted by $\Omega_{\cX/\Df_{\Phi}}$, as follows. For each atlas
$u\colon U\to\cX$ from a formal scheme which fits into a 2-commutative
diagram
\[
\xymatrix{U\ar[d]\ar[r] & \cX\ar[d]\\
\Df_{T}\ar[r] & \Df_{\Phi}
}
\]
for some atlas $T\to\Phi$, the sheaf $\Omega_{U/\Df_{T}}$ is independent
of the choice of $T\to\Phi$; we denote it by $\Omega_{U/\Df_{\Phi}}$.
For such an atlas $U\to\cX$, the morphism $U\times_{\cX}U\to\cX$
is also an atlas satisfying the same condition. For projections $p_{1,}p_{2}\colon U\times_{\cX}U\rightrightarrows U$,
we have canonical isomorphisms
\[
p_{1}^{*}\Omega_{U/\Df_{\Phi}}\cong\Omega_{U\times_{\cX}U/\Df_{\Phi}}\cong p_{2}^{*}\Omega_{U/\Df_{\Phi}}.
\]
We define $\Omega_{\cX/\Df_{\Phi}}$ to be the coherent $\cO_{\cX}$-module
associated to these data.
\end{defn}

\begin{defn}
Suppose that $\cX$ has pure relative dimension $d$ over $\Df_{\Phi}$.
We define the \emph{Jacobian ideal sheaf }$\Jac_{\cX/\Df_{\Phi}}\subset\cO_{\cX}$
to be the $d$-th Fitting ideal sheaf of $\Omega_{\cX/\Df_{\Phi}}$.
We denote the associated order function $\ord_{\Jac_{\cX/\Df_{\Phi}}}$
on $|\cJ_{\infty}\cX|$ by $\fj_{\cX}=\fj_{\cX/\Phi}$. The \emph{singular
locus} of $\cX$ over $\Phi$ is the closed substack defined by $\Jac_{\cX/\Df_{\Phi}}$
and denoted by $\cX_{\sing}$.
\end{defn}

\begin{defn}
\label{def:thin}We say that for a formal DM stack $\cX$ of finite
type and flat over $\Df_{K}$ with $K$ a field, a closed substack
$\cY\subset\cX$ is \emph{narrow }if for every geometric point $y\colon\Spec L\to\cY$,
we have $\dim\cO_{\cY,y}<\dim\cO_{\cX,y}$. When $\Phi$ is a DM stack
almost of finite type and $\cX$ is a formal DM stack of finite type
and flat over $\Df_{\Phi}$, we say that a closed substack $\cY\subset\cX$
is \emph{narrow }if for every geometric point $\Spec K\to\Phi$, $\cY\times_{\Phi}\Spec K$
is a narrow closed substack of $\cX\times_{\Phi}\Spec K$.
\end{defn}

\begin{defn}
We say that $\cX$ is \emph{fiberwise generically smooth over $\Df_{\Phi}$
}if $\cX_{\sing}\subset\cX$ is narrow.
\end{defn}

\begin{rem}
When $\Phi=\Spec k$ and $\cX$ is a formal affine scheme $\Spf A$,
then $\cX$ is said to be \emph{rig-smooth} \cite[1.14.7]{MR2815110}
if we have the inclusion $V(\Jac_{\cX/\Df})\subset V(t)$ of closed
subsets of $\Spec A$, where we regarded $\Jac_{\cX/\Df}$ as an ideal
of $A$. When $\cX$ is rig-pure, then ``rig-smooth'' implies ``generically
smooth,'' but the converse is not true.
\end{rem}

If $\cX$ is fiberwise generically smooth and has pure relative dimension
$d$ over $\Df_{\Phi}$, then $\Omega_{\cX/\Df_{\Phi}}$ has generic
rank $d$ and $\Omega_{\cX/\Df_{\Phi}}^{d}:=\bigwedge^{d}\Omega_{\cX/\Df_{\Phi}}$
has generic rank 1. (By this, we mean that for every etale morphism
$u\colon\Spf A\to\cX$, the associated $A$-module has the indicated
generic rank.)
\begin{defn}
Let $\Psi\to\Phi$ be a morphism of DM stacks almost of finite type.
Let $\cY,\cX$ be formal DM stacks of finite type, flat and fiberwise
generically smooth over $\Df_{\Psi}$ and $\Df_{\Phi}$ respectively,
both of which have pure relative dimension $d$. Let $f\colon\cY\to\cX$
be a morphism compatible with $\Df_{\Psi}\to\Df_{\Phi}$. Let $\gamma\colon\cE=[E/G]\to\cY$
be a twisted arc, let $\beta\colon E\to\cY$ be the induced morphism
and let $O_{E}$ be the coordinate ring of $E$. Suppose that $\fj_{\cY/\Psi}(\beta)<\infty$
and $\fj_{\cX/\Phi}(f\circ\beta)<\infty$. Let 
\begin{gather*}
\beta^{\flat}\Omega_{\cY/\Df_{\Psi}}^{d}:=\left(\beta^{*}\Omega_{\cY/\Df_{\Psi}}^{d}\right)/\tors,\\
(f\circ\beta)^{\flat}\Omega_{\cX/\Df_{\Phi}}^{d}:=\left((f\circ\beta)^{*}\Omega_{\cX/\Df_{\Phi}}^{d}\right)/\tors
\end{gather*}
be flat pullbacks, which are free $O_{E}$-modules of rank one. We
define 
\[
\fj_{f}(\gamma)=\fj_{f}(\beta):=\frac{1}{\sharp G}\length\frac{\beta^{\flat}\Omega_{\cY/\Df_{\Psi}}^{d}}{\Image\left((f\circ\beta)^{\flat}\Omega_{\cX/\Df_{\Phi}}^{d}\to\beta^{\flat}\Omega_{\cY/\Df_{\Psi}}^{d}\right)}\in\NN\cup\{\infty\}.
\]
When either $\fj_{\cY/\Psi}(\beta)=\infty$ or $\fj_{\cX/\Phi}(f\circ\beta)=\infty$,
we define $\fj_{f}(\gamma)=\fj_{f}(\beta):=\infty$.
\end{defn}

Similarly we define
\begin{gather*}
\beta^{\flat}\Omega_{\cY/\Df_{\Psi}}:=\left(\beta^{*}\Omega_{\cY/\Df_{\Psi}}\right)/\tors,\\
(f\circ\beta)^{\flat}\Omega_{\cX/\Df_{\Phi}}:=\left((f\circ\beta)^{*}\Omega_{\cX/\Df_{\Phi}}\right)/\tors,
\end{gather*}
which are free $O_{E}$-modules of rank $d$. For suitable bases,
the map $(f\circ\beta)^{\flat}\Omega_{\cX/\Df_{\Phi}}\to\beta^{\flat}\Omega_{\cY/\Df_{\Psi}}$
is represented by a diagonal matrix $\mathrm{diag}(\pi^{a_{1}},\dots,\pi^{a_{d}})$,
where $\pi$ is a uniformizer of $O_{E}$ and we have $\sharp G\cdot\fj_{f}(\gamma)=\sum_{i=1}^{d}a_{i}$.
\begin{defn}
We define the \emph{Jacobian ideal sheaf} $\Jac_{f}\subset\cO_{\cY}$
of $f$ to be the $0$-th Fitting ideal of $\Omega_{\cY/(\cX\times_{\Phi}\Psi)}$.
\end{defn}

\begin{lem}
\label{lem:relation-naive-jac}We have 
\[
\fj_{f}\le\ord_{\Jac_{f}}.
\]
The equality holds if $\cY$ is smooth over $\Df_{\Psi}$.
\end{lem}

\begin{proof}
We first recall two properties of Fitting ideals. Firstly Fitting
ideals commute with base change \cite[Cor. 20.5]{MR1322960}. Secondly,
for a discrete valuation ring $R$ with maximal ideal $(\pi)$ and
for an $R$-module of rank $r$, $M=R^{\oplus r}\oplus\bigoplus_{i=1}^{l}R/(\pi^{a_{i}})$,
we see
\[
\mathrm{Fitt}_{r}(M)=(\pi^{\sum_{i=1}^{l}a_{i}})
\]
by a direct computation. From these properties, we have 
\begin{equation}
\sharp G\cdot\ord_{\Jac_{f}}(\beta)=\length\beta^{*}\Omega_{\cY/(\cX\times_{\Phi}\Psi)}.\label{eq:G-ord}
\end{equation}
Consider the exact sequence
\[
(f\circ\beta)^{*}\Omega_{\cX/\Df_{\Phi}}\xrightarrow{\alpha}\beta^{*}\Omega_{\cY/\Df_{\Psi}}\to\beta^{*}\Omega_{\cY/(\cX\times_{\Phi}\Psi)}\to0.
\]
The map $\alpha$ induces $(f\circ\beta)^{\flat}\Omega_{\cX/\Df_{\Phi}}\to\beta^{\flat}\Omega_{\cY/\Df_{\Psi}}$
by taking quotients by the torsion parts. If this is represented by
a diagonal matrix $\mathrm{diag}(\pi^{a_{1}},\dots,\pi^{a_{d}})$,
then since its cokernel is isomorphic to $\bigoplus_{i=1}^{d}O_{E}/(\pi^{a_{i}})$,
we have
\begin{align}
\sharp G\cdot\fj_{f}(\beta) & =\length\frac{\beta^{\flat}\Omega_{\cY/\Df_{\Psi}}}{\Image((f\circ\beta)^{\flat}\Omega_{\cX/\Df_{\Phi}}\to\beta^{\flat}\Omega_{\cY/\Df_{\Psi}})}.\label{eq:Gjf}
\end{align}
The last quotient module is a quotient of $\mathrm{Coker}(\alpha)\cong\gamma^{*}\Omega_{\cY/\cX\times_{\Phi}\Psi}$.
The first assertion follows from (\ref{eq:G-ord}) and (\ref{eq:Gjf}).
If $\cY\to\Df_{\Psi}$ is smooth, then $\gamma^{\flat}\Omega_{\cY/\Df_{\Psi}}=\gamma^{*}\Omega_{\cY/\Df_{\Psi}}$
and 
\[
\frac{\gamma^{\flat}\Omega_{\cY/\Df_{\Psi}}}{\Image((f\circ\beta)^{\flat}\Omega_{\cX/\Df_{\Phi}}\to\beta^{\flat}\Omega_{\cY/\Df_{\Psi}})}=\gamma^{*}\Omega_{\cY/\cX\times_{\Phi}\Psi},
\]
which shows the second assertion.
\end{proof}
\begin{lem}
\label{lem:jac-associativity}Let $f\colon\cY\to\cX$ be as above
and let $g\colon\cZ\to\cY$ be another morphism satisfying the same
condition. Let $\beta\colon\Spf O_{L}\to\cZ$ be a $\Df$-morphism
for $L/K\tpars$ as before. Let $\beta_{\cY}=g\circ\beta\colon\Spf O_{L}\to\cY$
and $\beta_{\cX}=f\circ g\circ\gamma\colon\Spf O_{L}\to\cX$ be the
induced morphisms. Suppose that $\fj_{\cZ}(\beta),$ $\fj_{\cY}(\beta_{\cY})$
and $\fj_{\cX}(\beta_{\cX})$ have finite values. Then
\[
\fj_{f\circ g}(\beta)=\fj_{g}(\beta)+\fj_{f}(\beta_{\cY}).
\]
\end{lem}

\begin{proof}
Let $r:=[L:K\tpars]$ and let $\fm\subset O_{L}$ be the maximal ideal.
Suppose that $\cZ$ is defined over $\Df_{\Upsilon}$. We have
\begin{align*}
(\beta_{\cY})^{\flat}\Omega_{\cY/\Df_{\Psi}}^{d} & =\fm^{\fj_{g}(\beta)r}\beta{}^{\flat}\Omega_{\cZ/\Df_{\Upsilon},}^{d}\\
(\beta_{\cX})^{\flat}\Omega_{\cX/\Df_{\Phi}}^{d} & =\fm^{\fj_{f\circ g}(\beta)r}\beta^{\flat}\Omega_{\cZ/\Df_{\Upsilon}}^{d},\\
(\beta_{\cX})^{\flat}\Omega_{\cX/\Df_{\Phi}}^{d} & =\fm^{\fj_{f}(\beta_{\cY})r}(\beta_{\cY})^{\flat}\Omega_{\cY/\Df_{\Psi}}^{d}.
\end{align*}
The first and third equalities imply
\[
(\beta_{\cX})^{\flat}\Omega_{\cX/\Df_{\Phi}}^{d}=\fm^{\fj_{f}(\beta_{\cY})r+\fj_{g}(\beta)r}\beta^{\flat}\Omega_{\cZ/\Df_{\Upsilon}}^{d}.
\]
Comparing this with the second, we get the lemma.
\end{proof}

\section{\label{sec:Gro-Rings}Grothendieck rings and realizations}

In this section, we define several rings constructed from the Grothendieck
ring of varieties. We will define motivic measures to take values
in these rings in later sections.
\begin{defn}
\label{def:usual-K0}For an algebraic space $S$ of finite type over
$k$, we denote by $\Var_{S}$ the category of algebraic spaces of
finite type over $S$. We define $K_{0}(\Var_{S})$ to be the quotient
of the free $\ZZ$-module generated by isomorphism classes $\{X\}$
of $X\in\Var_{S}$ modulo the scissor relation: if $Y\subset X$ is
a closed subspace, then $\{X\}=\{X\setminus Y\}+\{X\}$.
\end{defn}

\begin{rem}
We do not use the usual notation $[X]$ to avoid the confusion that
the symbol $[X/G]$ would have two meanings, a quotient stack and
the class of a quotient space $X/G$ in the Grothendieck ring.
\end{rem}

\begin{rem}
The usual definition of the Grothendieck ring of varieties uses schemes
rather than algebraic spaces. But every algebraic space $X$ of finite
type admits a stratification $X=\bigsqcup_{i}X_{i}$ into finitely
many locally closed subspaces $X_{i}$ which are schemes. Therefore
the use of algebraic spaces causes no essential difference. But algebraic
spaces are more natural for our purpose because they appear as coarse
moduli spaces of DM stacks.
\end{rem}

\begin{rem}
One may consider the Grothendieck \emph{semiring }instead of the Grothendieck
ring so that one can obtain slightly finer invariants (for instances,
see \cite[page 725]{MR2271984}, \cite[Ch. 2, Def. 1.2.1]{MR3838446}).
\end{rem}

The scissor relation in particular implies that $\{X\}=\{X_{\red}\}$.
For a constructible subset $C$ of $X\in\Var_{S}$, if $C=\bigsqcup_{i=1}^{n}C_{i}$
is a decomposition into finitely many locally closed subsets, then
we define $\{C\}:=\sum_{i=1}^{n}\{C_{i}\}\in K_{0}(\Var_{S})$, which
is independent of the decomposition.

The additive group $K_{0}(\Var_{S})$ becomes a commutative ring by
the multiplication $\{X\}\{Y\}:=\{X\times_{S}Y\}$. We denote $\{\AA_{S}^{1}\}\in K_{0}(\Var_{S})$
by $\LL$. For a positive integer $r$, we define a ring $K_{0}(\Var_{S})_{r}$
by adjoining an $r$-th root $\LL^{1/r}$ of $\LL$ formally. Namely
$K_{0}(\Var_{S})_{r}=K_{0}(\Var_{S})[x]/(x^{r}-\LL)$ and denote the
class of $x$ by $\LL^{1/r}$. We denote $\cM_{S}$ and $\cM_{S,r}$
to be the localization of $K_{0}(\Var_{S})$ and $K_{0}(\Var_{S})_{r}$
by $\LL$ respectively.
\begin{defn}
\label{def:aff-bundle}We say that a morphism $f\colon Y\to X$ between
algebraic spaces of finite type is a \emph{pseudo-$\AA^{r}$-bundle
}if for every geometric point $x\colon\Spec K\to X$, its fiber is
universally homeomorphic over $K$ to $\AA_{K}^{r}/H$ for some finite
group action $H\curvearrowright\AA_{K}^{r}$. In particular, a universal
bijection $Y\to X$ between algebraic spaces of finite type is a pseudo-$\AA^{0}$-bundle.
More generally, let $f\colon\cY\to\cX$ be a morphism between DM stacks
of finite type and let $D\subset|\cY|$ and $C\subset|\cX|$ be constructible
subsets with $f(D)\subset C$. We say that the map $f|_{D}\colon D\to C$
is a \emph{pseudo-$\AA^{r}$-bundle} if there exist stratification
$D=\bigsqcup_{i=1}^{n}D_{i}$ and $C=\bigsqcup_{i=1}^{n}C_{i}$ into
locally closed subsets such that for every $i$, $f(D_{i})\subset C_{i}$
and if $\cD_{i}\subset\cY$ and $\cC_{i}\subset\cX$ are the associated
reduced locally closed substacks, then the induced morphisms of coarse
moduli spaces, $\overline{\cD_{i}}\to\overline{\cC_{i}}$, are pseudo-$\AA^{r}$-bundles.
\end{defn}

\begin{defn}
We define the quotient group $K_{0}(\Var_{S})'_{r}$ by the following
relation: for a pseudo-$\AA^{r}$-bundle $Y\to X$ of algebraic spaces
of finite type over $S$, $\{Y\}=\{X\}\LL^{r}$.
\end{defn}

In particular, for a universally bijective morphism $Y\to X$ in $\Var_{S}$,
we have $\{Y\}=\{X\}$ in $K_{0}(\Var_{S})'_{r}$. If $\cX$ is a
DM stack of finite type over $S$ and $\overline{\cX}$ is its coarse
moduli space, then we define $\{\cX\}:=\{\overline{\cX}\}\in K_{0}(\Var_{S})'_{r}$.
If $C\subset|\cX|$ is a constructible subset, then we define $\{C\}\in K_{0}(\Var_{S})'_{r}$
either as the class $\{\overline{C}\}$ of the image $\overline{C}\subset|\overline{\cX}|$
of $C$ or as $\sum_{i=1}^{n}\{\cC_{i}\}$ for locally closed substacks
$\cC_{i}\subset\cX$ such that $C=\bigsqcup_{i=1}^{n}|\cC_{i}|$.
They coincide thanks to the relation above. Indeed, if $C_{i}\subset\overline{\cX}$
is the image of $\cC_{i}$ say with the structure of a reduced algebraic
space, then the morphism $\overline{\cC_{i}}\to C_{i}$ is a universal
homeomorphism and $\{\cC_{i}\}=\{C_{i}\}$. We define $\cM_{S}'$
and $\cM_{S,r}'$ to be the localizations of $K_{0}(\Var_{S})'$ and
$K_{0}(\Var_{S})'_{r}$ by $\LL$ respectively.
\begin{lem}
\label{lem:pseudo-bdl-eq}Let $S$ be an algebraic space of finite
type, let $\cY,\cX$ be DM stacks of finite type over $S$ and let
$D\subset|\cY|$ and $C\subset|\cX|$ be constructible subsets. Let
$f\colon\cY\to\cX$ be an $S$-morphism such that $f(D)\subset C$
and $f|_{D}\colon D\to C$ is a pseudo-$\AA^{d}$-bundle. Then, for
every positive integer $r$, we have
\[
\{D\}=\{C\}\LL^{d}\in K_{0}(\Var_{S})_{r}'.
\]
\end{lem}

\begin{proof}
We take stratifications $D=\bigsqcup_{i}D_{i}$ and $C=\bigsqcup_{i}C_{i}$
as in \ref{def:aff-bundle}. We then have
\[
\{D\}=\sum_{i}\{D_{i}\}=\sum_{i}\{\overline{\cD_{i}}\}=\sum_{i}\{\overline{\cC_{i}}\}\LL^{d}=\sum_{i}\{C_{i}\}\LL^{d}=\{C\}\LL^{d}.
\]
\end{proof}
For a morphism $f\colon S\to T$ between algebraic spaces of finite
type, we have a map $K_{0}(\Var_{S})\to K_{0}(\Var_{T})$ sending
the class $\{\phi\colon X\to S\}$ of an $S$-space to the class $\{f\circ\phi\colon X\to T\}$
of the induced $T$-space. This induces maps $K_{0}(\Var_{S})'\to K_{0}(\Var_{T})'$
and $\cM_{S}'\to\cM_{T}'$. We denote these maps by $\int_{S\to T}$.
When $T=\Spec k,$ we denote also by $\int_{S}$. We call them \emph{integration
maps}.

Next consider the case where $S$ is only \emph{locally} of finite
type over $k$. We define
\[
\cM_{S}':=\varinjlim_{U}\cM_{U}',
\]
where $U$ runs over quasi-compact open subspaces of $S$. Let $f\colon S\to T$
be a morphism between algebraic spaces locally of finite type over
$k$. Let $U\subset U'\subset S$ and $V\subset V'\subset T$ be finite-type
open subspaces such that $f(U)\subset V$ and $f(U')\subset V'$.
Then we have the following commutative diagram of integration maps:
\[
\xymatrix{\cM_{U}'\ar[r]\ar[d]\ar[dr] & \cM_{U'}'\ar[d]\\
\cM_{V}'\ar[r] & \cM_{V'}'
}
\]
Therefore, taking limits, we get the integration map 
\[
\int_{S\to T}\colon\cM_{S}'\to\cM_{T}'.
\]

To an $S$-scheme $X$ which is of finite type over $k$, we define
$\{X\}\in\cM_{S,r}'$ in the obvious way; if a finite-type open subset
$U\subset S$ contains the image of $X$, then $\{X\}\in\cM_{S,r}'$
is defined to be the image of $\{X\}\in\cM_{U,r}'$. For $n\in\frac{1}{r}\ZZ$,
we also define $\{X\}\LL^{n}\in\cM_{S,r}'$ to be the image of $\{X\}\LL^{n}\in\cM_{U,r}'$.
Then $\cM_{S,r}'$ is generated by elements of this form.

For $m\in\frac{1}{r}\ZZ$, we define $F_{m}\subset\cM_{S,r}'$ to
be the subgroup generated by elements of the form $\{X\}\LL^{n}$,
$n\in\frac{1}{r}\ZZ$ such that $\dim X+n\le-m$. Here $\dim X$ means
the dimension of the scheme $X$, not the relative dimension over
$S$. These groups form a descending filtration of $\cM_{S,r}'$.
We define 
\[
\widehat{\cM_{S,r}'}:=\varprojlim\cM_{S,r}'/F_{m}.
\]
The integration map $\int_{S\to T}$ extends to completion:
\[
\int_{S\to T}\colon\widehat{\cM_{S,r}'}\to\widehat{\cM_{T,r}'}.
\]

To extract cohomological or numerical data from elements of the Grothendieck
ring of varieties or rings derived from it as above, one often considers
realization maps. We first consider the $l$-adic realization. Suppose
that $k$ is finitely generated. Let $l$ be a prime number different
from $p$. Let $G_{k}=\Gal(k^{\sep}/k)$ be the absolute Galois group
of $k$ and let $\WRep_{G_{k}}(\QQ_{l})$ be the abelian category
of mixed $l$-adic representations of $G_{k}$ (see \cite[page 152]{MR2885336}).
To this abelian category, we associate the Grothendieck group $K_{0}(\WRep_{G_{k}}(\QQ_{l}))$
which has the ring structure induced by tensor products. There exists
the realization map 
\begin{equation}
K_{0}(\Var_{k})\to K_{0}(\WRep_{G_{k}}(\QQ_{l})),\{X\}\mapsto\chi_{l}(X):=\sum_{i}(-1)^{i}[\H_{c}^{i}(X_{\bar{k}},\QQ_{l})].\label{eq:realization}
\end{equation}
This extends to a map of complete rings
\begin{equation}
\widehat{\cM_{k}}\to\widehat{K_{0}(\WRep_{G_{k}}(\QQ_{l}))}.\label{eq:real-complete}
\end{equation}
Here the completion $\widehat{K_{0}(\WRep_{G_{k}}(\QQ_{l}))}$ is
induced from a filtration defined in terms of weights.
\begin{lem}
\label{prop:coh-moduli-sp}Let $\cX$ be a DM stack of finite type
and $X$ its coarse moduli space. Then, for every $i\in\ZZ$, we have
isomorphisms of cohomology groups (with compact support) 
\[
\H^{i}(X_{\bar{k}},\QQ_{l})\cong\H^{i}(\cX_{\bar{k}},\QQ_{l})\text{ and }\H_{c}^{i}(X_{\bar{k}},\QQ_{l})\cong\H_{c}^{i}(\cX_{\bar{k}},\QQ_{l})
\]
which are compatible with $\Gal(\bar{k}/k)$-actions.
\end{lem}

\begin{proof}
We literally follow the proof of \cite[Prop. 7.3.2]{MR2950161}, a
similar result for $\overline{\QQ_{l}}$-coefficient cohomology of
stacks over a finite field. Let $\pi\colon\cX\to X$ be the given
morphism. Since $\pi$ is proper and quasi-finite, $\rR^{i}\pi_{!}\QQ_{l}=\rR^{i}\pi_{*}\QQ_{l}=0$
for $i>0$. Moreover the natural map $\QQ_{l}\to\pi_{*}\QQ_{l}=\pi_{!}\QQ_{l}$
is an isomorphism. The proposition follows from the degeneration of
the spectral sequences $\H^{i}(X_{\bar{k}},\rR^{j}\pi_{*}\QQ_{l})\Rightarrow\H^{i+j}(\cX_{\bar{k}},\QQ_{l})$
and $\H_{c}^{i}(X_{\bar{k}},\rR^{j}\pi_{!}\QQ_{l})\Rightarrow\H_{c}^{i+j}(\cX_{\bar{k}},\QQ_{l})$.
\end{proof}
\begin{lem}
\label{lem:aff-bdl-stacks}Let $f\colon\cY\to\cX$ be a representable
morphism between DM stacks of finite type and let $d\in\NN$. Suppose
that for every geometric point $\Spec K\to\cX$, the fiber $\cY\times_{\cX}\Spec K$
is isomorphic to $\AA_{K}^{d}$ over $K$. Then 
\[
\H_{c}^{i+2d}(\cY_{\bar{k}},\QQ_{l})\cong\H_{c}^{i}(\cX_{\bar{k}},\QQ_{l})(-d).
\]
\end{lem}

\begin{proof}
Since $\H_{c}^{i}(\AA_{K}^{d},\QQ_{l})=0$ for $i\ne2d$, we have
$\rR^{i}f_{!}\QQ_{l}=0$ for $i\ne2d$ and the trace map $\rR^{2d}f_{!}\QQ_{l}(d)\to\QQ_{l}$
\cite[Th. 4.1]{MR3344762} is an isomorphism. Thus
\[
\H_{c}^{i+2d}(\cY_{\bar{k}},\QQ_{l})\cong\H_{c}^{i}(\cX_{\bar{k}},\rR^{2d}f_{!}\QQ_{l})\cong\H_{c}^{i}(\cX_{\bar{k}},\QQ_{l})(-d).
\]
\end{proof}
\begin{lem}
\label{lem:H-for-A/G}For a finite group action $G\curvearrowright\AA_{k}^{d}$,
we have 
\[
\H_{c}^{i}((\AA_{k}^{d}/G)_{\bar{k}},\QQ_{l})\cong\H_{c}^{i}(\AA_{\bar{k}}^{d},\QQ_{l})\cong\begin{cases}
\QQ_{l}(-d) & (i=2d)\\
0 & (i\ne2d).
\end{cases}
\]
\end{lem}

\begin{proof}
We have
\begin{align*}
\H_{c}^{i}((\AA_{k}^{d}/G)_{\bar{k}},\QQ_{l}) & \cong\H_{c}^{i}([\AA_{k}^{d}/G]_{\bar{k}},\QQ_{l})\\
 & \cong\H_{c}^{i-2d}((\B G)_{\bar{k}},\QQ_{l})(-d)\\
 & \cong\H_{c}^{i-2d}(\Spec\bar{k},\QQ_{l})(-d),
\end{align*}
where the first and last isomorphisms follow from \ref{prop:coh-moduli-sp}
and the middle one from \ref{lem:aff-bdl-stacks}.
\end{proof}
\begin{lem}
\label{lem:aff-bdl-stacks-1}Let $f\colon\cY\to\cX$ be a representable
morphism between DM stacks of finite type and let $d\in\NN$. Suppose
that for every geometric point $\Spec K\to\cX$, the fiber $\cY\times_{\cX}\Spec K$
is universally homeomorphic to $\AA_{K}^{d}/G$ over $K$ for some
finite group action $G\curvearrowright\AA_{K}^{d}$ over $K$. Then
\[
\H_{c}^{i+2d}(\cY_{\bar{k}},\QQ_{l})\cong\H_{c}^{i}(\cX_{\bar{k}},\QQ_{l})(-d).
\]
\end{lem}

\begin{proof}
Thanks to \ref{lem:H-for-A/G}, the same proof as the one of \ref{lem:aff-bdl-stacks}
is valid.
\end{proof}
In particular, this shows that for a pseudo-$\AA^{r}$-bunde $Y\to X$,
we have 
\[
\chi_{l}(Y)=\chi_{l}(X)\chi_{l}(\AA_{k}^{r})=\chi_{l}(X)[\QQ_{l}(-r)].
\]
Therefore maps (\ref{eq:realization}) and (\ref{eq:real-complete})
induce maps
\begin{gather*}
K_{0}(\Var_{k})'\to K_{0}(\WRep_{G_{k}}(\QQ_{l})),\\
\widehat{\cM_{k}'}\to\widehat{K_{0}(\WRep_{G_{k}}(\QQ_{l}))}.
\end{gather*}
For each $r\in\NN$, we also have a map
\[
\widehat{\cM_{k,r}'}\to\widehat{K_{0}(\WRep_{G_{k}}(\QQ_{l}))_{r}}.
\]
Here we define $K_{0}(\WRep_{G_{k}}(\QQ_{l}))_{r}$ by formally adjoining
a $r$-th root of $[\QQ_{l}(-1)]=\chi_{l}(\AA_{k}^{1})$. If there
exists a one-dimensional representation $V\in\WRep_{G_{k}}(\QQ_{l})$
such that $V^{\otimes r}\cong\QQ_{l}(-1)$, then we also have a map
\[
\widehat{\cM_{k,r}'}\to\widehat{K_{0}(\WRep_{G_{k}}(\QQ_{l}))}
\]
sending $\LL^{1/r}$ to $[V]$. From \cite[page 1510]{MR2098399},
such $V$ exists if we replace the base field $k$ with a finite extension
of it.

Similarly, when $k$ is a subfield of $\CC$, we can define 
\[
\widehat{\cM_{k,r}'}\to\widehat{K_{0}(\mathbf{MHS}^{1/r})},
\]
where $\mathbf{MHS}^{1/r}$ is the abelian category of $\frac{1}{r}\ZZ$-indexed
mixed Hodge structures \cite[page 741]{MR2271984}.

Over any field $k$, we have the Poincare polynomial realization,
\[
\widehat{\cM_{k,r}'}\to\ZZ[T^{1/r}]\llbracket T^{-1/r}\rrbracket
\]
(see \cite[Appendix]{MR2770561}). When $k$ is a subfield of $\CC$,
we also have the E-polynomial (Hodge-Deligne polynomial) realization,
\[
\widehat{\cM_{k,r}'}\to\ZZ[u^{1/r},v^{1/r}]\llbracket u^{-1/r},v^{-1/r}\rrbracket.
\]
For other realization maps, we refer the reader to \cite{MR2885336,MR3838446}.

\section{\label{sec:Motivic-integration-untwisted}Motivic integration for
untwisted arcs}

In this section, we develop the motivic integration over formal DM
stacks with restricting ourselves to untwisted arcs. We can do this
in a quite similar way as people did in less general cases of schemes
and formal schemes. Throughout the section, we make the following
assumption.
\begin{assumption}
\label{assu:X}Let $\Psi$ and $\Phi$ be reduced DM stacks almost
of finite type and let $\cY$ and $\cX$ be formal DM stacks of finite
type and flat over $\Df_{\Psi}$ and $\Df_{\Phi}$ respectively which
have pure relative dimension $d$ and are fiberwise generically smooth.
Let $\Psi\to\Phi$ be a morphism and let $f\colon\cY\to\cX$ be a
morphism compatible with $\Df_{\Psi}\to\Df_{\Phi}$. We suppose that
the Jacobian ideal sheaf $\Jac_{f}$ defines a narrow closed substack
of $\cY$.
\end{assumption}

We also fix a positive integer $r$ so that functions on arc spaces
and motivic measures will take values in $\frac{1}{r}\ZZ\cup\{\infty\}$
and $\widehat{\cM'_{Z,r}}$ respectively.

We first note that for an etale morphism $V\to\cX$ from a formal
scheme and a geometric point $\Spec K\to\Phi$, we have
\begin{gather*}
(\J_{n}\cX)\times_{\cX_{0}}V_{0}\cong\J_{n}V,\\
(\J_{\Phi,n}V)\times_{\Phi}\Spec K\cong\J_{\Spec K,n}(V\times_{\Phi}\Spec K).
\end{gather*}
Therefore we can often reduce problems on untwisted arcs and jets
to the case of formal schemes over $\Df_{K}$, which was treated in
\cite{MR2075915,MR3838446}. We denote the morphism $\J_{\infty}\cX\to\J_{n}\cX$
by $\pi_{n}$.
\begin{defn}
Let $C\subset|\J_{\infty}\cX|$ be a subset. We say that $C$ is \emph{cylindrical
of level $n$} ($n\in\NN$) if there exists a quasi-compact constructible
subset $C_{n}\subset|\J_{n}\cX|$ such that $\pi_{n}^{-1}(C_{n})=C$.
We say that $C$ is \emph{cylindrical }if it is cylindrical of level
$n$ for some $n\in\NN$.
\end{defn}

Note that from \cite[Appendix, Prop 1.3.3]{MR3838446}, $C$ is cylindrical
if and only if it is a quasi-compact constructible subset.
\begin{defn}
We say that a subset $C\subset|\J_{\infty}\cX|$ is \emph{stable of
level $n$ }if
\begin{enumerate}
\item $C$ is cylindrical of level $n$ and
\item for every $n'\ge n$, $\pi_{n'+1}(C)\to\pi_{n'}(C)$ is a pseudo-$\AA^{d}$-bundle.
\end{enumerate}
We say that $C$ is \emph{stable }if it is stable at level $n$ for
some $n\in\NN$.
\end{defn}

\begin{defn}
Let $Z$ be an algebraic space almost of finite type and let $\Phi\to Z$
be a morphism. For a stable subset $C$ of level $n$, we define its
measure by
\[
\mu_{\cX,Z}(C)=\{\pi_{n}(C)\}\LL^{-nd}\in\widehat{\cM_{Z,r}'}.
\]
\end{defn}

Cleary we have $\int_{Z}\mu_{\cX,Z}(C)=\mu_{\cX,k}(C)$.

For a cylindrical subset $C$ and $e\in\NN$, let 
\[
C^{(e)}:=\{c\in C\mid\fj_{\cX}(c)\le e\}.
\]
From \ref{lem:stability} below, this is a stable subset.
\begin{defn}
For a cylindrical subset $C$, we define 
\[
\mu_{\cX,Z}(C):=\lim_{e\to\infty}\mu_{\cX,Z}(C^{(e)}).
\]
\end{defn}

The limit exists by the same argument as in \cite[A.2]{MR1905024}.

Recall that $\widehat{\cM_{Z,r}'}$ is defined as the completion $\varprojlim\cM_{Z,r}'/F_{m}$.
This has the filtration $F_{\bullet}\widehat{\cM_{Z,r}'}$ given by
$F_{m}\widehat{\cM_{Z,r}'}:=\varprojlim_{n\ge m}F_{m}/F_{n}$.
\begin{defn}
\label{def:measurable-subset}We define the semi-norm
\[
\widehat{\cM_{Z,r}'}\mapsto\RR_{\ge0},\,a\mapsto\left\Vert a\right\Vert :=2^{-m},
\]
where $m$ is the largest $m\in\frac{1}{r}\ZZ$ such that $a\in F_{m}\widehat{\cM_{Z,r}'}$.
\end{defn}

For an element $\{V\}\LL^{r}$, $r\in\frac{1}{m}\ZZ$, we have $\Vert\{V\}\LL^{r}\Vert=2^{\dim V+r}$,
which is independent of $Z$.
\begin{defn}
A subset $C\subset|\J_{\infty}\cX|$ is said to be \emph{measurable
}if, for every $\epsilon>0$, there exist cylindrical subsets $C(\epsilon)$
and $A_{i}(\epsilon)$, $i\in\NN$ such that $(C\triangle C(\epsilon))\subset\bigcup_{i}A_{i}(\epsilon)$
and $\Vert\mu_{\cX,Z}(A_{i}(\epsilon))\Vert\le\epsilon$ for every
$i$. For a measurable subset $C$, we define 
\[
\mu_{\cX,Z}(C):=\lim_{\epsilon\to0}\mu_{\cX,Z}(C(\epsilon))\in\widehat{\cM_{Z,r}'}.
\]
We define a \emph{negligible subset }to be a measurable subset of
measure zero.
\end{defn}

Note that $\Vert\mu_{\cX,Z}(-)\Vert$ is independent of $Z$. Therefore
the measurability and negligibility are also independent of $Z$.
Clearly cylindrical subsets are measurable. We have a map
\[
\mu_{\cX,Z}\colon\{\text{measurable subset of }|\J_{\infty}\cX|\}\to\widehat{\cM'_{Z,r}}.
\]
Let $C_{i}$, $i\in\NN$ be a countable family of mutually disjoint
measurable subsets of $|\J_{\infty}\cX|$. Then $\bigsqcup_{i}C_{i}$
is measurable if and only if $\lim_{i\to\infty}\Vert\mu_{\cX,Z}(C_{i})\Vert=0$
(see \cite[Ch. 6, Prop. 3.4.3]{MR3838446}). If this is the case,
we have 
\[
\mu_{\cX,Z}\left(\bigsqcup_{i}C_{i}\right)=\sum_{i}\mu_{\cX,Z}(C_{i}).
\]

\begin{lem}
\label{lem:negl}Let $\cZ\subset\cX$ be a narrow closed substack.
Then $|\J_{\infty}\cZ|\subset|\J_{\infty}\cX|$ is negligible.
\end{lem}

\begin{proof}
This is known for the case of formal schemes over $\Df_{K}$ and essentially
follows from the fact that fibers of the map
\[
\pi_{n+1}(|\cJ_{\infty}\cZ|)\to\pi_{n}(|\cJ_{\infty}\cZ|)
\]
have dimension $<d$. We can show that this fact holds also in our
situation; we can deduce it to the classical case by the base change
along $\Spec K\to\Phi$ and an atlas $V\to\cX\times_{\Phi}\Spec K$.
The fact similarly implies the lemma also in our situation.
\end{proof}
\begin{defn}
\label{def:meas-fn-integ}Let $A\subset|\J_{\infty}\cX|$ be a subset
and $h\colon A\to\frac{1}{r}\ZZ\cup\{\infty\}$ be a function. We
say that $F$ is \emph{measurable }if there exists a decomposition
$A=\bigsqcup_{i\in\NN}A_{i}$ into countably many measurable subsets
such that for each $i$, $h|_{A_{i}}$ is constant and if $h(A_{i})=\infty$,
then $A_{i}$ is negligible. If $h$ is measurable, for $A_{i}$'s
as above, we define 
\[
\int_{A}\LL^{h}\,d\mu_{\cX,Z}:=\sum_{i\in\NN}\mu_{\cX,Z}(A_{i})\LL^{h(A_{i})}\in\widehat{\cM_{Z,r}'}\cup\{\infty\}.
\]
\end{defn}

This is independent of the decomposition $A=\bigsqcup_{i}A_{i}$.
For a morphism $Z\to Z'$ of algebraic spaces almost of finite type,
we have
\[
\int_{Z\to Z'}\int_{A}\LL^{h}\,d\mu_{\cX,Z}=\int_{A}\LL^{h}\,d\mu_{\cX,Z'}.
\]

\begin{notation}
In the rest of this section, for a ring $R$, we denote the ideal
$(t^{n+1})\subset R\tbrats$ by $\ft_{R}^{n}$ and the quotient $(t^{m+1})/(t^{n+1})$
for $n\ge m$ by $\ft_{n,R}^{m}$. For an arc $\gamma\colon\Df_{K}\to\cX$,
we denote by $\gamma^{\flat}\Omega_{\cX/\Df_{\Phi}}$ the flat pullback
$\left(\gamma^{*}\Omega_{\cX/\Df_{\Phi}}\right)/\tors$. By $\gamma_{n}$,
we denote the $n$-jet $\D_{K,n}\to\cX$ induced from $\gamma$.
\end{notation}

Let 
\[
v\colon V=\Spf R\to\cX
\]
be an etale morphism. For $n',n\in\NN\cup\{\infty\}$ with $n'\ge n$,
we have the following 2-Cartesian diagram.
\[
\xymatrix{\J_{n'}V\ar[r]\ar[d] & \J_{n'}\cX\ar[d]\\
\J_{n}V\ar[r] & \J_{n}\cX
}
\]

Let $K$ be an algebraically closed field and let $\beta\in(\J_{\infty}V)(K)$
and $\gamma:=v(\beta)\in(\J_{\infty}\cX)(K)$. In particular, a $K$-point
$\beta'\in(\J_{\infty}V)(K)$ over $\beta_{n}$ is identified with
a pair $(\gamma',\alpha)$ of $\gamma'\in(\J_{\infty}\cX)(K)$ and
an isomorphism $\alpha\colon\gamma_{n}\xrightarrow{\sim}\gamma'_{n}$.
For such $\beta'$, the map
\[
(\beta')^{*}-\beta^{*}\colon R\to K\tbrats
\]
has image in $\ft^{n+1}$ and the induced map $R\to\ft_{2(n+1)}^{n+1}$
is a $k\tbrats$-derivation. This corresponds to a $K\tbrats$-module
homomorphism 
\[
\beta^{*}\Omega_{V/\Df_{\Phi}}=\gamma^{*}\Omega_{\cX/\Df_{\Phi}}\to\ft_{2(n+1)}^{n+1}.
\]
Suppose 
\[
a:=\fj_{\cX}(\gamma)<\infty.
\]
The torsion part of $\gamma^{*}\Omega_{\cX/\Df_{\Phi}}$ has length
$a$ and of the form $\bigoplus_{i}K\tbrats/\ft^{a_{i}}$, $\sum_{i}a_{i}=a$.
If $2(n+1)-(n+c+1)\ge a$, equivalently if $c\le n-a+1$, then the
composite 
\[
\gamma^{*}\Omega_{\cX/\Df_{\Phi}}\to\ft_{2(n+1)}^{n+1}\to\ft_{n+c+1}^{n+1}
\]
kills the torsion part of $\gamma^{*}\Omega_{\cX/\Df_{\Phi}}$ and
factors through the flat pullback $\gamma^{\flat}\Omega_{\cX/\Df_{\Phi}}$.
\begin{notation}
We denote the corresponding element of $\Hom_{K\tbrats}(\gamma^{\flat}\Omega_{\cX/\Df_{\Phi}},\ft_{n+c+1}^{n+1})$
by $\delta_{n}^{n+c}(\beta;\beta')$ or $\delta_{n}^{n+c}(\gamma;\gamma',\alpha)$.
\end{notation}

For another such $\beta''\in(\J_{\infty}V)(K)$, we have that 
\begin{equation}
\beta'_{n+c}=\beta''_{n+c}\Leftrightarrow\delta_{n}^{n+c}(\beta;\beta')=\delta_{n}^{n+c}(\beta;\beta'').\label{eq:equiv}
\end{equation}
Let $F$ denote the fiber of $\pi_{n}^{n+c}\colon\J_{n+c}V\to\J_{n}V$
over $\beta_{n}$ and let $F':=F\cap\pi_{n}((\J_{\infty}V)\times_{\Phi}\Spec K)$,
the locus of $n$-jets in $F$ that lift to arcs. Since $2n+1\ge n+c$,
from \cite[Prop. 2.2.6]{MR3838446}, $F(K)$ is identified with $\Hom_{K\tbrats}(\gamma^{*}\Omega_{\cX/\Df_{\Phi}},\ft_{n+c+1}^{n+1})\cong K^{r}$
with $r$ being some nonnegative integer. Moreover $F$ is isomorphic
to $\AA_{K}^{r}$.
\begin{lem}
\label{lem:fiber-Hom}If $c\le n-a+1$, the map $\beta'\mapsto\delta_{n}^{n+c}(\beta;\beta')$
gives a bijection 
\[
F'(K)\to\Hom_{K\tbrats}(\gamma^{\flat}\Omega_{\cX/\Df_{\Phi}},\ft_{n+c+1}^{n+1}).
\]
\end{lem}

\begin{proof}
The injectivity follows from (\ref{eq:equiv}). When $c=1$, we have
$F'\cong\AA_{K}^{d}$ (see \cite[Lem. 9.1]{MR1886763}). We denote
$F'$ by $F'_{c}$ to specify the dependence on $c$. For $0<c'\le c$,
from the case $c=1$, the map $F_{c'}'\to F_{c'-1}'$ is a pseudo-$\AA^{d}$-bundle.
From \ref{lem:pseudo-bdl-eq}, $\{F_{c'}'\}=\{F_{c'-1}'\}\LL^{d}$
and $\{F_{c}'\}=\LL^{dc}$. Since 
\[
F_{n}^{n+c}\to\Hom_{K\tbrats}(\gamma^{\flat}\Omega_{\cX/\Df_{\Phi}},\ft_{n+c+1}^{n+1})=K^{dc}
\]
is injective, its image is a constructible subset with class $\LL^{dc}$
in $K_{0}(\Var_{k})'$. Therefore the complement of the image has
class $\LL^{dc}-\LL^{dc}=0$. This shows that it is empty, thus the
map is surjective.
\end{proof}
From this lemma, we obtain:
\begin{lem}
\label{lem:A^d-closed}The $F'$ is a $d$-dimensional linear subspace
of $F\cong\AA_{K}^{r}$. In particular, $F'$ is a closed subscheme
of $F$ and isomorphic to $\AA_{K}^{d}$.
\end{lem}

\begin{lem}
\label{lem:strat-closed-fib}Let $h\colon\cW\to\cV$ be a map between
DM stacks of finite type and let $D\subset|\cW|$ and $C\subset|\cV|$
be constructible subsets such that $h(D)\subset C$. Suppose that
for every $c\in|C|$, $h^{-1}(c)\cap D$ is a closed subset of $h^{-1}(c)$.
Then there exists a stratification $C=\bigsqcup_{i=1}^{n}C_{i}$ such
that $h^{-1}(C_{i})\cap D$ is a locally closed subset of $|\cW|$.
\end{lem}

\begin{proof}
By taking an arbitrary stratification $C=\bigsqcup_{i=1}^{n}C_{i}$
and restricting to each $C_{i}$, we can reduce the lemma to the case
$|\cV|=C$. Moreover we may suppose that $\cV$ is irreducible. It
suffices to show under these assumptions that there exists an open
dense subset $U\subset|\cV|$ such that $h^{-1}(U)\cap D$ is a locally
closed subset of $|\cW|$. Let $\eta\in|\cV|$ be the generic point
and let $\overline{h^{-1}(\eta)\cap D}$ be the closure of $h^{-1}(\eta)\cap D$
in $|\cW|$. Since $\overline{h^{-1}(\eta)\cap D}$ and $D$ coincide
over $\eta$, the symmetric difference 
\[
\overline{h^{-1}(\eta)\cap D}\triangle D:=\left(\overline{f^{-1}(\eta)\cap D}\setminus D\right)\cup\left(D\setminus\overline{f^{-1}(\eta)\cap D}\right)
\]
has non-dense image in $|\cV|$. Let $U\subset C$ be the complement
of the Zariski closure of this image. Then 
\[
h^{-1}(U)\cap D=h^{-1}(U)\cap\overline{h^{-1}(\eta)\cap D}
\]
is a closed subset of $h^{-1}(U)$ and a locally closed subset of
$|\cW|$.
\end{proof}
\begin{lem}
\label{lem:stability}For $a\in\NN$ and any quasi-compact substack
$\Upsilon\subset\Phi$, the set 
\[
\fj_{\cX}^{-1}(a)\cap((\J_{\infty}\cX)\times_{\Phi}\Upsilon)
\]
is stable of level $a$.
\end{lem}

\begin{proof}
Let us denote this set by $C$. The $\pi_{a}(C)$ is quasi-compact
and constructible. It suffices to show that for every $b\ge a$, the
map $\pi_{b+1}(C)\to\pi_{b}(C)$ is a pseudo-$\AA^{d}$-bundle. From
\ref{lem:strat-closed-fib} and \ref{lem:A^d-closed}, there exists
a stratification $\pi_{b}(C)=\bigsqcup_{i}\pi_{b}(C)_{i}$ into locally
closed subsets such that the preimage $\pi_{b+1}(C)_{i}\subset\pi_{b+1}(C)$
of $\pi_{b}(C)_{i}$ is also locally closed. We regard these locally
closed subsets as substacks of $\J_{b}\cX$ or $\J_{b+1}\cX$. For
a geometric point $\xi\colon\Spec K\to\pi_{b}(C)_{i}$, if $G$ is
its automorphism group, we have the following 2-commutative diagram
of stacks over $K$:
\[
\xymatrix{\AA_{K}^{d}\ar[r]\ar[d] & [\AA_{K}^{d}/G]\ar@{^{(}->}[r]\ar[d] & \pi_{b+1}(C)_{i}\otimes_{k}K\ar[d]\\
\Spec K\ar[r] & \B G\ar@{^{(}->}[r] & \pi_{b}(C)_{i}\otimes_{k}K
}
\]
Here the two squares are 2-Cartesian and the arrows of the form $\hookrightarrow$
are closed immersions. This induces a universally injective morphism
of coarse moduli spaces
\[
\AA_{K}^{d}/G\to\overline{\pi_{b+1}(C)_{i}\otimes_{k}K}.
\]
The image of this map is the fiber of 
\[
\overline{\pi_{b+1}(C)_{i}\otimes_{k}K}\to\overline{\pi_{b}(C)_{i}\otimes_{k}K}
\]
over the point induced by $\xi$. We conclude that $\pi_{b+1}(C)\to\pi_{b}(C)$
is a pseudo-$\AA^{d}$-bundle.
\end{proof}
\begin{lem}
\label{lem:n-e}Let $\gamma\in(\J_{\infty}\cY)(K)$ and $\eta\in(\J_{\infty}\cX)(K)$
with $K$ an algebraically closed field. Let $e:=\fj_{f}(\gamma)$,
$a:=\fj_{\cX}(f_{\infty}(\gamma))$ and let $n\in\NN$ be such that
$n\ge2e+a$. Suppose that there exists an isomorphism $\sigma\colon f_{n}(\gamma_{n})\cong\eta_{n}$.
Then there exists $\xi\in(\J_{\infty}\cY)(K)$ such that $\gamma_{n-e}\cong\xi{}_{n-e}$
and $f_{\infty}(\xi)\cong\eta$.
\end{lem}

\begin{proof}
We have a $K\tbrats$-module homomorphism
\[
\delta_{n}^{n+1}(\gamma\circ f;\eta,\sigma)\colon(\gamma\circ f)^{\flat}\Omega_{\cX/\Df_{\Phi}}\to\ft_{n+2}^{n+1}.
\]
For suitable bases, the map of free modules $(\gamma\circ f)^{\flat}\Omega_{\cX/\Df_{\Phi}}\to\gamma^{\flat}\Omega_{\cY/\Df_{\Psi}}$
is represented by a diagonal matrix $\mathrm{diag}(t^{e_{1}},\dots,t^{e_{d}})$
with $\sum_{i=1}^{d}e_{i}=e$. Therefore $\delta_{n}^{n+1}(\gamma\circ f;\eta)$
is induced from some map $\gamma^{\flat}\Omega_{\cY/\Df_{\Psi}}\to\ft_{n+2}^{n-e+1}$.
Since $n\ge2e+a$, we have
\[
(n+2)-(n-e+1)=e+1\le(n-e)-a+1.
\]
From \ref{lem:fiber-Hom}, this map is $\delta_{n-e}^{n+1}(\gamma;\xi^{(1)},\alpha^{(1)})$
for some $\xi^{(1)}\in(\J_{\infty}\cY)(K)$ and an isomorphism $\alpha^{(1)}\colon\gamma_{n-e}\cong\xi_{n-e}^{(1)}$.
We have an isomorphism $f_{n+1}(\xi_{n+1}^{(1)})\cong\eta_{n+1}$
compatible with $\sigma$ and $\alpha^{(1)}$. Applying the same argument
to $\text{\ensuremath{\xi^{(1)}} and \ensuremath{\eta}}$ instead
of $\gamma$ and $\eta$, we obtain $\xi^{(2)}\in(\J_{\infty}\cY)(K)$
and isomorphisms $\xi_{n-e+1}^{(1)}\cong\xi_{n-e+1}^{(2)}$ and $f_{n+2}(\xi_{n+2}^{(2)})\cong\eta_{n+2}$.
Repeating this procedure, we obtain a sequence $\xi^{(i)}\in(\J_{\infty}\cY)(K)$,
$i>0$ with isomorphisms $\xi_{n-e+i}^{(i)}\cong\xi_{n-e+i}^{(i+1)}$
and $f_{n+i}(\xi_{n+i}^{(i)})\cong\eta_{n+i}$. The sequence $\xi_{n-e+i}^{(i)}\in(\J_{n-e+i}\cY)(K)$
gives the limit $\xi\in(\J_{\infty}\cY)(K)$ such that $\gamma_{n-e}\cong\xi_{n-e}$
and $f_{\infty}(\xi)\cong\eta$.
\end{proof}
\begin{rem}
If $\gamma$ and $\gamma'$ correspond to twisted arcs $\cE_{K}\to\cY$
and $\cE_{K}'\to\cY$ respectively, then the equality $\pi_{n-e}(\gamma)\cong\pi_{n-e}(\gamma')$
especially implies that $\cE_{K}\cong\cE_{K}'$.
\end{rem}

\begin{defn}
\label{def:geo-inj}For a subset $C\subset|\J_{\infty}\cY|$, we say
that $f_{\infty}|_{C}$ is \emph{geometrically injective }if for every
algebraically closed field $K$, the map $f_{\infty}|_{C[K]}\colon C[K]\to(\J_{\infty}\cX)[K]$
is injective. We say that $f|_{C}$ is \emph{almost geometrically
injective }if there exists a negligible subset $C'\subset C$ such
that $f|_{C\setminus C'}$ is geometrically injective.

Let $B\subset|\J_{\infty}\cX|$ be a subset such that $f_{\infty}(C)\subset B$.
We say that $f_{\infty}|_{C}\colon C\to B$ is \emph{almost geometrically
bijective }if there exist negligible subsets $B'\subset B$ and $C'\subset C$
such that for each algebraically closed field $K$, $f_{\infty}$
induces a bijection $(C\setminus C')[K]\to(B\setminus B')[K]$.
\end{defn}

\begin{lem}
\label{lem:key}Let $e,n\in\NN$ with $n\ge e$ and let $C\subset|\J_{\infty}\cY|$
be a stable subset of level $n-e$. Suppose that $\fj_{f}|_{C}$ takes
constant value $e$. Suppose that $n\ge2e+\fj_{\cX}(\gamma)$ for
every $\gamma\in C$. Suppose that $f_{\infty}|_{C}$ is geometrically
injective. Then $\pi_{n}(C)\to f_{n}\pi_{n}(C)$ is a pseudo-$\AA^{e}$-bundle.
\end{lem}

\begin{proof}
Let $\gamma,\gamma'\in C(K)$ be such that $f_{n}(\gamma_{n})\cong f_{n}(\gamma'_{n})$.
Applying \ref{lem:n-e} for $\eta=f_{\infty}(\gamma')$, there exists
$\xi\in(\J_{\infty}\cY)(K)$ such that $\gamma_{n-e}\cong\xi_{n-e}$
and $f_{\infty}(\xi)\cong f_{\infty}(\gamma')$. From the injectivity
assumption, we have $\xi\cong\gamma'$ and conclude that $\gamma_{n-e}\cong\gamma'_{n-e}$.
Namely a fiber of $\pi_{n}(C)\to f_{n}\pi_{n}(C)$ is contained in
a fiber of $\pi_{n}(C)\to\pi_{n-e}(C)$. We may suppose that $\pi_{n}(C)$
and $\pi_{n-e}(C)$ are locally closed subsets. We regard them as
substacks. From $\gamma_{n-e}\colon\Spec K\to\pi_{n-e}(C)$, we obtain
the following 2-Cartesian diagram similarly as in the proof of \ref{lem:stability}:
\[
\xymatrix{\AA_{K}^{de}\ar[r]\ar[d] & [\AA_{K}^{de}/G]\ar@{^{(}->}[r]\ar[d] & \pi_{n}(C)\otimes_{k}K\ar[d]\\
\Spec K\ar[r] & \B G\ar@{^{(}->}[r] & \pi_{n-e}(C)\otimes_{k}K
}
\]
From \ref{lem:fiber-Hom}, the $K$-point set of the above $\AA_{K}^{de}$
is identified with
\[
\Hom_{K\tbrats}(\gamma^{\flat}\Omega_{\cY/\Df_{\Psi}},\ft_{n+1}^{n-e+1}).
\]
The kernel of the map
\[
\Hom_{K\tbrats}(\gamma^{\flat}\Omega_{\cY/\Df_{\Psi}},\ft_{n+1}^{n-e+1})\to\Hom_{k\tbrats}((\gamma\circ f)^{\flat}\Omega_{\cX/\Df_{\Phi}},\ft_{n+1}^{n-e+1})
\]
parametrizes pairs $(\gamma_{n}',\alpha_{n})$ of $\gamma_{n}'\in\pi_{n}((\J_{\infty}\cX)(K))$
and an isomorphism $\gamma_{n-e}'\cong\gamma_{n-e}$ such that the
induced isomorphism $f_{n-e}(\gamma_{n-e})\cong f_{n-e}(\gamma_{n-e}')$
lifts to $f_{n}(\gamma_{n})\cong f_{n}(\gamma'_{n})$. This map is
represented by a diagonal matrix $\mathrm{diag}(t^{e_{1}},\dots,t^{e_{d}})$
with $\sum_{i=1}^{d}e_{i}=e$. The kernel is isomorphic to $\bigoplus_{i=1}^{d}\ft_{n+1}^{n-e_{i}+1}\cong K^{\oplus e}$.
The image of the corresponding subspace $\AA_{K}^{e}\cong W\subset\AA_{K}^{de}$
in $\pi_{n}(C)\otimes_{k}K$ is the fiber of 
\[
\pi_{n}(C)\otimes_{k}K\to f_{n}\pi_{n}(C)\otimes_{k}K
\]
over the point $f_{n}(\gamma_{n})$. We see that $W$ is stable under
the $G$-action. We obtain a closed immersion $[W/G]\to\pi_{n}(C)\otimes_{k}K$.
The induced morphism of coarse moduli spaces 
\[
W/G\to\overline{\pi_{n}(C)\otimes_{k}K}
\]
is a universally injective morphism onto the fiber of $\overline{\pi_{n}(C)}\to\overline{f_{n}\pi_{n}(C)}$
over the image of $f_{n}(\gamma_{n})$. We have proved the lemma.
\end{proof}
\begin{lem}
\label{lem:key-2}Let $C\subset|\J_{\infty}\cY|$ be a measurable
subset such that $\fj_{f}$ has constant value $e<\infty$ on $C$.
Suppose that $f_{\infty}|_{C}$ is almost geometrically injective.
Then 
\[
\mu_{\cY,Z}(C)=\mu_{\cX,Z}(f_{\infty}(C))\LL^{e}.
\]
\end{lem}

\begin{proof}
By a standard limit argument, we can reduce to the case where $C\subset|\J_{\infty}\cY|$
is a stable subset such that $\fj_{f}$ and $\fj_{\cX}\circ f_{\infty}$
are nowhere infinity on $C$ and $f|_{C[K]}$ are injective. For $n\gg0$,
from \ref{lem:key},
\begin{align*}
\mu_{\cY,Z}(C) & =\{\pi_{n}(C)\}\LL^{-dn}\\
 & =\{\pi_{n}f_{\infty}(C)\}\LL^{-dn+e}\\
 & =\mu_{\cX,Z}(f_{\infty}(C))\LL^{e}.
\end{align*}
\end{proof}
\begin{cor}
\label{cor:stable-image}Let $C\subset|\J_{\infty}\cY|$ be a stable
subset. Suppose that $\ord_{\Jac_{f}}$ and $\fj_{\cX}\circ f_{\infty}$
are nowhere infinity on $C$. Then $f_{\infty}(C)$ is also stable.
\end{cor}

\begin{proof}
For $e\in\NN$, let $C_{e}$ be the locus in $C$ with $\ord_{\Jac_{f}}=e$.
This is a stable subset, since whether $\ord_{\Jac_{f}}(\gamma)=e$
is determined by the $n$-jet $\gamma_{n}$. We have $C=\bigsqcup_{e}C_{e}$.
From \cite[Lem. 2.3]{MR1886763} (although this paper assume that
the base field has characteristic zero, this lemma is characteristic-free),
$C$ is covered by finitely many $C_{e}$. Thus $\ord_{\Jac_{f}}|_{C}$
is bounded. From \ref{lem:relation-naive-jac}, $\fj_{f}$ is also
bounded. By the same argument, $(\fj_{\cX}\circ f_{\infty})|_{C}$
is also bounded. Let us take $e,a,n\in\NN$ such that $n\ge2e+a$,
$\fj_{f}|_{C}\le e$, $(\fj_{\cX}\circ f_{\infty})|_{C}\le a$ and
$C$ is stable of level $n-e$. Let $\gamma\in C$ and $\eta\in|\J_{\infty}\cX|$.
If $f_{n}(\gamma_{n})=\eta_{n}$, then from \ref{lem:n-e}, if $f_{n}(\gamma_{n})=\eta_{n}$,
then there exists $\xi\in|\J_{\infty}\cX|$ such that $f_{\infty}(\xi)=\eta$
and $\gamma_{n-e}=\xi_{n-e}$. Since $C$ is stable of level $n-e$,
we have $\xi\in C$. Therefore $\eta\in f_{\infty}(C)$. We have showed
that $f_{\infty}(C)$ is cylindrical of level $n$. Since $\fj_{\cX}|_{f_{\infty}(C)}$
is bounded, we can see from \ref{lem:stability} that $f_{\infty}(C)$
is also stable.
\end{proof}
\begin{lem}
\label{lem:measurable-image}Let $C\subset|\J_{\infty}\cY|$ be a
measurable subset. Then $f_{\infty}(C)$ is also measurable and $\Vert\mu_{\cX}(f_{\infty}(C))\Vert\le\Vert\mu_{\cY}(C)\Vert$.
In particular, if $C$ is negligible, then so is $f_{\infty}(C)$. 
\end{lem}

\begin{proof}
Let $C_{\infty}\subset C$ be the locus where either $\fj_{f}$ or
$\fj_{\cX}\circ f_{\infty}$ has value $\infty$. From \ref{assu:X}
and \ref{lem:relation-naive-jac}, there exists a narrow closed substack
$\cZ\subset\cY$ such that $C_{\infty}\subset|\J_{\infty}\cZ|$. We
show that $f_{\infty}(|\J_{\infty}\cZ|)$ is negligible, hence so
is $f_{\infty}(C_{\infty})$. Since a countable union of negligible
subsets is negligible, we may suppose that $\cZ$ is contained in
$\cY\times_{\Psi}\Psi_{0}$ for some connected component $\Psi_{0}$
of $\Psi$. For each geometric point $\psi\colon\Spec K\to\Psi_{0}$,
consider the induced morphism $f_{K}\colon\cY_{K}\to\cX_{K}$ of formal
DM stacks of finite type over $\Df_{K}$ and the induced narrow substack
$\cZ_{K}\subset\cY_{K}$. The scheme-theoretic image $\cW_{K}$ of
$\cZ_{K}$ in $\cX_{K}$ is again narrow. Therefore
\[
\dim\pi_{n}\circ f_{\infty}(|\J_{\infty}\cZ_{K}|)\le\dim\pi_{n}(|\J_{\infty}\cW_{K}|)\le(d-1)(n+1).
\]
This shows 
\[
\dim\pi_{n}\circ f_{\infty}(|\J_{\infty}\cZ|)\le(d-1)(n+1)+\dim\Psi_{0}.
\]
Therefore the measure of the cylindrical subset $\pi_{n}^{-1}(\pi_{n}\circ f_{\infty}(|\J_{\infty}\cZ|))$
tends to zero as $n$ tends to infinity. Since 
\[
f_{\infty}(|\J_{\infty}\cZ|)\subset\bigcap_{n}\pi_{n}^{-1}(\pi_{n}\circ f_{\infty}(|\J_{\infty}\cZ|)),
\]
we see that $f_{\infty}(|\J_{\infty}\cZ|)$ is negligible.

We can write $C\setminus C_{\infty}=\bigsqcup_{i\in\NN}C_{i}$ such
that for each $i$, $C_{i}$ is measurable, and $\fj_{f}$ and $\fj_{\cX}\circ f_{\infty}$
are bounded on $C_{i}$. For each $C_{i}$, we can take stable subsets
$C_{i}(\epsilon)$ and $A_{ij}(\epsilon)$, $j\in\NN$ of $|\J_{\infty}\cY|$
as in \ref{def:measurable-subset} such that $\fj_{f}$ and $\fj_{\cX}\circ f_{\infty}$
are bounded on each of these stable subsets. We have
\[
f_{\infty}(C_{i})\triangle f_{\infty}(C_{i}(\epsilon))\subset f_{\infty}(C_{i}\triangle C_{i}(\epsilon))\subset\bigcup_{j}f_{\infty}(A_{ij}(\epsilon)).
\]
From \ref{cor:stable-image}, $f_{\infty}(C_{i}(\epsilon))$ and $f_{\infty}(A_{ij}(\epsilon))$
are stable and we have
\[
\Vert\mu_{\cX}(f_{\infty}(C_{i}(\epsilon)))\Vert\le\Vert\mu_{\cY}(C_{i}(\epsilon))\Vert\text{ and }\Vert\mu_{\cX}(f_{\infty}(A_{ij}(\epsilon))\Vert\le\Vert\mu_{\cY}(A_{ij}(\epsilon))\Vert.
\]
This shows that $f_{\infty}(C_{i})$ is measurable with $\Vert\mu_{\cX}(f_{\infty}(C_{i}))\Vert\le\Vert\mu_{\cX}(C_{i})\Vert$.
Since $\lim_{i\to\infty}\Vert\mu_{\cY}(C_{i})\Vert=0$, we have $\lim_{i\to\infty}\Vert\mu_{\cX}(f_{\infty}(C_{i}))\Vert=0$.
Therefore $f_{\infty}(C)=\bigsqcup_{i\in\NN\cup\{\infty\}}f_{\infty}(C_{i})$
is also measurable. The inequality of the lemma is clear.
\end{proof}
\begin{lem}
\label{lem:measurable-fnc}
\begin{enumerate}
\item Let $A\subset|\J_{\infty}\cX|$ be a subset with $A\subset f_{\infty}(|\J_{\infty}\cY|)$
and let $h\colon A\to\frac{1}{r}\ZZ\cup\{\infty\}$ be a function.
If $h\circ f_{\infty}\colon f_{\infty}^{-1}(A)\to\frac{1}{r}\ZZ\cup\{\infty\}$
is measurable, then so is $h$.
\item Let $B\subset|\J_{\infty}\cY|$ be a measurable subset such that $f_{\infty}|_{B}$
is almost geometrically injective. If a function $h\colon f_{\infty}(B)\to\frac{1}{r}\ZZ\cup\{\infty\}$
is measurable, then $h\circ f_{\infty}\colon B\to\frac{1}{r}\ZZ\cup\{\infty\}$
is also measurable.
\end{enumerate}
\end{lem}

\begin{proof}
(1) Let $f_{\infty}^{-1}(A)=\bigsqcup_{i\in\NN}B_{i}$ be a decomposition
as in the definition of measurable functions. Then $h$ is constant
on each $f_{\infty}(B_{i})$. From \ref{lem:measurable-image}, $f_{\infty}(B_{i})$
are measurable and $\mu_{\cX}(f_{\infty}(B_{i}))=0$ if $\mu_{\cY}(B_{i})=0$.
If we put $A_{i}:=f_{\infty}(B_{i})\setminus\bigcup_{j<i}f_{\infty}(B_{j})$,
we have
\[
A=\bigcup_{i\in\NN}f_{\infty}(B_{i})=\bigsqcup_{i\in\NN}A_{i}.
\]
This decomposition of $A$ shows that $h$ is measurable.

(2) Removing the negligible subset $(\ord_{\Jac_{f}})^{-1}(\infty)\cup(\fj_{\cX}\circ f_{\infty})^{-1}(\infty)$,
we may suppose that these two functions are nowhere infinity. For
$e,a\in\NN$, let 
\[
B_{e,a}:=(\ord_{\Jac_{f}})^{-1}(e)\cap(\fj_{\cX}\circ f_{\infty})^{-1}(a),
\]
which is a measurable subset. Then $h|_{f_{\infty}(B_{e,a})}\colon f_{\infty}(B)\to\frac{1}{r}\ZZ\cup\{\infty\}$
is a measurable function. Let $f_{\infty}(B_{e,a})=\bigsqcup_{i}C_{i}$
be a decomposition into countably many measurable subsets such that
for every $i$, $h|_{C_{i}}$ is constant and if $h|_{C_{i}}\equiv0$,
then $C_{i}$ is negligible. From \ref{lem:key}, we see that $f_{\infty}^{-1}(C_{i})$
is also measurable and if $C_{i}$ is negligible, so is $f_{\infty}^{-1}(C_{i})$.
This shows the assertion.
\end{proof}
\begin{lem}
\label{lem:ord-fn-measur}Let $\cI\subset\cO_{\cX}$ be an ideal sheaf
defining a narrow closed substack. Then $\ord_{\cI}\colon|\J_{\infty}\cX|\to\ZZ\cup\{\infty\}$
is a measurable function.
\end{lem}

\begin{proof}
It suffices to show that for each connected component $\Upsilon$
of $\Phi$, the restriction of $\ord_{\cI}$ to $(\J_{\infty}\cX)\times_{\Phi}\Upsilon$
is measurable. Therefore we may suppose that $\Phi$ is of finite
type. Let $C_{a,b}:=\{\ord_{\cI}=a,\,\fj_{\cX}=b\}$. We have $|\J_{\infty}\cX|=\bigsqcup_{0\le a,b\le\infty}C_{a,b}.$
If $a=\infty$ or $b=\infty$, then $C_{a,b}$ is negligible. Otherwise
$C_{a,b}$ is stable of level $\max\{a,b\}$ from \ref{lem:stability}.
The lemma follows.
\end{proof}
\begin{lem}
\label{lem:jac-measur}With the above notation, the function $\fj_{f}\colon|\J_{\infty}\cY|\to\ZZ\cup\{\infty\}$
is measurable .
\end{lem}

\begin{proof}
Let $h\colon\cZ\to\cY$ be the blowup along $\Omega_{\cY/\Df_{\Psi}}^{d}\oplus f^{*}\Omega_{\cX/\Df_{\Phi}}^{d}$
(in the sense of \cite{1087.14011}) so that 
\begin{gather*}
h^{\flat}\Omega_{\cY/\Df_{\Psi}}^{d}:=h^{*}\Omega_{\cY/\Df_{\Psi}}^{d}/\tors,\\
(f\circ h)^{\flat}\Omega_{\cX/\Df_{\Phi}}^{d}:=(f\circ h)^{*}\Omega_{\cX/\Df_{\Phi}}^{d}/\tors
\end{gather*}
are invertible sheaves. From assumption \ref{assu:X}, $h_{\infty}\colon|\J_{\infty}\cZ|\to|\J_{\infty}\cY|$
is almost geometrically bijective. The image of $(f\circ h)^{\flat}\Omega_{\cX/\Df_{\Phi}}^{d}\to h^{\flat}\Omega_{\cY/\Df_{\Psi}}^{d}$
is written as $\cI\cdot h^{\flat}\Omega_{\cY/\Df_{\Psi}}^{d}$ for
some locally principal ideal sheaf $\cI$, which defines a narrow
substack. Then $\fj_{f}\circ h_{\infty}$ is equal to $\ord_{\cI}$
outside the negligible subset $\{\fj_{\cX}\circ f_{\infty}\circ h_{\infty}=\infty\}\cup\{\fj_{\cY}\circ h_{\infty}=\infty\}$.
The function $\ord_{\cI}$ is measurable. From \ref{lem:measurable-fnc},
$\fj_{f}$ is measurable.
\end{proof}
The following is the \emph{change of variables formula }(also called
the \emph{transformation rule}) for untwisted arcs in the present
setting.
\begin{thm}
\label{thm:change-vars-I}We keep assumption \ref{assu:X}. Let $A\subset|\J_{\infty}\cY|$
be a measurable subset. Suppose that $f_{\infty}|_{A}$ is almost
geometrically injective. Let $h\colon f_{\infty}(A)\to\frac{1}{r}\ZZ\cup\{\infty\}$
be a function such that $h\circ f_{\infty}$ is measurable. Then 
\[
\int_{f_{\infty}(A)}\LL^{h}\,d\mu_{\cX,Z}=\int_{A}\LL^{h\circ f_{\infty}-\fj_{f}}\,d\mu_{\cY,Z}.
\]
\end{thm}

\begin{proof}
Removing a negligible subset from $A$, we may suppose that all relevant
measurable functions are nowhere infinity. There exists a decomposition
$A=\bigsqcup_{i}A_{i}$ into countably many measurable subsets $A_{i}$
such that $h\circ f_{\infty}$ and $\fj_{f}$ are constant on each
$A_{i}$. From \ref{lem:key},
\begin{align*}
\int_{f_{\infty}(A)}\LL^{h}\,d\mu_{\cX,Z} & =\sum_{i}\mu_{\cX,Z}(f_{\infty}(A_{i}))\LL^{h(f_{\infty}(A_{i}))}\\
 & =\sum_{i}\mu_{\cY,Z}(A_{i})\LL^{h(f_{\infty}(A_{i}))-\fj_{f}(A_{i})}\\
 & =\int_{A}\LL^{h\circ f_{\infty}-\fj_{f}}\,d\mu_{\cY,Z}.
\end{align*}
\end{proof}

\section{Motivic integration for twisted arcs\label{sec:integration-twisted}}

In this section, we develop the motivic integration over formal DM
stacks for twisted arcs. Using untwisting stacks, this is reduced
to the motivic integration for untwisted arcs.
\begin{defn}
\label{def:spacious}Let $\cX$ be a formal DM stack of finite type
over $\Df$ and let $\I\cX$ be its inertia stack with the trivial
section $\epsilon\colon\cX\to\I\cX$. We define the \emph{stacky locus
}$\cX_{\st}$ of $\cX$ to be the scheme-theoretic image of the finite
morphism $\I\cX\setminus\epsilon(\cX)\to\cX$. We say that $\cX$
is \emph{spacious }if $\cX_{\st}$ is a narrow closed substack of
$\cX$.
\end{defn}

In the rest of this section, we make the following assumption.
\begin{assumption}
\label{assu:tw}Let $\cX,\cY$ be formal DM stacks of finite type,
flat and of pure relative dimension $d$ over $\Df$. Suppose that
they are spacious and generically smooth over $\Df$. Let $f\colon\cY\to\cX$
be a morphism over $\Df$ such that the Jacobian ideal sheaf $\Jac_{f}$
defines a narrow closed substack of $\cY$.
\end{assumption}

Let $X$ be the coarse moduli space of $\cX$ and let $\overline{\cX_{\st}}\subset X$
be the scheme-theoretic image of $\cX_{\st}$.
\begin{lem}
The subspace $\overline{\cX_{\st}}\subset X$ is narrow. 
\end{lem}

\begin{proof}
From \ref{lem:formal-locally-quot}, we can reduce the lemma to the
case where $\cX$ is a quotient stack $[\Spf R/G]$ with $G$ a finite
group. Let $I\subset R$ be the ideal defining the preimage of $\cX_{\st}$
in $\Spf R$. The assumption that $\cX$ is spacious means that $V(I)\subset\Spec R$
is nowhere dense. Then the image of $V(I)$ in $\Spec R^{G}$ is also
nowhere dense. This shows the lemma.
\end{proof}
From the morphism $\pi_{\Gamma}^{\utg}\colon\Utg_{\Gamma}(\cX)^{\pur}\to X$
as in \ref{prop:Unt-morphism}, we obtain maps
\begin{equation}
|\cJ_{\infty}\cX|\to|\J_{\infty}X|\label{map1}
\end{equation}
and 
\begin{equation}
(\cJ_{\infty}\cX)[K]\to(\J_{\infty}X)(K)\label{map2}
\end{equation}
for each algebraically closed field $K$. These maps send a twisted
arc $\cE\to\cX$ to the arc of $X$ induced by taking coarse moduli
spaces. The notion of twisted arcs is necessary for the following
proposition to holds.
\begin{prop}
\label{lem:almost-bij}Let $\overline{\cX_{\st}}\subset X$ be the
scheme-theoretic image of $\cX_{\st}$. For every algebraically closed
field $K$, the map 
\[
(\cJ_{\infty}\cX)[K]\setminus(\cJ_{\infty}\cX_{\st})[K]\to(\J_{\infty}X)(K)\setminus(\J_{\infty}\overline{\cX_{\st}})(K)
\]
is bijective. In particular, the map $|\cJ_{\infty}\cX|\to|\J_{\infty}X|$
is almost geometrically bijective.
\end{prop}

\begin{proof}
First note that if a twisted arc $\cE\to\cX$ with $\cE$ a twisted
formal disk over $K$ does not factor through $\cX_{\st}$, then the
induced arc $\Df_{K}\to X$ does not factor through $\overline{\cX_{\st}}$
either. Therefore we have the map of the lemma. Conversely, given
an arc $\Df_{K}\to X$ which does not factor through $\overline{\cX_{\st}}$,
we can construct a twisted arc of $\cX$ as follows. Let $\cE$ be
the normalization of $(\cX\times_{X}\Df_{K})^{\pur}$. Since $\cX$
is spacious, so is $\cE$. We conclude that $\cE$ is a twisted formal
disk over $K$. The morphism $\cE\to\cX$ is a twisted arc over $K$,
which maps to the given arc $\Df_{K}\to X$. Therefore the map of
the lemma is surjective.

To show the injectivity, let $\cE'\to\cX$ be a twisted arc and $\Df_{K}\to X$
the induced arc. Let $\cE\to\cX$ be the twisted arc induced from
this arc on $X$ by the above construction. We show that the induced
morphism $\cE'\to\cE$ is an isomorphism. Let $E$ and $E'$ be Galois
covers of $\Df_{K}$ corresponding to $\cE$ and $\cE'$ respectively.
Since the morphism $\cE'\to\cE$ is representable, the Galois group
of $E'\to\Df_{K}$ can be embedded into the one of $E\to\Df_{K}$.
In particular, the degree of $E\to\Df_{K}$ is larger than or equal
to the one of $E'\to\Df_{K}$. The morphism $E'\times_{\cE}E\to E'$
is etale and finite. Therefore each component of $E'\times_{\cE}E$
is isomorphic to $E'$. For a component $F$, the morphism $E'\cong F\to E$
is a finite dominant $\Df$-morphism. But, the degree of $E\to\Df_{K}$
is larger than or equal to the one $E'\to\Df_{K}$. Therefore the
morphism $E'\to E$ is an isomorphism. This shows that the morphism
$\cE'\to\cE$ is an isomorphism. As a consequence, the isomorphism
class of the twisted arc $\cE\to\cX$ is determined by the induced
arc $\Df_{K}\to X$. This means the desired injectivity. We have proved
the first assertion of the proposition. The second assertion immediately
follows.
\end{proof}
\begin{defn}
Let $\Gamma=\Gamma_{\cX}$ be as in \ref{nota:Gamma}. Let $\Gamma\to Z$
be a morphism with $Z$ an algebraic space almost of finite type.
We define the \emph{motivic measure }$\mu_{\cX,Z}$ on $|\cJ_{\infty}\cX|=|\J_{\infty}\Utg_{\Gamma}(\cX)^{\pur}|$
to be $\mu_{\Utg_{\Gamma}(\cX)^{\pur}}$. For a measurable function
$h\colon A\to\frac{1}{r}\ZZ\cup\{\infty\}$ on a subset $A\subset|\cJ_{\infty}\cX|$,
the integral $\int_{A}\LL^{h}\,d\mu_{\cX,Z}\in\widehat{\cM_{Z,r}'}\cup\{\infty\}$
is defined as in \ref{def:meas-fn-integ}.
\end{defn}

Note that this motivic measure is independent of the choice of $\Gamma$.

Consider the following commutative diagram of natural morphisms among
formal DM stacks over $\Df_{\Gamma}$.
\begin{equation}
\xymatrix{\Utg_{\Gamma}(\cX)^{\pur}\times_{\Df_{\Gamma}}\cE_{\Gamma}\ar[d]_{r}\ar[rr]^{s} &  & \Utg_{\Gamma}(\cX)^{\pur}\ar[d]^{\pi^{\utg}}\\
\cX\ar[rr]_{\pi} &  & X
}
\label{diag-Utg-sq}
\end{equation}
Here $r$ is the universal morphism and $s$ is the projection. A
twisted arc $\gamma\colon\cE\to\cX$ say over an algebraically closed
field $K$ corresponds to an arc $\gamma'\colon\Df_{K}\to\Utg_{K}(\cX)^{\pur}$.
These lift to the same $\cE$-morphism $\tilde{\gamma}\colon\cE\to\Utg_{K}(\cX)^{\pur}\times_{\Df_{K}}\cE$.
We define the Jacobian ideal sheaves $\fj_{s}$ and $\fj_{r}$, regarding
$\Utg_{\Gamma}(\cX)^{\pur}\times_{\Df_{\Gamma}}\cE_{\Gamma}$ and
$\Utg_{\Gamma}(\cX)^{\pur}$ as formal DM stacks over $\Df_{\Gamma}$
and regarding $\cX$ as a formal DM stack over $\Df$.
\begin{defn}
\label{def:shift}We define the \emph{shift number }$\fs_{\cX}(\gamma)$
of $\gamma$ to be 
\[
\fs_{\cX}(\gamma):=\fj_{s}(\tilde{\gamma})-\fj_{r}(\tilde{\gamma})\in\ZZ.
\]
\end{defn}

Note that if $\cX$ is a formal algebraic space, then $\fs_{\cX}\equiv0$.
This follows from
\[
\Utg_{\Gamma}(\cX)^{\pur}=\Utg_{\Gamma_{[1]}}(\cX)^{\pur}=\cX.
\]

Since $r$ and $s$ are rig-etale (see \ref{lem:rig-etale-2}), we
always have $\fj_{s}(\tilde{\gamma})\ne\infty$ and $\fj_{r}(\tilde{\gamma})\ne\infty$.
If $\fj_{\pi}(\gamma)=\fj_{\tilde{\pi}}(\gamma)<\infty$, then $\fj_{\pi'}(\gamma)<\infty$.
From \ref{lem:jac-associativity}, outside the negligible subset $\fj_{\pi}^{-1}(\infty)$,
we have 
\begin{equation}
\fs_{\cX}(\gamma)=\fj_{\pi}(\gamma)-\fj_{\pi^{\utg}}(\gamma').\label{eq:s-j-j}
\end{equation}

\begin{lem}
\label{lem:ord-measurable}Let $\cI\subset\cO_{\cX}$ be an ideal
sheaf defining a narrow closed substack. Then the function $\ord_{\cI}$
on $|\cJ_{\infty}\cX|$ is measurable.
\end{lem}

\begin{proof}
Note that for the case of formal schemes, the lemma is well-known
and easy to show. It is straightforward to generalize it to formal
DM stacks as far as \emph{untwisted }arcs are concerned. We will reduce
the lemma to the case of untwisted arcs.

It suffices to show that for each connected component $\Gamma_{0}\subset\Gamma$
and $a\in\QQ_{\ge0}$, the set
\[
C:=(\ord_{\cI})^{-1}(a)\cap|\cJ_{\Gamma_{0},\infty}\cX|
\]
is cylindrical. Whether a twisted arc $\gamma\colon\cE\to\cX$ has
the value $\ord_{\cI}=a$ is determined by the induced twisted $n$-jet
for $n\ge\lfloor a\rfloor$. Namely we have $C=\pi_{n}^{-1}\pi_{n}(C)$.
We need to show that for $n\gg0$, $C_{n}:=\pi_{n}(C)$ is a constructible
subset. To do this, let us take an atlas $\Spec R\to\Gamma_{0}$ such
that the corresponding torsor over $\D_{R}^{*}$ is uniformizable.
Then we can write $\cE_{R}=[\Spf R\llbracket s\rrbracket/G]$. For
each $n$, we have maps:
\[
\xymatrix{\cJ_{\Spec R,n}\cX\ar[r]^{\alpha}\ar[d]_{\beta} & \cJ_{\Gamma_{0},n}\cX\\
\J_{\sharp G\cdot(n+1)-1}(\cX\times_{\Df}\Spf k\llbracket s\rrbracket)
}
\]
The map $\alpha$ is surjective. The $\beta$ sends a twisted $n$-jet
of $\cX$
\[
\left[\left(\Spec K\llbracket s\rrbracket/(t^{n+1})\right)/G\right]\to\cX
\]
to the induced untwisted $(\sharp G(n+1)-1)$-jet of $\cX\times_{\Df}\Spf k\llbracket s\rrbracket$
\begin{multline*}
\Spec K\llbracket s\rrbracket/(t^{n+1})=\Spec K\llbracket s\rrbracket/(s^{\sharp G(n+1)})\to\\
\left[\left(\Spec K\llbracket s\rrbracket/(t^{n+1})\right)/G\right]\to\cX\times_{\Df}\Spf k\llbracket s\rrbracket.
\end{multline*}
Let $\cI'$ be the pullback of $\cI$ to $\cX\times_{\Df}\Spf k\llbracket s\rrbracket$.
For $n\gg0$,  let 
\[
B_{a}\subset|\J_{\sharp G\cdot(n+1)-1}(\cX\times_{\Df}\Spf k\llbracket s\rrbracket)|
\]
be the locus where $\ord\cI'$ has value $\sharp G\cdot a$, which
is constructible. We see that $C_{a}=\alpha(\beta^{-1}(B_{a}))$ is
also constructible, which completes the proof. 
\end{proof}
\begin{lem}
\label{lem:jpi-measurable}The Jacobian order function $\fj_{\pi}$
on $|\cJ_{\infty}\cX|$ associated to the coarse moduli space $\pi\colon\cX\to X$
is measurable.
\end{lem}

\begin{proof}
As in the proof of \ref{lem:jac-measur}, we take a blowup $h\colon\cZ\to\cX$
such that $h^{\flat}\Omega_{\cX/\Df}^{d}$ and $(\pi\circ h)^{\flat}\Omega_{X/\Df}^{d}$
are invertible sheaves. The image of $(\pi\circ h)^{\flat}\Omega_{X/\Df}^{d}\to h^{\flat}\Omega_{\cX/\Df}^{d}$
is written as $\cI\cdot h^{\flat}\Omega_{\cX/\Df}^{d}$ for some locally
principal ideal sheaf $\cI\subset\cO_{\cZ}$, which defines a narrow
substack. Then $\fj_{\pi}\circ h_{\infty}$ is equal to $\ord_{\cI}$
outside a negligible subset. Now the lemma follows from \ref{lem:ord-measurable}
and \ref{lem:measurable-fnc}.
\end{proof}
\begin{cor}
\label{cor:sht-measurable}The function $\fs_{\cX}$ on $|\cJ_{\infty}\cX|$
is measurable.
\end{cor}

\begin{proof}
The functions $\fj_{\pi}$ and $\fj_{\pi'}$ are measurable from \ref{lem:jpi-measurable}
and \ref{lem:jac-measur} respectively. From (\ref{eq:s-j-j}), $\fs_{\cX}$
is also measurable.
\end{proof}
\begin{lem}
\label{lem:jac-shift-relation}Outside a negligible subset of $|\cJ_{\infty}\cY|$,
we have 
\[
\fs_{\cY}-\fs_{\cX}\circ f_{\infty}=\fj_{f}-\fj_{(\pi\circ f)^{\utg}}+\fj_{\pi^{\utg}}\circ f_{\infty}.
\]
In particular, if $\cX$ is a formal algebraic space, then
\[
\fs_{\cY}=\fj_{f}-\fj_{f^{\utg}}.
\]
\end{lem}

\begin{proof}
In this proof, we ignore negligible subsets and equalities below hold
outside some negligible subset. First consider the case where $\cX$
is a formal algebraic space. To a twisted arc $\gamma$ on $\cY$,
we define $\tilde{\gamma}$ and $\gamma'$ as above. Consider a diagram
similar to (\ref{diag-Utg-sq}) with $f$ in place of $\pi$. From
\ref{lem:jac-associativity},
\begin{align*}
\fj_{f}(\gamma)+\fj_{r}(\tilde{\gamma}) & =\fj_{f\circ r}(\tilde{\gamma})\\
 & =\fj_{f^{\utg}\circ s}(\tilde{\gamma})\\
 & =\fj_{s}(\tilde{\gamma})+\fj_{f^{\utg}}(\gamma').
\end{align*}
Therefore
\[
\fs_{\cY}(\gamma)=\fj_{s}(\tilde{\gamma})-\fj_{r}(\tilde{\gamma})=\fj_{f}(\gamma)-\fj_{f^{\utg}}(\gamma').
\]
Thus the second assertion of the lemma holds. We then apply this to
$\pi$ and $\pi\circ f$ to get 
\begin{align*}
\fs_{\cY} & =\fj_{\pi\circ f}-\fj_{(\pi\circ f)^{\utg}},\\
\fs_{\cX} & =\fj_{\pi}-\fj_{\pi^{\utg}}.
\end{align*}
Therefore
\begin{align*}
\fs_{\cY}-\fs_{\cX}\circ f_{\infty} & =\fj_{\pi\circ f}-\fj_{\pi}\circ f_{\infty}-\fj_{(\pi\circ f)^{\utg}}+\fj_{\pi^{\utg}}\circ f_{\infty}\\
 & =\fj_{f}-\fj_{(\pi\circ f)^{\utg}}+\fj_{\pi^{\utg}}\circ f_{\infty},
\end{align*}
where the last equality follows from \ref{lem:jac-associativity}.
We have obtained the first equality of the lemma.
\end{proof}
\begin{rem}
If we constructed $f^{\utg}\colon\Utg_{\Gamma_{\cY}}(\cY)^{\pur}\to\Utg_{\Gamma_{\cX}}(\cX)^{\pur}$
(see \ref{rem:utg-map-general}), then we would get
\[
\fs_{\cY}-\fs_{\cX}\circ f_{\infty}=\fj_{f}-\fj_{f^{\utg}}
\]
in place of the equality of \ref{lem:jac-shift-relation}.
\end{rem}

\begin{cor}
The function $\fj_{f}$ on $|\cJ_{\infty}\cY|$ is measurable.
\end{cor}

\begin{proof}
From \ref{lem:jac-shift-relation}, outside a negligible subset, we
have
\[
\fj_{f}=\fs_{\cY}-\fs_{\cX}\circ f_{\infty}+\fj_{(\pi\circ f)^{\utg}}-\fj_{\pi^{\utg}}\circ f_{\infty}.
\]
From \ref{lem:measurable-fnc}, \ref{lem:jac-measur} and \ref{cor:sht-measurable},
the four functions on the right side are measurable. Therefore $\fj_{f}$
is also measurable.
\end{proof}
Generalizing maps (\ref{map1}) and (\ref{map2}) associated to $\pi\colon\cX\to X$,
we now define maps $|\cJ_{\infty}\cY|\to|\cJ_{\infty}\cX|$ and $(\cJ_{\infty}\cY)[K]\to(\cJ_{\infty}\cX)[K]$
associated to the given morphism $f\colon\cY\to\cX$. For a twisted
arc $\gamma\colon\cE\to\cY$ over an algebraically closed field $K$,
let 
\[
\cE\to\cE'\xrightarrow{\gamma'}\cX
\]
be the canonical factorization of the composition $\cE\to\cY\to\cX$
as in \cite[Lem. 25]{MR2271984} (cf. \cite[Th. 3.1]{MR2786662})
so that $\gamma'$ is a twisted arc. Note that the paper \cite{MR2271984}
treats the non-formal situation, but it is straightforward to generalize
it to the formal situation. If $\cE=[E/G]$ for a $G$-cover $E\to\Df_{K}$
and if $y$ and $x$ denote the induced $K$-points of $\cY$ and
$\cX$ respectively, then we have maps
\[
G\to\Aut_{\cY}(y)\to\Aut_{\cX}(x),
\]
the left one being injective. Let $N\lhd G$ be the kernel of this
composite map. Then $\cE'\cong[(E/N)/(G/N)]$. We have the map
\[
(\cJ_{\infty}\cY)[K]\to(\cJ_{\infty}\cX)[K],\,[\gamma]\mapsto[\gamma']
\]
as well as 
\[
|\cJ_{\infty}\cY|\to|\cJ_{\infty}\cX|.
\]
We denote these maps by $f_{\infty}$. We say that $f_{\infty}|_{A}$
is \emph{geometrically injective }if for every algebraically closed
field $K$, $f_{\infty}|_{A[K]}$ is injective. We say that $f_{\infty}|_{A}$
is \emph{almost geometrically injective }if there exists a negligible
subset $A'$ such that $f_{\infty}|_{A\setminus A'}$ is geometrically
injective.

The following is the \emph{change of variables formula} for twisted
arcs.
\begin{thm}
\label{thm:change-vars-II}We keep Assumption \ref{assu:tw}. Let
$A\subset|\cJ_{\infty}\cY|$ be a subset such that $f_{\infty}|_{A}$
is almost geometrically injective. Let $h\colon f_{\infty}(A)\to\frac{1}{r}\ZZ\cup\{\infty\}$
be a measurable function. Then 
\[
\int_{f_{\infty}(A)}\LL^{h+\fs_{\cX}}\,d\mu_{\cX}=\int_{A}\LL^{h\circ f_{\infty}-\fj_{f}+\fs_{\cY}}\,d\mu_{\cY}\in\widehat{\cM_{k,r}'}\cup\{\infty\}.
\]
\end{thm}

\begin{proof}
First consider the case where $\cX$ is a formal algebraic space.
From \ref{lem:rig-etale-2}, the morphism $f^{\utg}\colon\Utg_{\Gamma_{\cY}}(\cY)^{\pur}\to\cX$
satisfy Assumption \ref{assu:X} with $\Psi=\Gamma_{\cY}$ and $\Phi=\Spec k$.

By definition, $|\cJ_{\infty}\cY|$ is identical to $|\J_{\infty}\Utg_{\Gamma_{\cY}}(\cY)^{\pur}|$.
Moreover the two measures $\mu_{\Utg_{\Gamma_{\cY}}(\cY)^{\pur}}$
and $\mu_{\cY}$ on this space are equal. Applying \ref{thm:change-vars-I}
to $f^{\utg}$, we obtain
\begin{align*}
\int_{f_{\infty}(A)}\LL^{h}\,d\mu_{\cX} & =\int_{A}\LL^{h\circ f_{\infty}-\fj_{f^{\utg}}}\,d\mu_{\Utg_{\Gamma_{\cY}}(\cY)^{\pur}}=\int_{A}\LL^{h\circ f_{\infty}-\fj_{f^{\utg}}}\,d\mu_{\cY}.
\end{align*}
From \ref{lem:jac-shift-relation}, this is equal to
\[
\int_{A}\LL^{h\circ f_{\infty}-\fj_{f}+\fs_{\cY}}\,d\mu_{\cY}.
\]
Thus the theorem holds in this case.

For the general case, let $\pi\colon\cX\to X$ be the coarse moduli
space. Let $\overline{h}$ and $\overline{\fj_{\pi}}$ be the functions
corresponding to $h$ and $\fj_{\pi}$ respectively through the almost
geometric bijection $\pi_{\infty}\colon|\cJ_{\infty}\cX|\to|\J_{\infty}X|$
(see \ref{lem:almost-bij}). From \ref{lem:jac-associativity}, outside
a negligible subset, we have
\[
\overline{\fj_{\pi}}\circ(\pi\circ f)_{\infty}-\fj_{\pi\circ f}=\fj_{\pi}\circ f_{\infty}-\fj_{\pi\circ f}=-\fj_{f}.
\]
Therefore, from the special case discussed above, we have

\begin{align*}
\int_{A}\LL^{h\circ f_{\infty}-\fj_{f}+\fs_{\cY}}\,d\mu_{\cY} & =\int_{A}\LL^{\bar{h}\circ(\pi\circ f)_{\infty}+\overline{\fj_{\pi}}\circ(\pi\circ f)_{\infty}-\fj_{\pi\circ f}+\fs_{\cY}}\,d\mu_{\cY}\\
 & =\int_{(\pi\circ f)_{\infty}(A)}\LL^{\overline{h}+\overline{\fj_{\pi}}}\,d\mu_{X}\\
 & =\int_{f_{\infty}(A)}\LL^{\overline{h}\circ\pi_{\infty}+\overline{\fj_{\pi}}\circ\pi_{\infty}-\fj_{\pi}+\fs_{\cX}}\,d\mu_{\cX}\\
 & =\int_{f_{\infty}(A)}\LL^{h+\fs_{\cX}}\,d\mu_{\cX}.
\end{align*}
\end{proof}

\section{Log pairs\label{sec:Log-pairs}}

In birational geometry, one often consider log pairs; a log pair means
the pair of a normal variety $X$ and a $\QQ$-divisor $D$ on $X$
such that $K_{X}+D$ is $\QQ$-Cartier. In this section, we generalize
this notion by allowing $X$ to be a formal DM stack. In the next
section, we will associate an invariant to a generalized log pair
and apply the change of variables formula to study properties of this
invariant.

Let $\cX$ be a formal stack over $\Df_{\Phi}$ as in \ref{assu:X}.
Moreover we suppose that $\cX$ is normal; local rings $\cO_{\cX,x}$
of all geometric points are normal.
\begin{defn}
\label{def:divisors}A closed substack $\cH\subset\cX$ is said to
be \emph{irreducible }if for any closed substacks $\cH_{i}\subset\cX$,
$i=1,2$ such that $\cH=\cH_{1}\cup\cH_{2}$, we have either $\cH=\cH_{1}$
or $\cH=\cH_{2}$. A \emph{prime divisor }on $\cX$ means an irreducible
and reduced closed substack $B\subset\cX$ of codimension one. A \emph{divisor}
(resp.\ $\QQ$-\emph{divisor})$A$ on $\cX$ is a formal combination
$\sum_{i=1}^{l}c_{i}A_{i}$, $c_{i}\in\ZZ$ (resp.\ $c_{i}\in\QQ$)
of prime divisors $A_{i}$ on $\cX$. We say that a ($\QQ$-)divisor
$A=\sum_{i}c_{i}A_{i}$ is \emph{narrow }if for every $i$ with $c_{i}\ne0$,
the closed substack $A_{i}\subset\cX$ is narrow. The support $\Supp A$
of $A$ is defined to be the reduced closed substack $\bigcup_{c_{i}\ne0}A_{i}$.
\end{defn}

We need to consider fractional ideals of rings which are not necessarily
integral domains; for this notion which is not so common in this generality,
we refer the reader to \cite{MR0414528}.
\begin{defn}
Let $\cS\subset\cO_{\cX}$ be the subsheaf of sections which are nonzerodivisors
at the stalk $\cO_{\cX,x}$ of every geometric point $x\colon\Spec K\to\cX$.
We define the \emph{sheaf of total quotient rings, }denoted by $\cK_{\cX}$,
to be the sheaf associated to the presheaf $U\mapsto\cS(U)^{-1}\cO_{\cX}(U)$.
A \emph{fractional ideal sheaf }on $\cX$ is a coherent $\cO_{\cX}$-submodule
$\cL\subset\cK_{\cX}$ such that for every geometric point $\Spec K\to\cX$,
there exist an etale neighborhood $U\to\cX$ and $f\in\cS(U)$ such
that $f\cdot\cL|_{U}\subset\cO_{\cX}|_{U}$.

Let $\cF$ be a torsion-free coherent sheaf on $\cX$ of generic rank
1. A \emph{fractional submodule }of $\cF$ is a coherent $\cO_{\cX}$-submodule
$\cL$ of $\cS^{-1}\cF=\cF\otimes_{\cO_{\cX}}\cK_{\cX}$ such that
for every geometric point $\Spec K\to\cX$, there exist an etale neighborhood
$U\to\cX$ and $f\in\cS(U)$ such that $f\cdot\cL|_{U}\subset\cF|_{U}$.
\end{defn}

\begin{defn}
For a divisor $A=\sum_{i}c_{i}A_{i}$, we define a fractional ideal
sheaf $\cO_{\cX}(A)$ in the usual way: for a local section $s$ of
$\cK_{\cX}$, it belongs to $\cO_{\cX}(A)$ if and only if for each
$i$, the order of $s$ along $A_{i}$ is $\ge-c_{i}$.
\end{defn}

\begin{defn}
We define the \emph{canonical sheaf }$\omega_{\cX/\Df_{\Phi}}$ to
be $(\Omega_{\cX/\Df_{\Phi}}^{d})^{\vee\vee}$, the double dual (that
is, the reflexive hull) of $\Omega_{\cX/\Df_{\Phi}}^{d}$. For $r\in\NN$,
we define $\omega_{\cX/\Df_{\Phi}}^{[r]}$ to be the reflexive $r$-th
power $\left((\omega_{\cX/\Df_{\Phi}})^{\otimes r}\right)^{\vee\vee}$.
For a divisor $A$, we define $\omega_{\cX/\Df_{\Phi}}^{[r]}(A)$
to be $\left(\omega_{\cX/\Df_{\Phi}}^{[r]}\otimes\cO_{\cX}(A)\right)^{\vee\vee}$.
A \emph{log pair }means the pair $(\cX,A)$ of a formal stack $\cX$
as above and a narrow $\QQ$-divisor $A$ on it such that for some
$r>0$, $rA$ has integral coefficients and $\omega_{\cX/\Df_{\Phi}}^{[r]}(rA)$
is an invertible sheaf. (Namely $K_{\cX/\Df_{\Phi}}+A$ is $\QQ$-Cartier,
according to the terminology which is little inaccurate but common
in the birational geometry.) We call $A$ a \emph{boundary divisor
}or simply \emph{boundary. }If $\omega_{\cX/\Df_{\Phi}}^{[r]}$ is
invertible for some $r>0$, then we identify $\cX$ with the log pair
$(\cX,0)$.

A \emph{morphism $(\cY,B)\to(\cX,A)$ }of log pairs is just a morphism
$\cY\to\cX$ of ambient stacks satisfying \ref{assu:X}. We say that
a morphism $f\colon(\cY,B)\to(\cX,A)$ is \emph{crepant }if the morphism
$f^{*}\Omega_{\cX/\Df_{\Phi}}^{d}\to\Omega_{\cY/\Df_{\Psi}}^{d}$
induces an isomorphism $f^{*}\omega_{\cX/\Df_{\Phi}}^{[r]}(rA)\to\omega_{\cY/\Df_{\Psi}}^{[r]}(rB)$
of fractional modules for sufficiently factorial $r$, (that is, $f^{*}(K_{\cX/\Df_{\Phi}}+A)=K_{\cY/\Df_{\Psi}}+B$).
\end{defn}

\begin{lem}
\label{lem:boundary-coarse}Let $(\cX,A)$ be a log pair and $X$
the coarse moduli space of $\cX$. There exists a unique boundary
$\overline{A}$ on $X$ such that the morphism $(\cX,A)\to(X,\overline{A})$
is crepant. 
\end{lem}

\begin{proof}
Consider the case where $\cX$ is a quotient stack $[V/G]$ for a
finite group $G$. Then $X=V/G$. We have the relative canonical divisor
$K_{V/X}$, a divisor on $V$. Let $\tilde{A}$ be the $\QQ$-divisor
on $V$ corresponding to $A$ and let
\[
\overline{A}:=\frac{1}{\sharp G}\pi_{*}(K_{V/X}+\tilde{A})
\]
with $\pi$ denoting the morphism $V\to X$. Then 
\[
\pi^{*}(K_{X}+\overline{A})=\pi^{*}K_{X}+K_{V/X}+\tilde{A}=K_{V}+\tilde{A}.
\]
This means that $(\cX,A)\to(X,\overline{A})$ is crepant. It is clear
that $\overline{A}$ is the unique boundary with this property. The
uniqueness allows us to glue $\overline{A}$'s defined etale locally
to get the desired boundary $\overline{A}$ in the general case.
\end{proof}
\begin{defn}
Let $(\cX,A)$ be a log pair and let $r$ be a positive integer such
that $\omega_{\cX/\Df_{\Phi}}^{[r]}(rA)$ is invertible. Let $L/K\tpars$
be a finite extension and let $\beta\colon\Spf O_{L}\to\cX$ be a
$\Df$-morphism which do not factor through $\Supp A$. Then $\beta^{*}\omega_{\cX/\Df_{\Phi}}^{[r]}(rA)$
is a nonzero fractional submodule of $\beta^{\flat}(\Omega_{\cX/\Df_{\Phi}}^{d})^{\otimes r}$.
We define 
\[
\ff_{(\cX,A)}(\beta):=\frac{1}{r}\ord[\beta^{\flat}(\Omega_{\cX/\Df_{\Phi}}^{d})^{\otimes r}:\beta^{*}\omega_{\cX/\Df_{\Phi}}^{[r]}(rA)].
\]
Here $[M:N]:=\{s\in O_{L}\mid sN\subset M\}$, a fractional ideal
of $O_{L}$. When $\cX$ is defined over $\Df$, for a twisted arc
$\gamma\colon\cE\to\cX$ which do not factor through $\Supp A$, we
define $\ff_{(\cX,A)}(\gamma)$ to be $\ff_{(\cX,A)}(\beta)$ for
$\beta\colon\Spf O_{L}\to\cE\to\cX$ induced by any finite dominant
$\Df$-morphism $\Spf O_{L}\to\cE$.
\end{defn}

\begin{lem}
The function $\ff_{(\cX,A)}$ on $|\J_{\Phi,\infty}\cX|$ is measurable.
When $\Phi=\Spec k$, the function $\ff_{(\cX,A)}$ on $|\cJ_{\infty}\cX|$
is also measurable.
\end{lem}

\begin{proof}
There exists an ideal sheaf $\cJ$ such that 
\[
\cJ(\Omega_{\cX/\Df_{\Phi}}^{d})^{\otimes r}\subset\omega_{\cX/\Df_{\Phi}}^{[r]}(rA).
\]
As in the proof of \ref{lem:jac-measur}, there exists a blowup $h\colon\cZ\to\cX$
such that $h^{-1}\cJ$ and $h^{\flat}(\Omega_{\cX/\Df_{\Phi}}^{d})^{\otimes r}$
are invertible. Then, for some locally principal ideal sheaf $\cI$
on $\cZ$, 
\[
h^{-1}\cJ\cdot h^{\flat}(\Omega_{\cX/\Df_{\Phi}}^{d})^{\otimes r}=\cI\cdot h^{*}\omega_{\cX/\Df_{\Phi}}^{[r]}(rA).
\]
Outside a negligible subset, we have
\[
\ff_{(\cX,A)}\circ h_{\infty}=\frac{1}{r}(\ord_{\cI}-\ord_{h^{-1}\cJ}).
\]
From \ref{lem:ord-fn-measur} and \ref{lem:ord-measurable}, the right
side is measurable and so is the left side. From \ref{lem:measurable-fnc},
$\ff_{(\cX,A)}$ is measurable.
\end{proof}
\begin{lem}
\label{lem:crepant-j'}Let $f\colon(\cY,B)\to(\cX,A)$ be a crepant
morphism of log pairs. We have the following equality between functions
on $|\cJ_{\infty}\cY|$ outside a negligible subset,
\[
\ff{}_{(\cX,A)}\circ f_{\infty}=\ff{}_{(\cY,B)}+\fj_{f}.
\]
\end{lem}

\begin{proof}
Let $\beta\colon\Spf O_{L}\to\cY$ be a morphism as before. For a
sufficiently factorial $r$, we consider three $O_{L}$-modules 
\begin{gather*}
M_{1}:=\beta^{*}f^{*}\omega_{\cX/\Df_{\Phi}}^{[r]}(rA)=\beta^{*}\omega_{\cY/\Df_{\Psi}}^{[r]}(rB),\\
M_{2}:=\beta^{\flat}(\Omega_{\cY/\Df_{\Psi}}^{d})^{\otimes r},\\
M_{3}:=(f\circ\beta)^{\flat}(\Omega_{\cX/\Df_{\Psi}}^{d})^{\otimes r}.
\end{gather*}
Ignoring a negligible subset, we may suppose that they are free of
rank one. We have
\begin{align*}
[M_{2}:M_{1}] & =\fm_{E}^{\sharp G\cdot r\cdot\ff_{(\cY,B)}(\beta)},\\{}
[M_{3}:M_{1}] & =\fm_{E}^{\sharp G\cdot r\cdot\ff_{(\cX,A)}(f\circ\beta)},\\{}
[M_{3}:M_{2}] & =\fm_{E}^{\sharp G\cdot r\cdot\fj_{f}(\beta)}.
\end{align*}
The lemma follows from the equality 
\[
[M_{3}:M_{1}]=[M_{3}:M_{2}][M_{2}:M_{1}].
\]
\end{proof}
For a log pair $(\cX,A)$, let $X$ be the coarse moduli space of
$\cX$ and let $\overline{A}$ be the $\QQ$-divisor such that $(\cX,A)\to(X,\overline{A})$
is crepant. Let $\Utg_{\Gamma}(\cX)^{\nor}$ be the normalization
of $\Utg_{\Gamma}(\cX)^{\pur}$. Let $A^{\utg}$ be the unique $\QQ$-divisor
such that $(\Utg_{\Gamma}(\cX)^{\nor},A^{\utg})\to(X,\overline{A})$
is crepant; with the conventional notation, if $\pi\colon\cX\to X$
is the given morphism and if $\phi\colon\Utg_{\Gamma}(\cX)^{\nor}\to X$
is the induced morphism, then 
\[
A^{\utg}=\phi^{*}(K_{X/\Df}+\overline{A})-K_{\Utg_{\Gamma}(\cX)^{\nor}/\Df_{\Gamma}}.
\]

\begin{lem}
\label{lem:shift-j'}If $\nu$ denotes the normalization morphism
$\Utg_{\Gamma}(\cX)^{\nor}\to\Utg_{\Gamma}(\cX)$, then we have the
following equality of functions on $|\J_{\infty}\Utg_{\Gamma}(\cX)^{\nor}|$
outside a negligible subset,
\[
\fs_{\cX}\circ\nu_{\infty}=\ff_{(\Utg_{\Gamma}(\cX)^{\nor},A^{\utg})}-\ff_{(\cX,A)}\circ\nu_{\infty}+\fj_{\nu}.
\]
(Note that $\ff_{(\cX,A)}$ is a function on $|\cJ_{\infty}\cX|=|\J_{\infty}\Utg_{\Gamma}(\cX)^{\pur}|.$)
\end{lem}

\begin{proof}
By \ref{lem:jac-shift-relation}, 
\[
\fs_{\cX}=\fj_{\pi}-\fj_{\pi^{\utg}}.
\]
From \ref{lem:crepant-j'}, 
\begin{gather*}
\fj_{\pi}=\ff_{(X,\overline{A})}\circ\pi_{\infty}-\ff_{(\cX,A)},\\
\fj_{\phi}=\ff_{(X,\overline{A})}\circ\phi_{\infty}-\ff_{(\Utg_{\Gamma}(\cX)^{\nor},A^{\utg})}.
\end{gather*}
From \ref{lem:jac-associativity}, 
\[
\fj_{\phi}=\fj_{\nu}+\fj_{\pi^{\utg}}\circ\nu_{\infty}.
\]
Thus 
\begin{align*}
\fs_{\cX}\circ\nu_{\infty} & =\left(\ff_{(X,\overline{A})}\circ\pi_{\infty}-\ff_{(\cX,A)}\right)\circ\nu_{\infty}-\fj_{\pi^{\utg}}\circ\nu_{\infty}\\
 & =\ff_{(X,\overline{A})}\circ\phi_{\infty}-\ff_{(\cX,A)}\circ\nu_{\infty}+\fj_{\nu}-\fj_{\phi}\\
 & =\ff_{(\Utg_{\Gamma}(\cX)^{\nor},A^{\utg})}-\ff_{(\cX,A)}\circ\nu_{\infty}+\fj_{\nu}.
\end{align*}
\end{proof}

\section{Stringy motives and the wild McKay correspondence\label{sec:Stringy-invariants}}

We associate stringy motives to log pairs and reformulate the change
of variables formula as a property of stringy motives. This invariant
is a common generalization of stringy E-function and orbifold E-function
\cite{MR1677693}. Following Denef-Loeser's approach \cite{MR1905024},
we define this invariant as an integral on the space of twisted arcs,
$|\cJ_{\infty}\cX|$, even if the ambient formal stack $\cX$ is singular.

We fix a sufficiently factorial positive integer $r$ so that all
relevant functions on spaces of (twisted or untwisted) arcs have values
in $\frac{1}{r}\ZZ\cup\{\infty\}$.
\begin{defn}
\label{def:stringy-motive}Let $(\cX,A)$ be a log pair over $\Df$
and let $C\subset|\cX|$ be a constructible subset. Let $|\cJ_{\infty}\cX|_{C}\subset|\cJ_{\infty}\cX|$
be the preimage of $C$ by $\pi_{0}$. Let $\Gamma=\Gamma_{\cX}\to Z$
be a morphism to an algebraic space almost of finite type. We define
the \emph{stringy motive of $(\cX,A)$ along $C$ }as 
\begin{align*}
\M_{\st,Z}(\cX,A)_{C} & :=\int_{|\cJ_{\infty}\cX|_{C}}\LL^{\ff_{(\cX,A)}+\fs_{\cX}}\,d\mu_{\cX,Z}\in\widehat{\cM_{Z,r}'}\cup\{\infty\}.
\end{align*}
\end{defn}

\begin{defn}
\label{def:stringy-motive-utd}Let $(\cX,A)$ be a log pair over $\Df_{\Phi}$
with $\Phi$ a reduced DM stack almost of finite type and let $C\subset|\cX|$
be a constructible subset. Let $|\J_{\infty}\cX|_{C}\subset|\J_{\infty}\cX|$
be the preimage of $C$ by $\pi_{0}$. Let $\Phi\to Z$ be a morphism
to an algebraic space almost of finite type. We define the \emph{untwisted
stringy motive of $(\cX,A)$ along $C$ }as 
\begin{align*}
\M_{\st,Z}^{\utd}(\cX,A)_{C} & :=\int_{|\J_{\infty}\cX|_{C}}\LL^{\ff_{(\cX,A)}}\,d\mu_{\cX,Z}\in\widehat{\cM_{Z,r}'}\cup\{\infty\}.
\end{align*}
\end{defn}

For these invariants, we usually omit the subscript $Z$ when $Z=\Spec k$
and $C$ when $C=|\cX|$.

We have 
\[
\int_{Z\to Z'}\M_{\st,Z}(\cX,A)_{C}=\M_{\st,Z'}(\cX,A)_{C},
\]
in particular, 
\[
\int_{Z}\M_{\st,Z}(\cX,A)_{C}=\M_{\st}(\cX,A).
\]
Similarly for the untwisted stringy motive. When $\cX$ is a formal
algebraic space over $\Df$, obviously $\M_{\st,Z}(\cX,A)_{C}=\M_{\st,Z}^{\utd}(\cX,A)_{C}$.

Let $(\cX,A)$ be a log pair over $\Df$ and let $(X,\overline{A})$
and $(\Utg_{\Gamma}(\cX)^{\nor},A^{\utg})$ be the induced log pairs
(see Section \ref{sec:Log-pairs}). Let $C\subset|\cX|$ be a constructible
subset, $\overline{C}\subset|X|$ its image of $C$ and $C^{\utg}\subset|\Utg_{\Gamma}(\cX)^{\nor}|$
the preimage of $\overline{C}$.
\begin{thm}
\label{thm:stringy-coarse-unt}We keep the notation. Let $\overline{\Gamma}$
be the coarse moduli space of $\Gamma$. We have 
\begin{align*}
\M_{\st}(\cX,A)_{C} & =\M_{\st}(X,\overline{A})_{\overline{C}}\\
 & =\M_{\st}^{\utd}(\Utg_{\Gamma}(\cX)^{\nor},A^{\utg})_{C^{\utg}}\\
 & =\int_{\overline{\Gamma}}\M_{\st,\overline{\Gamma}}^{\utd}(\Utg_{\Gamma}(\cX)^{\nor},A^{\utg})_{C^{\utg}}.
\end{align*}
In particular, if $A=0$ and the stacky locus $\cX_{\st}\subset\cX$
has codimension $\ge2$, then 
\[
\M_{\st}(\cX)_{C}=\M_{\st}(X)_{\overline{C}}.
\]
\end{thm}

\begin{proof}
From the change of variables, \ref{lem:jac-shift-relation}, and from
\ref{lem:crepant-j'}, we have
\begin{align*}
\int_{|\cJ_{\infty}\cX|_{C}}\LL^{\ff_{(\cX,A)}+\fs_{\cX}}\,d\mu_{\cX} & =\int_{|\J_{\infty}X|_{\overline{C}}}\LL^{\ff_{(X,\overline{A})}}\,d\mu_{X}\\
 & =\int_{|\J_{\infty}\Utg_{\Gamma}(\cX)^{\nor}|_{C^{\utg}}}\LL^{\ff_{(\Utg_{\Gamma}(\cX)^{\nor},A^{\utd})}}\,d\mu_{\Utg_{\Gamma}(\cX)}\\
 & =\int_{\overline{\Gamma}}\int_{|\J_{\infty}\Utg_{\Gamma}(\cX)^{\nor}|_{C^{\utg}}}\LL^{\ff_{(\Utg_{\Gamma}(\cX)^{\nor},A^{\utd})}}\,d\mu_{\Utg_{\Gamma}(\cX),\overline{\Gamma}}.
\end{align*}
These equalities show the first assertion. If $A=0$ and the stacky
locus $\cX_{\st}\subset\cX$ has codimension $\ge2$, then $\overline{A}=0$
and the second assertion follows.
\end{proof}
We apply this theorem to the case of a quotient stack. Let $V$ be
a formal scheme of finite type, flat and generically smooth over $\Df$.
Suppose that a finite group $H$ acts on $V$ and that the quotient
stack $\cX:=[V/H]$ is spacious, equivalently that if $V^{h}$, $h\in H$
are the $h$-fixed subscheme, then $\bigcup_{h\in H\setminus\{1\}}V^{h}$
is a narrow subscheme of $V$. Let $(V,A)$ be a log pair such that
$A$ is stable under the $H$-action. Let $(V/H,\overline{A})$ be
the induced log pair such that $(V,A)\to(V/H,\overline{A})$ is crepant.
Let $\Gamma:=\Gamma_{\cX}$ for $\cX$ and let $\Gamma_{G}:=\Gamma\times_{\Theta}\Theta_{G}$.
Then $\Gamma_{G}\to\Theta_{G}$ is representable and universally bijective.
Moreover, $\Utg_{\Gamma_{G},\iota}(V)^{\pur}$ is flat over $\Df_{\Gamma_{G}}$
for every $\iota\in\Emb(G,H)$. We denote by $\Utg_{\Gamma_{G},\iota}(V)^{\nor}$
its normalization. Let $C\subset|V|$ be an $H$-stable constructible
subset and $\overline{C}\subset|V/H|$ its image. For an embedding
$\iota\colon G\hookrightarrow H$ of a Galoisian group, let $A_{\iota}$
be the restriction of $A^{\utg}$ to $[\Utg_{\Gamma_{G},\iota}(V)^{\nor}/\N_{H}(\iota(G))]$
and let $C_{\iota}\subset\left|[\Utg_{\Gamma_{G},\iota}(V)^{\nor}/\N_{H}(\iota(G))]\right|$
be the preimage of $\overline{C}$.
\begin{cor}
\label{cor:stringy-quot}We have
\[
\M_{\st}(V/H,\overline{A})_{\overline{C}}=\sum_{G\in\GG}\sum_{\iota\in\frac{\Emb(G,H)}{\Aut(G)\times H}}\M_{\st}^{\utd}\left(\left[\frac{\Utg_{\Gamma_{G},\iota}(V)^{\nor}}{\N_{H}(\iota(G))}\right],A_{\iota}\right)_{C_{\iota}}.
\]
\end{cor}

\begin{proof}
Let $\cX=[V/H]$ and define $(\Utg_{\Gamma}(\cX)^{\nor},A^{\utg})$
as before. From \ref{thm:stringy-coarse-unt}, 
\[
\M_{\st}(V/H,\overline{A})_{\overline{C}}=\M_{\st}^{\utd}(\Utg_{\Gamma}(\cX)^{\nor},A^{\utg})_{C^{\utg}}.
\]
From \ref{lem:Unt-quot} and (\ref{eq:utg-N_H}),
\begin{align*}
\Utg_{\Gamma}(\cX) & =\coprod_{G\in\GG}\Utg_{\Gamma_{[G]}}(\cX)\\
 & \cong\coprod_{G\in\GG}\left[\left(\coprod_{\iota\in\Emb(G,H)}\Utg_{\Gamma_{G},\iota}(V)\right)/\Aut(G)\times H\right]\\
 & \cong\coprod_{G\in\GG}\coprod_{\iota\in\Emb(G,H)/(\Aut(G)\times H)}[\Utg_{\Theta_{G},\iota}(V)/\N_{H}(\iota(G))].
\end{align*}
Therefore $\M_{\st}^{\utd}(\Utg_{\Gamma}(\cX)^{\nor},A^{\utg})_{C^{\utg}}$
is equal to the right side of the equality of the lemma.
\end{proof}
\begin{rem}
\label{rem:mot-version}This corollary is the motivic version of the
main result of \cite{MR3730512}, formulated by using stacks. Note
that the untwisting stack considered in this paper is essentially
the same as the one considered in \cite{MR3508745,MR3730512} (see
Remark \ref{rem:embed-linear}).
\end{rem}

\begin{defn}
Let $f\colon\cY\to\cX$ be a morphism of formal DM stacks in \ref{assu:tw}.
We say that $f\colon\cY\to\cX$ is a \emph{pseudo-modification }if
$f_{\infty}\colon|\cJ_{\infty}\cX|\to|\cJ_{\infty}\cX|$ is almost
geometrically bijective.
\end{defn}

\begin{example}
The following are examples of pseudo-modification;
\begin{enumerate}
\item the $t$-adic completion of a (not necessarily representable) proper
birational morphism $\cW\to\cV$ of DM stacks of finite type over
$\D=\Spec k\tbrats$,
\item the blowup along a closed substack,
\item the relative coarse moduli space $\cY\to\cX$ of some morphism $\cY\to\cZ$
of formal DM stacks of finite type over $\Df$.
\end{enumerate}
A composition of finitely many such morphisms is also a pseudo-modification.
\end{example}

\begin{thm}
\label{thm:Mst-crepant}Let $f\colon\cY\to\cX$ be a pseudo-modification
which induces a crepant morphism $(\cY,B)\to(\cX,A)$ of log pairs
over $\Df$. Let $C\subset|\cX|$ be a constructible subset and $\tilde{C}\subset|\cY|$
its preimage. Then 
\[
\M_{\st}(\cX,A)_{C}=\M_{\st}(\cY,B)_{\tilde{C}}.
\]
\end{thm}

\begin{proof}
With notation as before, we have 
\[
\M_{\st}(\cX,A)_{C}=\M_{\st}(X,\overline{A})_{\overline{C}}.
\]
Applying \ref{thm:change-vars-II} to the crepant morphism $(\cY,B)\to(X,\overline{A})$,
we see 
\[
\M_{\st}(X,\overline{A})_{\overline{C}}=\M_{\st}(\cY,B)_{\tilde{C}}.
\]
\end{proof}

\section{Linear actions\label{sec:Linear-actions}}

In this section, we focus on a linear action of a finite group on
an affine space over $\Df$ and the associated quotient stack/scheme.
We specialize results obtained in earlier sections to this situation
and get a more explicit formula, the wild McKay correspondence.

Let $K$ be an algebraically closed field, let $M$ be a free $K\tbrats$-module
of rank $d$ and let $S^{\bullet}M$ be its symmetric algebra over
$K\tbrats$, which is isomorphic to the polynomial ring $K\tbrats[x_{1},\dots,x_{d}]$.
Suppose that $M$ is given a $K\tbrats$-linear action of a finite
group $H$, which induces an $H$-action on the scheme $W=\Spec S^{\bullet}M\cong\AA_{\D_{K}}^{d}$
and one on the formal scheme $V:=\Spf\widehat{S^{\bullet}M}\cong\AA_{\Df_{K}}^{d}$,
where $\widehat{S^{\bullet}M}$ is the $t$-adic completion of $S^{\bullet}M$.
Let $\iota\colon G\hookrightarrow H$ be an embedding of a Galoisian
group. For a $G$-cover $E\to\Df_{K}$, we define the \emph{tuning
module}
\[
\Xi_{E,\iota}:=\Hom_{K\tbrats}^{\iota}(M,O_{E}),
\]
the module of $\iota$-equivariant $K\tbrats$-linear maps \cite[page 127]{MR3508745}
(cf. \cite[Def. 3.1]{MR3431631}). This is again a free $K\tbrats$-module
of rank $d$. Moreover this is also a saturated $K\tbrats$-submodule
of $\Hom_{K\tbrats}(M,O_{E})$, that is, the quotient module $\Hom_{K\tbrats}(M,O_{E})/\Xi_{E}$
is a torsion-free (hence free) $K\tbrats$-module. It follows that
the map of dual $K\tbrats$-modules 
\[
\Hom_{K\tbrats}(M,O_{E})^{\vee}\to\Xi_{E,\iota}^{\vee}
\]
 is surjective.

Let $M_{\iota}^{|E|}:=\Xi_{E,\iota}^{\vee}$ and let 
\begin{gather*}
W_{\iota}^{|E|}:=\Spec S^{\bullet}(M_{\iota}^{|E|}),\\
V_{\iota}^{|E|}:=\Spf\widehat{S^{\bullet}(M_{\iota}^{|E|})}=\widehat{W_{\iota}^{|E|}}.
\end{gather*}
We have the following commutative diagram of natural morphisms \cite[page 130]{MR3508745}:
\[
\xymatrix{W_{\iota}^{|E|}\otimes_{K\tbrats}O_{E}\ar[r]\ar[d] & W_{\iota}^{|E|}\ar[d]\\
W\ar[r] & W/H
}
\]
Taking $t$-adic completion, we have the corresponding diagram for
$V$ in place of $W$.

Let $E\colon\Spec K\to\Theta_{G}$ be a morphism corresponding to
our $G$-cover $E\to\Df_{K}$ and let $\Utg_{E,\iota}(V):=\Utg_{\Theta_{G},\iota}(V)\times_{\Theta_{G},E}\Spec K$.
\begin{lem}
\label{lem:linear-Utg}We have a $\Df$-isomorphism $\Utg_{E,\iota}(V)^{\pur}\cong V_{\iota}^{|E|}$
compatible with the morphisms to $V/H$. In particular, $\Utg_{E,\iota}(V)^{\pur}$
is smooth over $\Df$.
\end{lem}

\begin{proof}
Let $E^{\a}:=\Spec O_{E}$ be the scheme corresponding to the formal
scheme $E$. The formal scheme $\Utg_{E,\iota}(V)^{\pur}$ is the
completion of the closure of $\ulHom_{\D}^{\iota}(E^{\a},W)\times_{\D}\D^{*}$
in $\ulHom_{\D}(E^{\a},W)$ (see Remark \ref{rem:unt-generic-fib}).
We note that 
\begin{align*}
M\otimes O_{E}^{\vee} & =\Hom_{K\tbrats}(M^{\vee},K\tbrats)\otimes_{K\tbrats}O_{E}^{\vee}\\
 & =\Hom_{K\tbrats}(M^{\vee},O_{E}^{\vee})\\
 & =\Hom_{K\tbrats}(M^{\vee}\otimes_{K\tbrats}O_{E},K\tbrats)\\
 & =\Hom_{K\tbrats}(M,O_{E})^{\vee}.
\end{align*}
Here duals are taken as $K\tbrats$-modules. The $\D$-scheme $\ulHom_{\D}(E^{\a},W)$
is identified with 
\[
\Spec S^{\bullet}(M\otimes_{K\tbrats}O_{E}^{\vee})=\Spec S^{\bullet}\Hom_{K\tbrats}(M,O_{E})^{\vee}.
\]
Indeed, for a morphism $\Spec R\to D$, 
\begin{align*}
\ulHom_{\D}(E^{\a},W)(\Spec R) & =\Hom_{K\tbrats\textrm{-alg}}(S^{\bullet}M,O_{E}\otimes_{K\tbrats}R)\\
 & =\Hom_{K\tbrats\textrm{-mod}}(M,O_{E}\otimes_{K\tbrats}R)\\
 & =\Hom_{K\tbrats\textrm{-mod}}(M,\Hom_{K\tbrats}(O_{E}^{\vee},R))\\
 & =\Hom_{K\tbrats\textrm{-mod}}(M\otimes_{K\tbrats}O_{E}^{\vee},R)\\
 & =\Hom_{K\tbrats\textrm{-alg}}(S^{\bullet}(M\otimes_{K\tbrats}O_{E}^{\vee}),R)\\
 & =\left(\Spec S^{\bullet}(M\otimes_{K\tbrats}O_{E}^{\vee})\right)(\Spec R).
\end{align*}
Therefore
\[
\ulHom_{\D}(E^{\a},W)\times_{\D}\D^{*}=\Spec S^{\bullet}\Hom_{K\tpars}(M\otimes_{K\tbrats}K\tpars,O_{E}\otimes_{K\tbrats}K\tpars)^{\vee}
\]
and for $\iota\colon G\hookrightarrow H$,
\begin{gather*}
\ulHom_{\D}^{\iota}(E^{\a},W)\times_{\D}\D^{*}=\Spec S^{\bullet}\Hom_{K\tpars}^{\iota}(M\otimes_{K\tbrats}K\tpars,O_{E}\otimes_{K\tbrats}K\tpars)^{\vee}.
\end{gather*}
Since $\Hom_{K\tbrats}(M,O_{E})^{\vee}\to\Xi_{E,\iota}^{\vee}$ is
surjective, we may regard $W_{\iota}^{|E|}$ as a closed subscheme
of $\Spec S^{\bullet}(M\otimes_{K\tbrats}O_{E}^{\vee})$. This closed
subscheme is integral and coincides over $\D^{*}$ with $\ulHom_{\D}^{\iota}(E^{\a},W)\times_{\D}\D^{*}$.
Therefore $W_{\iota}^{|E|}$ is the closure of $\ulHom_{\D}^{\iota}(E^{\a},W)\times_{\D}\D^{*}$.
Since the completion of $W_{\iota}^{|E|}$ is $V_{\iota}^{|E|}$,
the lemma follows.
\end{proof}
Following \cite[Def. 3.3]{MR3431631} and \cite[Def. 5.4]{MR3508745}
(cf. \cite[Def. 6.5]{MR3791224}), we define 
\[
v(E):=\frac{1}{\sharp G}\length\frac{\Hom_{K\tbrats}(M,O_{E})}{O_{E}\cdot\Xi_{E,\iota}}.
\]

\begin{lem}
\label{lem:s-and-v}Let $\beta\colon E\to V$ be an $\iota$-equivariant
$\Df_{K}$-morphism, let $\cX:=[V/H]$ and let $\gamma\colon\cE\to\cX$
be the induced twisted arc. Then 
\[
\fs_{\cX}(\gamma)=-v(E).
\]
\end{lem}

\begin{proof}
Let $B$ be the $\QQ$-divisor on $V/H$ such that $V\to(V/H,B)$
is crepant. Then the induced boundary on $\Utg_{E,\iota}(V)^{\pur}\cong\AA_{\Df_{K}}^{d}$
is the relative anti-canonical divisor 
\[
-K_{\Utg_{E,\iota}(V)^{\pur}/(V/H,B)}:=-K_{\Utg_{E,\iota}(V)^{\pur}/\Df_{K}}+u^{*}(K_{(V/H)/\Df_{K}}+B),
\]
where $u$ is the morphism $\Utg_{E,\iota}(V)^{\pur}\to V/H$. (Note
that since $\Utg_{E,\iota}(V)^{\pur}$ is already normal, we don't
need to normalize it.) This is the vertical divisor $\Utg_{E,\iota}(V)_{0}^{\pur}\cong\AA_{K}^{d}$
with multiplicity $-v(E)$ (see \cite[Lem. 6.5]{MR3508745}). Since
$V$ is smooth over $\Df_{K}$, we have $\Omega_{V/\Df}^{d}=\omega_{V/\Df}$
and we have $\ff_{([V/H],0)}\equiv0$. Therefore, from \ref{lem:shift-j'},
\begin{align*}
\fs_{\cX}(\gamma) & =\ff_{(\Utg_{\cE}(\cX)^{\pur},A^{\utg}|_{\Utg_{\cE}(\cX)^{\pur}})}(\gamma')-\ff_{([V/H],0)}\\
 & =\ff_{(\Utg_{E,\iota}(V)^{\pur},-v(E)\cdot\Utg_{E,\iota}(V)_{0}^{\pur})}(\beta)\\
 & =-v(E).
\end{align*}
Note that the last equality holds because $\Utg_{E,\iota}(V)_{0}^{\pur}$
as well as its pullback by $\beta$ is a closed subscheme defined
by $t\in K\tbrats$.
\end{proof}
The following lemma was proved in \cite{Formal-Torsors-II} by a different
and more direct method.
\begin{lem}
\label{lem:v-const}Let us fix an embedding $\iota\colon G\hookrightarrow H$.
The map 
\[
|\Lambda_{G}|\mapsto\QQ,\,[E]\mapsto v(E)
\]
is locally constructible; for every substack $\cW\subset\Lambda_{G}$
of finite type, the restriction of this function to $|\cW|$ is constructible
(that is, every fiber is a constructible subset).
\end{lem}

\begin{proof}
Let $\Gamma_{G}$ be as in the paragraph before \ref{cor:stringy-quot}.
We further assume that each connected component of it is irreducible.
We claim that for a suitable choice of $\Gamma_{G}$, the boundary
on $\Utg_{\Gamma_{G}}(V)$ is vertical. Then the function $v$ on
$|\Lambda_{G}|=|\Gamma_{G}|$ is constant on each connected component
of $\Gamma_{G}$ and the lemma follows.

To show the claim, let $\Gamma_{G,0}$ be a connected component of
$\Gamma_{G}$. Let us write the boundary on $\Utg_{\Gamma_{G,0}}(V)$
as $A=\sum_{i=0}^{l}c_{i}A_{i}$ with $A_{i}$ prime divisors and
$c_{i}\in\QQ$, where $A_{0}$ is the vertical divisor $\Utg_{\Gamma_{G,0}}(V)$.
For a geometric generic point $\bar{\eta}\colon\Spec K\to\Gamma_{G,0}$,
the pullback of $A$ to $\Utg_{\bar{\eta}}(V)$ is the vertical divisor.
Therefore the image of $A_{0}\cap\bigcup_{i\ge1}A_{i}$ on $\Gamma_{G,0}$
is a constructible subset which does not containing the generic point.
Removing its closure, we get an open dense substack $\Upsilon\subset\Gamma_{G,0}$
such that the boundary on $\Utg_{\Upsilon}(V)$ is vertical. This
shows the above claim.
\end{proof}
From this lemma, for each $s\in\frac{1}{\sharp H}\ZZ$, there exist
a DM stack $\cC_{s}$ almost of finite type and a stabilizer-preserving
and geometrically injective map $\cC_{s}\to\Delta_{H}$ which maps
onto $v^{-1}(s)$. We define $\{v^{-1}(s)\}$ to be $\{\cC_{s}\}$
if $\cC_{s}$ is of finite type. If not, we put $\{v^{-1}(s)\}:=\infty$.
We can define
\[
\int_{\Delta_{H}}\LL^{d-v}:=\sum_{s\in\frac{1}{\sharp H}\ZZ}\{v^{-1}(s)\}\LL^{d-s}.
\]
In \cite{Formal-Torsors-II}, we study this kind of integrals more
systematically and especially proves the well-definedness of the above
integral.
\begin{cor}
\label{cor:wild-McKay-linear}Suppose that a finite group $H$ acts
on $\AA_{\Df}^{d}$ linearly. Let $X:=\AA_{\Df}^{d}/H$ and let $A$
be the boundary on $X$ such that $\AA_{\Df}^{d}\to(X,A)$ is crepant.
(Note that if the morphism $\AA_{\Df}^{d}\to X$ is etale in codimension
one, then $A=0$.) Then we have
\[
\M_{\st}(X,A)=\int_{\Delta_{H}}\LL^{d-v}\left(:=\sum_{s\in\frac{1}{\sharp H}\ZZ}\{v^{-1}(s)\}\LL^{d-s}\right).
\]
\end{cor}

\begin{proof}
As in the proof of \ref{lem:v-const}, we can choose $\Gamma_{[G]}$
so that the function $v$ on $|\Gamma_{[G]}|$ is locally constant.
Let $v_{G,i}$ be the value of $v$ on a connected component $|\Gamma_{[G],i}|$.
From \ref{cor:stringy-quot}, 
\begin{align*}
\M_{\st}(X,A) & =\M_{\st}(\cX)\\
 & =\int_{|\cJ_{\infty}\cX|}\LL^{-v}\,d\mu_{\cX}\\
 & =\sum_{G\in\GG}\sum_{i}\mu_{\cX}(\cJ_{\Gamma_{[G],i},\infty}\cX)\LL^{-v_{G,i}}\\
 & =\sum_{G\in\GG}\sum_{i}\{\Unt_{\Gamma_{[G],i}}(\cX)_{0}^{\pur}\}\LL^{-v_{G,i}}.
\end{align*}
From \ref{lem:Unt-quot} and (\ref{eq:utg-N_H}),
\[
\{\Unt_{\Gamma_{[G],i}}(\cX)_{0}^{\pur}\}=\coprod_{\iota\in\frac{\Emb(G,H)}{\Aut(G)\times H}}\{\Utg_{\Gamma_{G,i},\iota}(V)_{0}^{\pur}/\N_{H}(\iota(G))\}.
\]
Here $\Gamma_{G,i}:=\Gamma_{[G],i}\times_{\cA_{G}}\Spec k$. Since
$\Utg_{\Gamma_{G,i},\iota}(V)_{0}^{\pur}\to\Gamma_{G,i}$ is $\N_{H}(\iota(G))$-equivariant
$\AA^{d}$-bundle, we have
\[
\{\Utg_{\Gamma_{G,i},\iota}(V)_{0}^{\pur}/N_{H}(\iota(G))\}=\{\Gamma_{G,i}/\N_{H}(\iota(G))\}\LL^{d}.
\]
We conclude that
\begin{align*}
\M_{\st}(X,\cP) & =\sum_{G\in\GG}\sum_{i\in\NN}\coprod_{\iota\in\frac{\Emb(G,H)}{\Aut(G)\times H}}\{\Gamma_{G,i}/N_{H}(\iota(G))\}\LL^{d-v_{G,i}}.
\end{align*}
The map
\[
\coprod_{G\in\GG}\coprod_{i\in\NN}\coprod_{\iota\in\frac{\Emb(G,H)}{\Aut(G)\times H}}\Gamma_{G,i}/\N_{H}(\iota(G))\to\Delta_{H}
\]
is geometrically bijective. Therefore

\begin{align*}
\M_{\st}(X,\cP) & =\int_{\Delta_{H}}\LL^{d-v}.
\end{align*}
\end{proof}
\begin{rem}
\label{rem:embed-linear}If $X$ is a formal affine scheme of finite
type over $\Df$ with an action of a finite group $H$, then we can
embed it by an equivariant closed immersion into an affine space $\AA_{\Df}^{n}$
with a linear action of $H$ (see \cite[Rem. 7.1]{MR3508745}). From
\ref{lem:closed-immersion-Unt}, $\Utg_{\Gamma}([X/H])^{\pur}$ is
a closed substack of $\Utg_{\Gamma}([\AA_{\Df}^{n}/H])^{\pur}$. Therefore,
for a twisted formal disk $\cE$ over $K$, the fiber $\Utg_{\cE}([X/H])^{\pur}=\Utg_{\Gamma}([X/H])^{\pur}\times_{\Gamma,\cE}\Spec K$
is essentially the same as the untwisting variety denoted by $\bv^{|F|}$
in \cite{MR3508745}.
\end{rem}

\section{The tame case\label{sec:The-tame-case}}

In this section, we suppose that $\cX$ is a \emph{tame} formal DM
stack of finite type over $\Df$, that is, the automorphism group
$\Aut(x)$ of every geometric point $x\colon\Spec K\to\cX$ has order
prime to $p$ (we follow the convention that $0$ is prime to every
positive integer so that the above assumption holds whenever $p=0$).
We also suppose that $\cX$ is spacious and generically smooth over
$\Df$ as before. Specializing to this situation, we see how to recover
results in \cite{MR2271984} from the more general ones in the present
paper.

For a wild Galoisian group $G$, $\Utg_{\Gamma_{[G]}}(\cX)$ is empty.
Hence 
\[
\Utg_{\Gamma}(\cX)=\Utg_{\Lambda_{\tame}}(\cX)=\coprod_{l>0;\,p\nmid l}\Utg_{\Lambda_{[C_{l}]}}(\cX).
\]
Theorem \ref{thm:stringy-coarse-unt} is now rephrased as follows.
\begin{prop}
\label{prop:stringy-tame}Let $(\cX,A)$ be a log pair over $\Df$
such that $\cX$ is tame. Let $A_{l}^{\utg}$ be the induced boundary
divisor on $\Utg_{\Lambda_{[C_{l}]}}(\cX)^{\pur}$. Let $C\subset|\cX|$
be a constructible subset and $C_{l}^{\utg}\subset|\Utg_{\Lambda_{[C_{l}]}}(\cX)^{\pur}|$
the induced constructible subsets. Then we have 
\[
\M_{\st}(\cX,A)_{C}=\sum_{l>0;\,l\nmid p}\M_{\st}^{\utd}(\Utg_{\Lambda_{[C_{l}]}}(\cX)^{\pur},A_{l}^{\utg})_{C_{l}^{\utg}}.
\]
\end{prop}

\begin{lem}
\label{lem:tame-linearization}Suppose that a tame finite group $G$
acts on $k\tbrats\llbracket x_{1},\dots,x_{n}\rrbracket$ over $k\tbrats$.
Then there exist coordinates $y_{1},\dots,y_{n}\in k\tbrats\llbracket x_{1},\dots,x_{n}\rrbracket$
(that is, $k\tbrats\llbracket y_{1},\dots,y_{n}\rrbracket=k\tbrats\llbracket x_{1},\dots,x_{n}\rrbracket$)
such that the action preserves the subring $k\llbracket y_{1},\dots,y_{n}\rrbracket$
and the induced action on $k\llbracket y_{1},\dots,y_{n}\rrbracket$
is $k$-linear with respect to $y_{1},\dots,y_{n}$.
\end{lem}

\begin{proof}
As is well known, we can linearize the action over $k$; there exists
a regular system of parameters $z_{0},\dots,z_{n}$ of $k\tbrats\llbracket x_{1},\dots,x_{n}\rrbracket$
for which the $G$-action is $k$-linear. Since $G$ is tame, from
Maschke's theorem, linear $G$-representations over $k$ are semisimple.
If $\fm$ denotes the maximal ideal of $k\tbrats\llbracket x_{1},\dots,x_{n}\rrbracket$,
then the image of the map of linear representations
\[
(t)/(t^{2})\to\fm/\fm^{2}\cong\bigoplus_{i}k\cdot z_{i}
\]
is a direct summand. Let $V\subset\bigoplus_{i}k\cdot z_{i}$ be a
complementary subrepresentation of it and let $y_{1},\dots,y_{n}$
be a basis of $V$. Since $t,y_{1},\dots,y_{n}$ generates $\fm/\fm^{2}$,
we have 
\[
k\tbrats\llbracket x_{1},\dots,x_{n}\rrbracket=k\tbrats\llbracket y_{1},\dots,y_{n}\rrbracket.
\]
We easily see that $y_{1},\dots,y_{n}$ satisfy the desired property. 
\end{proof}
\begin{lem}
\label{lem:tame-smooth}If $\cX$ is smooth of pure relative dimension
$d$ over $\Df$, then for every $l$, $\Utg_{\Lambda_{[C_{l}]}}(\cX)^{\pur}$
is smooth of pure relative dimension $d$ over $\Df_{\Lambda_{[C_{l}]}}$,
unless it is empty.
\end{lem}

\begin{proof}
We may assume that $k$ is algebraically closed. A $k$-point of $\Utg_{\Lambda_{[C_{l}]}}(\cX)$
corresponds to a morphism $\cE_{\Lambda_{[C_{l}]},0}\to\cX$. Let
$x\colon\Spec k\to\cX$ be the induced $k$-point. We have the corresponding
closed embedding $\cY:=\B\Aut(x)\hookrightarrow\cX$. The completion
$\widehat{\cX}$ of $\cX$ along this closed substack is isomorphic
to $[\Spf k\tbrats\llbracket x_{1},\dots,x_{d}\rrbracket/\Aut(x)]$.
From \ref{lem:tame-linearization}, we may suppose that the $\Aut(x)$-action
on $k\tbrats\llbracket x_{1},\dots,x_{d}\rrbracket$ is the scalar
extension of a $k$-linear action on $k\llbracket x_{1},\dots,x_{d}\rrbracket$.
For $i\gg0$, the given $k$-point of $\Utg_{\Lambda_{[C_{l}]}}(\cX)$
sits in $\Utg_{\Lambda_{[C_{l}]}}(\cY_{(i)})$, where $\cY_{(i)}$
is the $i$th infinitesimal neighborhood of $\cY$. From \ref{lem:Utg-completion},
the completion of $\Utg_{\Lambda_{[C_{l}]}}(\cX)$ along $\Utg_{\Lambda_{[C_{l}]}}(\cY_{(i)})$
is isomorphic to $\Utg_{\Lambda_{[C_{l}]}}(\widehat{\cX})$. On the
other hand, $\Utg_{\Lambda_{[C_{l}]}}(\widehat{\cX})$ is also isomorphic
to the completion of 
\[
\Utg_{\Lambda_{[C_{l}]}}([\AA_{\Df}^{d}/\Aut(x)])
\]
along some closed substack. The lemma follows from \ref{lem:linear-Utg}.
\end{proof}
Recall $\Lambda_{\mu_{l}}\cong\coprod_{(\ZZ/l\ZZ)^{*}}\B\mu_{l}$
(see subsection \ref{subsec:Lambda-tame}). If $\Lambda_{\mu_{l},1}$
denotes its first component, then the morphism $\Lambda_{\mu_{l},1}\to\Lambda_{[C_{l}]}$
is an isomorphism (see (\ref{eq:Delta-mu-1})). Hence the morphism
\[
\Utg_{\Lambda_{\mu_{l},1}}(\cX)_{0}\to\Utg_{\Lambda_{[C_{l}]}}(\cX)_{0}
\]
is also an isomorphism. We define a morphism
\[
\Utg_{\Lambda_{\mu_{l},1}}(\cX)_{0}\to\I^{(\mu_{l})}\cX_{0}
\]
as follows. For an $S$-point of $\Utg_{\Lambda_{\mu_{l},1}}(\cX)$
corresponding to $\cE_{S}=[E_{S}/\mu_{l}]\to\cX$. We have the induced
closed immersion $\B\mu_{l}\times S\hookrightarrow\cE_{S}$. To the
above $S$-point of $\Utg_{\Lambda_{\mu_{l},1}}(\cX)$, we associate
the $S$-point of $\I^{(\mu_{l})}\cX$ given by the composition $\B\mu_{l}\times S\to\cE_{S}\to\cX$.

Suppose now that $\cX$ is smooth and of pure relative dimension $d$
over $\Df$. Let $\alpha\colon\Spec K\to\I^{(\mu_{l})}\cX_{0}$ be
a geometric point. The tangent space $T_{\cX_{0,K},x}$ at the induced
point $x\colon\Spec K\to\cX_{0,K}$, which is a $d$-dimensional $K$-vector
space, has an $\Aut(x)$-action. The point $\alpha$ induces an injection
$C_{l}\cong\mu_{l,K}\hookrightarrow\Aut(x)$ and a $\mu_{l,K}$-action
on $T_{\cX_{0,K},x}$.
\begin{lem}
\label{lem:Utg-I}With notation as above, the morphism $\Utg_{\Lambda_{\mu_{l},1}}(\cX)_{0}\to\I^{(\mu_{l})}\cX_{0}$
is representable and its fiber over $\alpha\in(\I^{(\mu_{l})}\cX_{0})(K)$
is isomorphic to $\AA_{K}^{c}$, where $c$ is the codimension of
the fixed locus $(T_{\cX_{0,K},x})^{\mu_{l,K}}$ in $T_{\cX_{0,K},x}$.
\end{lem}

\begin{proof}
From \ref{lem:imm-Aut}, 
\[
\Aut(\cE_{S}\to\cX_{0})\to\Aut(\B\mu_{l}\times S\to\cX_{0})
\]
is injective. Therefore the morphism of the lemma is representable.

Let $x\colon\Spec K\to\cX_{0}$ be the point induced by $\alpha$.
Suppose that the completion of $\cX_{0,K}$ at $x$ is $[\Spf K\llbracket x_{1},\dots,x_{d}\rrbracket/G]$,
where $G$ acts on $K\llbracket x_{1},\dots,x_{d}\rrbracket$ linearly.
The $\mu_{l,K}$ also acts on this power series ring through $\alpha$.
We can choose coordinates $x_{1},\dots,x_{d}$ so that $\mu_{l,K}$
acts on $\bigoplus_{i}Kx_{i}$ diagonally; for $\zeta\in\mu_{l,K}$,
$\zeta\cdot x_{i}=\zeta^{a_{i}}x_{i}$ ($0\le a_{i}<l$). A $K$-point
of the fiber over $\alpha$ corresponds to a $\mu_{l,K}$-equivariant
morphism 
\[
\gamma\colon\Spec K\llbracket t^{1/l}\rrbracket/(t)\to\Spf K\llbracket x_{1},\dots,x_{d}\rrbracket.
\]
Such a morphism is determined by the images of $x_{i}$ in $K\llbracket t^{1/l}\rrbracket/(t)$
which should be of the form
\[
\gamma^{*}(x_{i})\begin{cases}
\in K\cdot t^{a_{i}/l} & (a_{i}\ne0)\\
=0 & (a_{i}=0)
\end{cases}.
\]
Note that $\gamma^{*}$ is a morphism of local rings and $x_{i}$
should be mapped into the maximal ideal $(t^{1/l})$, thus $\gamma^{*}(x_{i})$
cannot be a nonzero scalar in $K\cdot t^{0}$. This shows that the
fiber over $\alpha$ is an affine space of dimension $\sharp\{i\mid\alpha_{i}\ne0\}$,
which is equal to the codimension of $(T_{\cX_{0,K},x})^{\mu_{l,K}}$.
\end{proof}
Note that $\Utg_{\Lambda_{\mu_{l},1}}(\cX)_{0}$ has pure dimension
$d$, while $\I^{(\mu_{l})}\cX$ has the same dimension as $(T_{\cX_{0,K},x})^{\mu_{l,K}}$
at $\alpha$. The above lemma is compatible with this fact. The codimension
of $(T_{\cX_{0,K},x})^{\mu_{l,K}}$ depends only on the connected
component of $\I^{(\mu_{l})}\cX_{0}$ on which $\alpha$ lies. The
value of the function $\fs_{\cX}$ as a point of $|\cJ_{\infty}\cX|$
also depends only on the connected component of its image on $\I^{(\mu_{l})}\cX_{0}$
through the map 
\[
\cJ_{\infty}\cX\to\cJ_{0}\cX=\Utg_{\Lambda_{\tame}}(\cX)_{0}\to\I^{(\mu_{l})}\cX_{0}.
\]
For a connected component $\cZ\subset\I^{(\mu_{l})}\cX_{0}$, we write
these values as $c(\cZ)$ and $\fs_{\cX}(\cZ)$ respectively.
\begin{cor}[{cf. \cite[Cor. 72]{MR2271984}}]
\label{cor:tame-smooth-stringy}If $\cX$ is smooth over $\Df$,
we have
\[
\M_{\st}(\cX)_{C}=\sum_{l>0;\,p\nmid l}\sum_{\cZ\subset\I^{(\mu_{l})}\cX_{0}}\{C_{\cZ}\}\LL^{c(\cZ)+\fs_{\cX}(\cZ)},
\]
where $\cZ$ runs over connected components of $\I^{(\mu_{l})}\cX_{0}$
and $C_{\cZ}$ denotes the preimage of $C$ in $|\cZ|$.
\end{cor}

\begin{proof}
From \ref{prop:stringy-tame}, 
\[
\M_{\st}(\cX)_{C}=\sum_{l\nmid p}\M_{\st}^{\utg}(\Utg_{\Lambda_{[C_{l}]}}(\cX)^{\pur},A_{l}^{\utg})_{C_{l}^{\utg}}.
\]
Locally at the level of completion, $\cX$ is the quotient stack associated
a linear action. The boundary $A_{l}^{\utg}$ is a vertical divisor;
on each connected component $\cW$ of $\Utg_{\Lambda_{[C_{l}]}}(\cX)$,
the coefficient of this divisor is $-\fs_{\cX}(\cW)$ (see the proof
of \ref{lem:s-and-v}). Therefore 
\begin{align*}
\M_{\st}(\cX) & =\sum_{l\nmid p}\sum_{\cW\subset\Utg_{\Lambda_{[C_{l}]}}(\cX)}\{\cW\}\LL^{\fs_{\cX}(\cW)}.
\end{align*}
From \ref{lem:Utg-I}, this is equal to 
\[
\sum_{l\nmid p}\sum_{\cZ\subset\I^{(\mu_{l})}(\cX)}\{\cZ\}\LL^{c(\cZ)+\fs_{\cX}(\cZ)}.
\]
\end{proof}
\begin{rem}
\label{rem:difference-from-Yas06}In the present paper, we adopt a
slightly different definition of twisted jets from the one in \cite{MR2271984}.
A twisted $n$-jet is a morphism from 
\[
\left[\left(\Spec K\llbracket t^{1/l}\rrbracket/(t^{n+1})\right)/\mu_{l}\right]
\]
in the present paper, while a morphism from 
\[
\left[\left(\Spec K\llbracket t^{1/l}\rrbracket/(t^{(nl+1)/l})\right)/\mu_{l}\right]
\]
in \cite{MR2271984}. Our stack $\I^{(\mu_{l})}\cX$ was the stack
of twisted $0$-jets of order $l$ in \cite{MR2271984}, which was
denoted by $\cJ_{0}^{l}\cX$. Lemma \ref{lem:Utg-I} shows that our
motivic measure on $|\cJ_{\infty}\cX|$ and the one given in \cite{MR2271984}
differ by the factor of $\LL^{c(\cZ)}$. The value of $\fs$ also
differs in the present paper and the paper \cite{MR2271984} by $c(\cZ)$.
\end{rem}

\begin{rem}
Corollary \ref{cor:tame-smooth-stringy} readily implies a refinement
of Reid--Shepherd-Barron--Tai criterion for terminal/canonical tame
quotient singularities, the homological McKay correspondence and Ruan's
conjecture (and its $l$-adic cohomology version) on orbifold cohomology
\cite[Cors. 5 to 8]{MR2271984}. (There is an error in \cite[Cor. 5]{MR2271984};
we need to take the minimal discrepancy of divisors over $X$ \emph{having
centers in the singular locus} rather than all divisors over $X$.
See \cite{MR3929517} for relation between this version of minimal
discrepancy and stringy invariants.)
\end{rem}

\section{DM stacks over $k$}

People would be interested more in DM stacks of finite type over $k$
rather than formal DM stacks of finite type over $\Df$. One reason
of the use of formal stacks is that it is essential both in our construction
of stacks of twisted arcs and in the proof of the change of variables
formula. However many results are formulated also for DM stacks of
finite type over $k$ and regarded as a special case of those results
for formal stacks. For the reader's convenience, we state them below.
We also give a few applications of them.

Let $\cX$ be a generically smooth DM stack of finite type over $k$
which has pure dimension $d$. This induces the formal DM stack $\cX\times\Df$
of finite type over $\Df$. We define the \emph{stack of twisted arcs}
on $\cX$, denoted by $\cJ_{\infty}\cX$, to be $\cJ_{\infty}(\cX\times\Df)$.
We define the measure $\mu_{\cX}$ and the shift function $\fs_{\cX}$
on it to be $\mu_{\cX\times\Df}$ and $\fs_{\cX\times\Df}$ respectively.
A point of $|\cJ_{\infty}\cX|$ corresponds a class of a representable
morphism $\cE\to\cX$ from a twisted formal disk $\cE$. With the
obvious translation of notions to this setting, from \ref{thm:change-vars-II},
we obtain the change of variables formula:
\begin{thm}
\label{thm:change-vars-II-1}Let $\cY,\cX$ be generically smooth
DM stacks of finite type over $k$ which has pure dimension $d$.
Let $f\colon\cY\to\cX$ be a morphism such that $\Jac_{f}$ defines
a closed substack of positive codimension. Let $A\subset|\cJ_{\infty}\cY|$
be such that $f_{\infty}|_{A}$ is almost geometrically injective.
Let $h\colon f_{\infty}(A)\to\frac{1}{r}\ZZ\cup\{\infty\}$ be a measurable
function. Then 
\[
\int_{f_{\infty}(A)}\LL^{h+\fs_{\cX}}\,d\mu_{\cX}=\int_{A}\LL^{h\circ f_{\infty}-\fj_{f}+\fs_{\cY}}\,d\mu_{\cY}\in\widehat{\cM_{k,r}'}\cup\{\infty\}.
\]
\end{thm}

We can also define a log pair $(\cX,A)$ for a DM stack $\cX$ of
finite type over $k$ which is generically smooth and normal. We define
its stringy motive $\M_{\st,Z}(\cX,A)\in\widehat{\cM'_{Z,r}}\cup\{\infty\}$,
where $Z$ is an algebraic space almost of finite type given with
a morphism $\Gamma_{\cX}\to Z$. In this setting, it is more natural
to consider (not necessarily representable) proper birational morphisms
instead of pseudo-modifications. Here a morphism $\cY\to\cX$ of DM
stacks is said to be \emph{birational }if it restricts to an isomorphism
$\cY'\to\cX'$ of open dense substacks $\cY'\subset\cX$ and $\cX'\subset\cX$.
From \ref{thm:Mst-crepant}, we obtain:
\begin{thm}
\label{thm:Mst-crepant-nonformal}Let $f\colon\cY\to\cX$ be a proper
birational morphism of DM stacks of finite type over $k$ which induces
a crepant morphism $(\cY,B)\to(\cX,A)$ of log pairs over $k$. Let
$C\subset|\cX|$ be a constructible subset and $\tilde{C}\subset|\cY|$
its preimage. Then 
\[
\M_{\st}(\cX,A)_{C}=\M_{\st}(\cY,B)_{\tilde{C}}.
\]
\end{thm}

From \ref{cor:wild-McKay-linear}, we obtain:
\begin{cor}[The wild McKay correspondence]
\label{cor:wild-McKay-linear-1}Suppose that a finite group $G$
acts on $\AA_{k}^{d}$ linearly. Let $X:=\AA_{k}^{d}/G$ and let $A$
be the boundary on $X$ such that $\AA_{k}^{d}\to(X,A)$ is crepant.
(Note that if the morphism $\AA_{k}^{d}\to X$ is etale in codimension
one, then $A=0$.) Then we have
\[
\M_{\st}(X,A)=\int_{\Delta_{G}}\LL^{d-v}.
\]
\end{cor}

We can use the wild McKay correspondence to evaluate the discrepancy
of quotient singularities, an invariant of singularities appearing
in the minimal model program. Consider the invariant 
\[
\mathrm{discrep}(\mathrm{centers}\subset X_{\sing};X)
\]
defined to be the infimum of the discrepancies of divisors over $X$
with center included in the singular locus $X_{\sing}$ (see \cite[p. 157]{MR3057950}
or \cite[Introduction]{MR3929517}). Let $o\in\Delta_{G}$ be the
point corresponding to the trivial torsor and let $\Delta_{G}\setminus\{o\}=\bigsqcup_{i}C_{i}$
be a stratification into countably many constructible subsets (of
finite dimensions) such that $v$ is constant on each stratum. We
define 
\[
\dim\int_{\Delta_{G}\setminus\{o\}}\LL^{d-v}=\sup_{i}\left\{ \dim C_{i}+d-v(C_{i})\right\} \in\QQ\cup\{\infty\}.
\]
This is independent of the choice of the stratification.
\begin{cor}
\label{cor:discrep}We keep the notation of \ref{cor:wild-McKay-linear-1}.
Suppose that the morphism $\AA_{k}^{d}\to X$ is etale in codimension
one. 
\begin{enumerate}
\item We have
\begin{align*}
 & \mathrm{discrep}(\mathrm{centers}\subset X_{\sing};X)\\
 & =d-1-\max\left\{ \dim X_{\sing},\dim\int_{\Delta_{G}\setminus\{o\}}\LL^{d-v}\right\} .
\end{align*}
\item If the integral $\int_{\Delta_{G}}\LL^{d-v}$ coverges, then $X$
is log terminal. If $X$ has a log resolution, then the converse is
also true.
\end{enumerate}
\end{cor}

\begin{proof}
Assertions (1) and (2) can be proved in the same way as proved in
\cite{MR3929517} and \cite{MR3230848} respectively for $G=\ZZ/p\ZZ$.
We only give a sketch.

(1) We can similarly define the dimension of $\M_{\st}(X)_{X_{\sing}}\in\ZZ\cup\{\infty\}$,
whether the integral diverges or not. Evaluating it by using infinitely
many modifications $Y\to X$, we get 
\[
\mathrm{discrep}(\mathrm{centers}\subset X_{\sing};X)=d-1-\dim\M_{\st}(X)_{X_{\sing}}
\]
(see \cite[Prop. 2.1]{MR3929517}). On the other hand, the version
of the wild McKay correspondence for $\M_{\st}(X)_{X_{\sing}}$ shows
\[
\M_{\st}(X)_{X_{\sing}}=\{X_{\sing}\}+\int_{\Delta_{G}\setminus\{o\}}\LL^{d-v}.
\]
Note that the term $\{X_{\sing}\}$ on the right side is the contribution
of the trivial torsor to $\M_{\st}(X)_{X_{\sing}}$. Assertion (1)
follows from these.

(2) The convergence of $\M_{\st}(X)$ implies that $X$ is log terminal.
The converse is true if $X$ has a log resolution. From the wild McKay
correspondence, the convergence of $\M_{\st}(X)$ is equivalent to
the one of $\int_{\Delta_{G}}\LL^{d-v}$. 
\end{proof}
Let us consider the case where $G=S_{n}$, the symmetric group, and
let $\ba\colon\Delta_{S_{n}}\to\ZZ$ be the Artin conductor. This
function associates to a geometric point $\Spec K\to\Delta_{S_{n}}$
the Artin conductor (see \cite[Ch. VI]{MR554237}) of the corresponding
map $G_{K\tpars}\to S_{n}\subset\mathrm{GL}_{n}(\CC)$, where $G_{K\tpars}$
denotes the absolute Galois group of $K\tpars$ and $S_{n}$ is embedded
into $\mathrm{GL}_{n}(\CC)$ by the standard permutation representation.
As an application of \ref{cor:wild-McKay-linear-1}, we obtain the
motivic version of Bhargava's mass formula \cite{MR2354798}:
\begin{cor}
\label{cor:motivic-Bhargava}We have the following equality in $\widehat{\cM_{k}'}$,
\[
\int_{\Delta_{S_{n}}}\LL^{-\ba}=\sum_{i=0}^{n-1}P(n,n-i)\LL^{-i}.
\]
Here $P(n,m)$ denotes the number of partitions of $n$ into exactly
$m$ parts.
\end{cor}

\begin{proof}
Consider the action of $S_{n}$ on $\AA_{k}^{2n}=(\AA_{k}^{2})^{n}$
by permutation. The quotient variety $X=\AA_{k}^{2n}/S_{n}$ is the
symmetric product $S^{n}\AA_{k}^{2}$ and it has a crepant resolution
by the Hilbert scheme of $n$ points on $\AA_{k}^{2}$, $\mathrm{Hilb}^{n}(\AA_{k}^{2})\to X$
\cite{0537.53056,MR2107324,MR1825408}. From a cell decomposition
of $\mathrm{Hilb}^{n}(\AA_{k}^{2})$ as in \cite{MR870732} (see also
\cite{MR2492446} for the case of a general base field), we can show
\[
\M_{\st}(X)=\{\mathrm{Hilb}^{n}(\AA_{k}^{2})\}=\sum P(n,n-i)\LL^{2n-i}.
\]
On the other hand, from \cite[Th. 4.8]{MR3431631}, the function $\ba$
is equal to the $v$ function associated to the action $S_{n}\curvearrowright\AA_{k}^{2n}$.
The corollary now follows from \ref{cor:wild-McKay-linear-1}. 
\end{proof}

\appendix

\section{General results on stacks\label{sec:General-stacks}}

In Appendix, we collect general results on (mainly DM) stacks which
are necessary in this paper. Some of them would be well-known to specialists,
but the author could not find in the literature.

We first recall that a morphism $f\colon\cY\to\cX$ of DM stacks is
representable if and only if for every geometric point $y\colon\Spec K\to\cY$,
the group homomorphism $\Aut_{\cY}(y)\to\Aut_{\cX}(f(y))$ is injective
\cite[Lem. 4.4.3]{MR1862797}. We say that $f$ is \emph{stabilizer-preserving
}if for every $y$, the map $\Aut_{\cY}(y)\to\Aut_{\cX}(f(y))$ is
an isomorphism. If $\cY$ lies over another stack $\cS$, then we
can consider the kernel $\Aut_{\cY/\cS}(y)\subset\Aut_{\cY}(y)$ of
$\Aut_{\cY}(y)\to\Aut_{\cS}(\bar{y})$ with $\bar{y}\in\cS$ the image
of $y$.
\begin{lem}
\label{lem:Aut-inj}Suppose that $f\colon\cY\to\cX$ is an $\cS$-morphism
and let $y$ be a geometric point of $\cY$. Then $\Aut_{\cY}(y)\to\Aut_{\cX}(f(y))$
is injective if and only if $\Aut_{\cY/\cS}(y)\to\Aut_{\cX/\cS}(f(y))$
is injective.
\end{lem}

\begin{proof}
The ``only if'' is clear. We prove the ``if'' part. For $\alpha\in\Aut_{\cS}(\bar{y})$,
if $\Aut_{\cY}(y)_{\alpha}$ and $\Aut_{\cX}(f(y))_{\alpha}$ denote
the preimages of $\alpha$ in $\Aut_{\cY}(y)$ and $\Aut_{\cX}(f(y))$
respectively and if these preimages are note empty, then each map
$\Aut_{\cY}(y)_{\alpha}\to\Aut_{\cX}(f(y))_{\alpha}$ is isomorphic
to the map
\[
\Aut_{\cY}(y)_{0}=\Aut_{\cY/\cS}(y)\to\Aut_{\cX}(f(y))_{0}=\Aut_{\cX/\cS}(f(y))
\]
as a set map. This shows the ``if'' part.
\end{proof}
\begin{lem}
\label{lem:diagonal-property}Let $g\colon\cZ\to\cY$ and $f\colon\cY\to\cX$
be morphisms of stacks. Let $\mathbf{P}$ be a property of morphisms
of stacks which is stable under the base change and composition. If
the diagonals $\Delta_{g}$ and $\Delta_{f}$ have property $\mathbf{P}$,
then so does $\Delta_{f\circ g}$.
\end{lem}

\begin{proof}
Consider the 2-commutative diagram:
\[
\xymatrix{\cZ\ar[d]^{\Delta_{g}}\ar@/{}_{3pc}/[dd]_{\Delta_{f\circ g}}\\
\cZ\times_{\cY}\cZ\ar[r]\ar[d]^{\phi} & \cY\ar[d]^{\Delta_{f}}\\
\cZ\times_{\cX}\cZ\ar[r] & \cY\times_{\cX}\cY
}
\]
The square in the diagram is 2-Cartesian. Therefore $\phi$ has property
$\mathbf{P}$. It follows that $\Delta_{f\circ g}\cong\phi\circ\Delta_{g}$
has property $\mathbf{P}$.
\end{proof}
\begin{lem}
\label{lem:DM-base-ch}Let 
\[
\xymatrix{\cX_{T}\ar[r]\ar[d] & \cX\ar[d]\\
T\ar[r] & \cS
}
\]
be a 2-Cartesian diagram of stacks. Suppose that $\cS$ and $\cX_{T}$
are DM stacks and $T\to\cS$ is an atlas. Then $\cX$ is a DM stack.
\end{lem}

\begin{proof}
From the assumption, the diagonal $\Delta_{\cX_{T}/T}$ is representable,
unramified and finite morphism. This is the base change of $\Delta_{\cX/\cS}$
along $T\to\cS$. Therefore $\Delta_{\cX/\cS}$ is also representable,
unramified and finite morphism. From \ref{lem:diagonal-property},
$\Delta_{\cX}=\Delta_{\cX/\Spec k}$ is also representable, unramified
and finite. If $U\to\cX_{T}$ is an atlas, then so is $U\to\cX_{T}\to\cX$.
We have proved the lemma.
\end{proof}

\begin{lem}
\label{lem:rep-stab-preserving}Let $h\colon\cZ\to\cY$ and $f\colon\cY\to\cX$
be morphisms of DM stacks. Suppose that $f$ is representable (resp.\ stabilizer-preserving)
and unramified. Then $\Aut(h)\to\Aut(f\circ h)$ is injective (resp.\ bijective).
\end{lem}

\begin{proof}
Consider the morphisms
\[
\I\cY\xrightarrow{b}\I\cX\times_{\cX}\cY\xrightarrow{a}\cY.
\]
Since $a\circ b$ is representable, so is $b$. Since $a\circ b$
is finite and $a$ is separated, $b$ is finite. Since $a\circ b$
and $a$ are unramified, $b$ is unramified. Thus $b$ is representable,
finite and unramified. When $f$ is representable, then $b$ is also
universally injective. From \cite[tag 04XV]{stacks-project}, $b$
is a closed immersion. When $f$ is stabilizer-preserving, from \cite[tag 0DU9]{stacks-project},
$b$ is an isomorphism. Therefore, for each object $u\in\cZ$ say
over $U\in\Aff$, we have an injection (resp.\ an isomorphism) $\ulAut(h(u))\to\ulAut(f\circ h(u))$
of group schemes over $U$ and an injection (resp.\ an isomorphism)
$\Aut(h(u))\to\Aut(f\circ h(u))$ of groups. An automorphism $\Aut(h)$
is by definition a collection of automorphisms $\alpha_{u}\in\Aut(h(u))$,
$u\in\cZ$ with a suitable compatibility. Therefore the lemma follows.
\end{proof}
\begin{lem}
\label{lem:imm-Aut}Let $h\colon\cZ\to\cY$ and $f\colon\cY\to\cX$
be morphisms of DM stacks. Suppose that $h$ is a thickening. Then
the map $\Aut(f)\to\Aut(f\circ h)$ is injective.
\end{lem}

\begin{proof}
Let $s\colon S\to\cY$ be a morphism from an affine scheme and let
\[
t\colon T:=S\times_{\cY}\cZ\to\cZ
\]
be the projection. Automorphisms $\alpha\in\Aut(f\circ s)$ and $\beta\in\Aut(f\circ h\circ t)$
with $\alpha\mapsto\beta$ induce the following commutative diagram:
\[
\xymatrix{\ulAut_{T}(f\circ h\circ t)\ar[r]\ar[d] & \ulAut_{S}(f\circ s)\ar[d]\\
T\ar@{^{(}->}[r]\ar@<-1ex>[u]_{\beta_{s}} & S\ar@<-1ex>[u]_{\alpha_{s}}
}
\]
Since $\cX$ is a DM stack, the morphism $\ulAut_{S}(f\circ s)\to S$
is unramified. Since $T\to S$ is a thickening, $\alpha_{s}$ is determined
by $\beta_{s}$. Since we have this for each object $s\in\cY(S)$,
we conclude the lemma.
\end{proof}
\begin{lem}
\label{lem:iso-characterization}Let $f\colon\cY\to\cX$ be a stabilizer-preserving
etale morphism of DM stacks such that for every algebraically closed
field $K$, the map $\cY[K]\to\cX[K]$ is bijective. Then $f$ is
an isomorphism.
\end{lem}

\begin{proof}
Let $V\to\cX$ be an atlas. It suffices to show that $V\times_{\cX}\cY\to V$
is an isomorphism. This is an etale surjective morphism of algebraic
spaces. It suffices to show that $(V\times_{\cX}\cY)(K)\to V(K)$
is injective. In turn, it suffices to show that for every geometric
point $x\colon\Spec K\to\cX$, we have $(\Spec K\times_{\cX}\cY)_{\red}\cong\Spec K$.
Let $G$ be the automorphism group of $x$. Then we have the induced
closed immersion $\B G=[\Spec K/G]\hookrightarrow\cX_{K}$. Since
$\cY\to\cX$ is stabilizer-preserving, we have 
\[
(\B G\times_{\cX_{K}}\cY_{K})_{\red}\cong(\B G\times_{\cX}\cY)_{\red}\cong\B G.
\]
Therefore 
\[
(\Spec K\times_{\cX}\cY)_{\red}\cong(\Spec K\times_{\B G}\B G\times_{\cX}\cY)_{\red}\cong\Spec K.
\]
\end{proof}
\begin{lem}
\label{lem:etale}Let $f\colon\cY\to\cX$ be a representable morphism
of DM stacks which is formally etale (in the sense that for every
morphism $V\to\cX$ from algebraic space, $\cY\times_{\cX}V\to V$
is formally etale). Let $\iota\colon\cD\hookrightarrow\cE$ be a thickening
of DM stacks. Suppose that we have the following 2-commutative diagram.
\[
\xymatrix{\cD\ar[r]\ar[d]_{\iota} & \cY\ar[d]^{f}\\
\cE\ar@{=>}[ur]_{\eta}\ar[r] & \cX
}
\]
Then there exist a dashed arrow $\gamma$ (1-morphism) and two thick
arrows $\alpha,\beta$ (2-morphisms) forming a 2-commutative diagram
\[
\xymatrix{\cD\ar[rrrr]\ar[d]_{\iota} &  & {} &  & \cY\ar[d]^{f}\\
\cE\ar@{-->}[urrrr]\sb(0.4){\gamma}\ar[rrrr]\ar@{=>}[urr]\sp(0.5){\alpha} &  & {}\ar@{=>}[urr]_{\beta} &  & \cX
}
\]
which induces $\eta$. Moreover, for two such tuples $(\gamma,\alpha,\beta)$
and $(\gamma',\alpha',\beta')$, there exists a unique 2-isomorphism
$\gamma\Rightarrow\gamma'$ through which $\alpha$ and $\alpha'$
as well as $\beta$ and $\beta'$ correspond to each other.
\end{lem}

\begin{proof}
Consider the 2-commutative diagram:
\[
\xymatrix{\cD\ar[dr]\ar[drr]\ar[ddr]\\
 & \cE\times_{\cX}\cY\ar[r]\ar[d] & \cY\ar[d]^{f}\\
 & \cE\ar[r] & \cX
}
\]
The two triangles and the square in this diagram are given witnessing
2-isomorphisms which induces $\eta$. The image of $|\cD|\to|\cE\times_{\cX}\cY|$
is an open and closed subset. Let $\cE'$ be the corresponding open
and closed substack. From \ref{lem:iso-characterization}, $\cE'\to\cE$
is an isomorphism, which we denote by $a$. We choose the quasi-inverse
$b\colon\cE\to\cE'$ of the natural isomorphism $a\colon\cE'\to\cE$
and witnessing 2-isomorphisms $\alpha$ and $\beta$ of the two right
triangles in the following diagram which induces the identity of $a$:
\[
\xymatrix{\cD\ar@/_{1pc}/[dddr]\ar[dr]\ar@/^{1pc}/[drrrrrr]\\
 & \cE'\ar[rrrr]^{\id}\ar[dd]_{a} &  & {} &  & \cE'\ar[r]\ar[dd]_{a} & \cY\ar[dd]\\
\\
 & \cE\ar[rrrr]_{\id}\ar@{-->}[uurrrr]_{b}\ar@{=>}[uurr]^{\alpha} &  & {}\ar@{=>}[uurr]_{\beta} &  & \cE\ar[r] & \cX
}
\]
Indeed such a choice exists. First we choose any isomorphism $\alpha\colon\id_{\cE'}\to b\circ a$,
which induces $a(\alpha)\colon a\to a\circ b\circ a$ and an isomorphism
$\Iso(a\circ b\circ a,a)\to\Aut(a)$. We also have an isomorphism
$\Aut(a\circ b,\id_{\cE})\to\Iso(a\circ b\circ a,a)$ induced by $a$.
We choose $\beta\in\Aut(a\circ b,\id_{\cE})$ to be the one corresponding
to $\id_{a}$ through the two isomorphisms. The first assertion of
the lemma follows. For another quasi-inverse $b'\colon\cE\to\cE'$
and an isomorphism $\alpha'\colon\id_{\cE'}\to b'\circ a$, we obtain
a unique isomorphism $\xi\colon b\to b'$ compatible with $\alpha$
and $\alpha'$; for any $e'\in\cE'$, $\xi\colon b(a(e'))\to b'(a(e'))$
is given by $\alpha'\circ\alpha^{-1}$.
\end{proof}
\begin{lem}
\label{lem:etale-invariance}Let $\cY\to\cX$ be a thickening of DM
stacks and let $U\to\cY$ be an etale morphism from an algebraic space
(resp.\ scheme, affine scheme). There exists a thickening $U\hookrightarrow\tilde{U}$
of algebraic spaces (resp.\ schemes, affine schemes) and an etale
morphism $\tilde{U}\to\cX$ such that the diagram
\[
\xymatrix{U\ar[r]\ar[d] & \tilde{U}\ar[d]\\
\cY\ar[r] & \cX
}
\]
is 2-Cartesian. For another pair $(U\hookrightarrow\tilde{U}',\tilde{U}'\to\cX)$
of such morphisms, there exists a unique isomorphism $\tilde{U}\xrightarrow{\sim}\tilde{U}'$
compatible with these morphisms.
\end{lem}

\begin{proof}
We take an etale groupoid in affine schemes, $W\rightrightarrows V$,
inducing $\cX$. Let $W_{\cY}\rightrightarrows V_{\cY}$ and $W_{U}\rightrightarrows V_{U}$
be its base changes by $\cY\to\cX$ and $U\to\cX$ respectively. These
are again etale groupoids in affine schemes. The morphisms $W_{U}\to W_{\cY}$
and $V_{U}\to V_{\cY}$ are etale and the morphisms $W_{\cY}\to W$
and $V_{\cY}\to V$ are thickenings. From the invariance of small
etale site by a thickening \cite[tag 05ZH]{stacks-project}, they
fit into the unique Cartesian diagrams: 
\[
\xymatrix{V_{U}\ar[r]\ar[d] & V_{\tilde{U}}\ar[d]\\
V_{\cY}\ar[r] & V
}
\quad\xymatrix{W_{U}\ar[r]\ar[d] & W_{\tilde{U}}\ar[d]\\
W_{\cY}\ar[r] & W
}
\]
Here horizontal arrows are thickenings and vertical arrows are etale.
We obtain an etale groupoid $W_{\tilde{U}}\rightrightarrows V_{\tilde{U}}$,
which induces a DM stack $\tilde{U}$. It fits into the Cartesian
diagram
\[
\xymatrix{U\ar[r]\ar[d] & \tilde{U}\ar[d]\\
\cY\ar[r] & \cX
}
\]
with horizontal arrows being thickenings and vertical arrows being
etale. It follows that $\tilde{U}$ is an algebraic space. From \cite[tags 06AD and  05ZR]{stacks-project},
if $U$ is a scheme (resp.\ an affine scheme), then so is $\tilde{U}$.
We have proved the first assertion.

For the second assertion, let $\tilde{U}''\subset\tilde{U}\times_{\cX}\tilde{U}'$
be the image of $U$, which is an open and closed subscheme. We get
the third pair $(U\hookrightarrow\tilde{U}'',\tilde{U}''\to\cX)$
with the same property. The projections $\tilde{U}''\to\tilde{U}$
and $\tilde{U}''\to\tilde{U}'$ are isomorphisms, giving the desired
isomorphism $\tilde{U}\to\tilde{U}'$. 
\end{proof}
\bibliographystyle{alpha}
\bibliography{../../mybib}

\end{document}